\documentclass[10pt,a4paper]{article}

\usepackage[margin=1.00in]{geometry}

\usepackage[utf8]{inputenc}
\usepackage{hyperref}
\usepackage{mathrsfs}
\usepackage{pdfpages}
\usepackage[utf8]{inputenc}
\usepackage{amsmath}
\usepackage{amsfonts}
\usepackage{amssymb}
\usepackage{dsfont}
\usepackage{xifthen}
\usepackage{adjmulticol}
\usepackage{graphicx}
\usepackage{multicol}
\usepackage{pdfpages}
\usepackage{dsfont}
\usepackage{bm}

\SetSymbolFont{largesymbols}{bold}{OMX}{txex}{b}{n}

\usepackage{stmaryrd} 
\usepackage{amsthm}
\usepackage{mathtools}

\DeclareMathAlphabet{\mathcalligra}{T1}{calligra}{m}{n}
 \newcommand{\pred}{\scr{P}\text{r}}

\newcommand{\dd}{{\mathrm{d}}}

\usepackage{comment}
\excludecomment{toexclude}

\newcommand{\comentario}[1]{}

\usepackage{mathtools}
\newcommand{\equalexpl}[2]{%
  \underset{\substack{\uparrow\\\mathrlap{\text{\hspace{-1em}#1}}}}{#2}}

\newcommand{\MR}[1]{}

\newcommand{\mbm}[1]{\textbf{\textsc{#1}}}

\newcommand{\esp}[2]{%
\ifthenelse{\equal{#2}{}}{\mathds{E}\left[ #1 \right] }{\mathds{E}^{#2}\left[ #1 \right] }%
}

\newcommand{\esptilde}[2]{%
\ifthenelse{\equal{#2}{}}{\tilde{\mathds{E}}\left[ #1 \right] }{\tilde{\mathds{E}}^{#2}\left[ #1 \right] }%
}
\newcommand{\espC}[2]{%
\ifthenelse{\equal{#2}{}}{{{E}}\left[ #1 \right] }{\tilde{\mathds{E}}^{#2}\left[ #1 \right] }%
}

\newcommand{\test}[1][]{%
\ifthenelse{\equal{#1}{}}{omitted}{given}%
}

\newcommand{\proba}[2]{%
\ifthenelse{\equal{#2}{}}{\mathds{P} \left( #1 \right) }{\mathds{P}^{#2}\left( #1 \right) }%
}
\newcommand{\probaC}[2]{%
\ifthenelse{\equal{#2}{}}{{P} \left( #1 \right) }{\mathds{P}^{#2}\left( #1 \right) }%
}

\newcommand{\scr}[1]{\mathscr{#1}}
\newcommand{\indic}[1]{\mathds{1}_{#1}}
\newcommand{\ptilde}[2]{%
\ifthenelse{\equal{#2}{}}{\tilde{\mathds{P}} \left( #1 \right) }{\mathds{P}^{#2}\left( #1 \right) }%
}
\newcommand\restr[2]{{
  \left.\kern-\nulldelimiterspace 
  #1 
  \vphantom{\big|} 
  \right|_{#2} 
  }}

\usepackage{framed}

\usepackage{scalerel,stackengine}
\stackMath
\newcommand\reallywidehat[1]{%
\savestack{\tmpbox}{\stretchto{%
  \scaleto{%
    \scalerel*[\widthof{\ensuremath{#1}}]{\kern.1pt\mathchar"0362\kern.1pt}%
    {\rule{0ex}{\textheight}}
  }{\textheight}%
}{2.4ex}}%
\stackon[-6.9pt]{#1}{\tmpbox}%
}
\renewcommand{\epsilon}{\varepsilon} 

\newcommand{\ols}[1]{\mskip.5\thinmuskip\overline{\mskip-.5\thinmuskip {#1} \mskip-.5\thinmuskip}\mskip.5\thinmuskip} 
\newcommand{\olsi}[1]{\,\overline{\!{#1}}} 
\makeatletter
\newcommand\closure[1]{
  \tctestifnum{\count@stringtoks{#1}>1} 
  {\ols{#1}} 
  {\olsi{#1}} 
}

\long\def\count@stringtoks#1{\tc@earg\count@toks{\string#1}}
\long\def\count@toks#1{\the\numexpr-1\count@@toks#1.\tc@endcnt}
\long\def\count@@toks#1#2\tc@endcnt{+1\tc@ifempty{#2}{\relax}{\count@@toks#2\tc@endcnt}}
\def\tc@ifempty#1{\tc@testxifx{\expandafter\relax\detokenize{#1}\relax}}
\long\def\tc@earg#1#2{\expandafter#1\expandafter{#2}}
\long\def\tctestifnum#1{\tctestifcon{\ifnum#1\relax}}
\long\def\tctestifcon#1{#1\expandafter\tc@exfirst\else\expandafter\tc@exsecond\fi}
\long\def\tc@testxifx{\tc@earg\tctestifx}
\long\def\tctestifx#1{\tctestifcon{\ifx#1}}
\long\def\tc@exfirst#1#2{#1}
\long\def\tc@exsecond#1#2{#2}
\makeatother

\newtheorem{theorem}{Theorem}

\newtheorem{remark}[theorem]{Remark}
\newtheorem{proposition}[theorem]{Proposition}
\newtheorem{lemma}[theorem]{Lemma}
\newtheorem{corollary}[theorem]{Corollary}

\newtheorem{definition}[theorem]{Definition}

\title{ Noise Reinforced Lévy Processes: \\ Lévy-Itô Decomposition and Applications}
\author{Alejandro Rosales-Ortiz\footnote{alejandro.rosalesortiz@math.uzh.ch}}
\date{Institute of Mathematics, University of Zürich \footnote{Research supported by the Swiss National Science Foundation (SNSF).}}

\usepackage{tcolorbox} 
\newcommand\contentsSpacingBefore{-10 pt}
\newcommand\contentsSpacing{0}  

\numberwithin{equation}{section}
\numberwithin{theorem}{section}

\usepackage{caption}
\begin{document}
\maketitle

\begin{abstract}
    A step reinforced random walk is a discrete time process with memory such that at each time step, with fixed probability $p \in (0,1)$, it repeats a previously performed step chosen uniformly at random while with complementary probability $1-p$, it performs an independent step with fixed law. In the continuum, the main result of Bertoin in \cite{BertoinNRLP} states that the random walk constructed from the   discrete-time skeleton of a  Lévy process for a time partition of mesh-size $1/n$ converges, as $n \uparrow \infty$ in the sense of finite dimensional distributions, to a process $\hat{\xi}$ referred to as a noise reinforced Lévy process. Our first main result states that a noise reinforced Lévy processes has rcll paths and satisfies a $\textit{noise reinforced}$ Lévy Itô decomposition in terms of the $\textit{noise reinforced}$ Poisson point process of its jumps. We introduce the joint distribution of a Lévy process and its reinforced version $(\xi, \hat{\xi})$ and show that the pair, conformed by the  skeleton of the Lévy process and its step reinforced version, converge towards $(\xi, \hat{\xi})$ as the mesh size tend to $0$.  As an application, we analyse the rate of growth of $\hat{\xi}$ at the origin and identify its main features  as an infinitely divisible process. 
\end{abstract}

\section{Introduction}
\addtocontents{toc}{\vspace{\contentsSpacing pt}}

The Lévy-Itô decomposition is one of the main tools for the study of Lévy processes. In short,  any real Lévy process $\xi$ has rcll sample paths and its jump process induces a Poisson random measure -- called the jump measure $\mathcal{N}$ of $\xi$ -- whose intensity is described by its Lévy measure $\Lambda$. 
Moreover, it states that $\xi$ can be written as the sum of tree process 
\begin{equation*}
    \xi_t = \xi^{(1)}_t +  \xi^{(2)}_t + \xi^{(3)}_t , \quad \quad t \geq 0, 
\end{equation*}
of radically different nature. More precisely, the continuous  part of $\xi$ is given by   $\xi^{(1)}= (at + qB_t: t \geq 0)$ for a Brownian motion $B$ and reals $a,q$, while 
 $\xi^{(2)}$ is a compound Poisson process  with jump-sizes greater than 1 and   $\xi^{(3)}$ is a purely discontinuous martingale with jump-sizes smaller than 1. Moreover, the processes $\xi^{(2)}$, $\xi^{(3)}$ can be reconstructed from the jump measure $\mathcal{N}$. It is well known that  $\mathcal{N}$ is characterised by the two following  properties: 
for any Borel $A$ with $\Lambda(A) < \infty$, the counting process of jumps $\Delta {\xi}_s \in A$ that we denote by  ${N}_A$ is a Poisson process with rate  $\Lambda(A)$, and  for any  disjoint Borel sets $A_1, \dots , A_k$  with $\Lambda(A_i) < \infty$, the corresponding Poisson processes ${N}_{A_1}, \dots, {N}_{A_k}$ are independent.  We refer to e.g. \cite{Bertoin_LevyProcessesBook, KyprianouBook, SatoBook} for a complete account on the theory of Lévy processes.

\par  In this work, we shall give an analogous description for \textit{noise reinforced Lévy processes} (abbreviated NRLPs). This  family of processes has been recently introduced by Bertoin in \cite{BertoinNRLP} and correspond to weak limits of   step reinforced random walks of skeletons of Lévy process. In order to be more precise,  let us briefly  recall  the connection between these discrete objects and our continuous time setting.  Fix a Lévy process $\xi$  and denote, for each fixed $n$, by $X^{(n)}_k := \xi_{k/n}-\xi_{(k-1)/n}$ the $k$-th increment of ${\xi}$ for a partition of size $1/n$ of the real line. The process $S^{(n)}_k := X^{(n)}_1+ \dots + X^{(n)}_k = \xi_{k/n}$ for $k \geq 1$ is a random walk, also called the $n$-skeleton of $\xi$. 
 Now, fix a real number $p \in (0,1)$ that we call  the reinforcement  or memory   parameter and let  $\hat{S}^{(n)}_1 := X^{(n)}_1$. Then,  define recursively $\hat{S}^{(n)}_k$ for $k \geq 2$  according to the following rule:  for each $k \geq 2$,  set $\hat{S}^{(n)}_k := \hat{S}^{(n)}_{k-1} + \hat{X}^{(n)}_k$ where, with probability $1-p$, the step $\hat{X}^{(n)}_k$ is the increment $X^{(n)}_k$ with law $\xi_{1/n}$ -- and hence independent from the previously performed steps -- while with probability $p$, $\hat{X}^{(n)}_k$ is an increment chosen uniformly at random from the previous ones $\hat{X}^{(n)}_1, \dots , \hat{X}^{(n)}_k$. When the former occurs, the step is called an innovation, while in the latter case it is referred to as a reinforcement. The process  $(\hat{S}^{(n)}_k)$ is called the step-reinforced version of $(S^{(n)}_k)$. It was shown in \cite{BertoinNRLP} that, under appropriate assumptions on the memory parameter $p$, we have the following convergence  \textit{in the sense of finite dimensional distributions} as the mesh-size tends to $0$ 
 \begin{equation} \label{equation:limitNRLP}
     (\hat{S}^{(n)}_{\lfloor nt \rfloor})_{t \in [0,1]} \overset{f.d.d.}{\longrightarrow} 
     \big(\hat{\xi}_t \big)_{t \in [0,1]}, 
 \end{equation}
  towards a process $\hat{\xi}$ identified in \cite{BertoinNRLP} and called a noise reinforced Lévy process. It should be noted that the process $\hat{\xi}$ constructed in \cite{BertoinNRLP} is a  priori not even rcll, and this will be one of our first concerns.
  \par We are now in position to briefly state the main results of  this work. First,    we shall prove the existence of a rcll modification for $\hat{\xi}$.  In particular, this allow  us to consider the jump process $(\Delta \hat{\xi}_s)$; a proper understanding of its  nature  will be crucial for this work. In this direction, we introduce a new family of random measures in $\mathbb{R}^+ \times \mathbb{R}$ of independent interest under the name \textit{noise reinforced} Poisson point processes (abbreviated NRPPPs) and we study its basic properties. This lead us towards our first main result, which is a version of the Lévy-Itô decomposition in the reinforced setting. More precisely, we show that the jump measure of  $\hat{\xi}$ is a NRPPP and that $\hat{\xi}$ can be written as
  \begin{equation*}
      \hat{\xi}_t = \hat{\xi}^{(1)}_t + \hat{\xi}^{(2)}_t  + \hat{\xi}^{(3)}_t, \quad t \geq 0, 
  \end{equation*}
  where now,  $\hat{\xi}^{(1)} = (at + q \hat{B}_t: t \geq 0)$ for a  continuous Gaussian process $\hat{B}$, the process  $\hat{\xi}^{(2)}$ is a reinforced compound Poisson process with jump-sizes greater than one, while $\hat{\xi}^{(3)}$ is a purely discontinuous semimartingale. The continuous Gaussian process $\hat{B}$ is the so-called noise reinforced Brownian motion, a Gaussian process introduced in \cite{BertoinUniversality} with law singular with respect to $B$, and arising as the universal scaling limit of noise reinforced random walks when the law of the typical step is in $L_2(\mathbb{P})$  --  and hence plays the role of  Brownian motion in the reinforced setting, see also \cite{InvariancePrinciplesNRRW} for related results. Needless to say that if the starting Lévy process $\xi$ is a Brownian motion, the limit $\hat{\xi}$ obtained in (\ref{equation:limitNRLP}) is a noise reinforced Brownian motion. As in the non-reinforced case, $\hat{\xi}^{(2)}$ and $\hat{\xi}^{(3)}$ can be recovered from the jump measure $\hat{\mathcal{N}}$, but in contrast, they are not Markovian. The terminology used for the  jump measure of $\hat{\xi}$ is  justified by the following remarkable property: for any Borel $A$ with $\Lambda(A) < \infty$, the counting process of jumps $\Delta \hat{\xi}_s \in A$ that we denote by  $\hat{N}_A$ is a reinforced Poisson process and, more precisely, it has the law of the noise reinforced version of $N_A$ (hence,  the terminology $\hat{N}_A$ is consistent). Moreover, for any  disjoint Borel sets $A_1, \dots , A_k$  with $\Lambda(A_i) < \infty$, the corresponding  $\hat{N}_{A_1}, \dots, \hat{N}_{A_k}$ are independent noise reinforced Poisson processes. Informally, the reinforcement induces memory on the jumps of $\hat{\xi}$, and these are repeated at the jump times of an independent counting process. When working on the unit interval, this counting process is the so-called  \textit{Yule-Simon} process. 
  \par The second main  result of this work consists in defining pathwise, the  noise reinforced version $\hat{\xi}$ of the  Lévy process $\xi$. We always denote such a pair by $(\xi, \hat{\xi})$.  This is mainly achieved by transforming the jump measure of $\xi$ into a NRPPP, by a procedure that can be interpreted as the continuous time analogue of the reinforcement algorithm we described for random walks. More precisely, the steps $X_k^{(n)}$ of the $n$-skeleton are replaced by the jumps $\Delta \xi_s$ of the Lévy process;  each jump of $\xi$ is  shared with its reinforced version $\hat{\xi}$ with probability $1-p$, while with  probability $p$ it is  discarded and remains  independent of $\hat{\xi}$. We then proceed to justify our construction  by showing that   the  skeleton of $\xi$ and its reinforced version  $(S^{(n)}_{\lfloor n \cdot \rfloor}, \hat{S}^{(n)}_{\lfloor n \cdot \rfloor})$  converge weakly towards $(\xi, \hat{\xi})$,  strengthening \eqref{equation:limitNRLP} considerably. 
  \par Section \ref{section:applications} is devoted to applications: on the one hand, in Section \ref{section:ratesOfGrowth} we study the  rates of growth at the origin of $\hat{\xi}$ and prove that well know results established by Blumenthal and Getoor in \cite{BlumenthalGetoor_SampleFunctions} for Lévy processes still hold for NRLPs. On the other hand, in Section \ref{section:IDprocesses} we analyse NRLPs under the scope of infinitely divisible processes in the sense of \cite{RosinskiID}.  We shall give a proper description of $\hat{\xi}$ in terms of the usual terminology of infinitely divisible processes, as well as an application, by making use of the so-called  Isomorphism theorem for infinitely divisible processes. 
  \par Let us mention that in the discrete setting,  reinforcement of processes and models has been  subject of active research for a  long time, see for instance the survey by Pemantle \cite{PemantleSurvey} as well as  e.g.  \cite{BertoinCounterbalancing, Bertenghi, BaurBertoin, MaillerBravo, Bercu, GUTzerosELR}  and references therein for related work. However, reinforcement of time-continuous stochastic processes, which is the topic of this work,  remains a rather unexplored subject. 
  \smallskip \\
\noindent \textit{The rest of the work is organised as follows:} in Section \ref{section:preliminaries} we recall the basic building blocs needed for the construction of NRLPs and recall the main results that will be needed. Notably, we give a brief overview of the features of the Yule-Simon process  and  present some important examples of NRLPs. In Section \ref{section:extensionAndRegularity} we show  that a NRLP has a rcll modification.  In Section \ref{section:NRPPP}  we construct NRPPPs, study their main properties of interest,  and in Section \ref{subection:jump_LevyIto-Part1} we prove  that the jump measure of a NRLP is a NRPPP -- a result that we refer to as the "reinforced Lévy-Itô decomposition".  In Section \ref{section:jointConvergence} we  show that the pair conformed by the $n$-skeleton of a Lévy process and its reinforced version converge in distribution, as the mesh size tends to 0, towards $(\xi, \hat{\xi})$. To achieve this, first we start by proving in Section \ref{subsection:jump_LevyItoPart2}   that a  NRLP can be reconstructed from its jump measure -- a result that we refer to as the "reinforced Lévy Itô synthesis". Making use of this result in Section  \ref{section:jointlaw} we define the joint law $(\xi, \hat{\xi})$ and in  Section \ref{subsection:PruebajointConvergence} we establish our convergence result. Finally,  Section \ref{section:applications} is devoted to applications. Particular attention is given through this work at comparing, when possible and pertinent, our results for NRLPs to the classical ones for Lévy processes. \smallskip  \\

{\noindent
\begin{tcolorbox}[
    boxrule=0.7pt,
    standard jigsaw,
    opacityback=0
]

  \tableofcontents
\end{tcolorbox}
}

\addtocontents{toc}{\vspace{\contentsSpacingBefore}}
\section{Preliminaries}\label{section:preliminaries}
\addtocontents{toc}{\vspace{\contentsSpacing pt}} 

\subsection{Yule-Simon processes}

In this section, we recall several results from \cite{BertoinNRLP} concerning Yule-Simon processes needed for defining NRLPs. These  results will be used frequently in this work and are re-stated for ease of reading. \par 

A Yule-Simon process on the interval $[0,1]$ is a counting process, started from $0$, with first jump time uniformly distributed in $[0,1]$, and behaving afterwards as a (deterministically) time-changed standard Yule process. More precisely, for fixed $p \in (0,1)$, if $U$ is a uniform random variable in $[0,1]$ and $Z$ a standard Yule process,  
\begin{equation} \label{formula:RepresentacionYuleSimon}
    Y(t) := \indic{\{ U \leq t \}} Z_{p (\ln(t) - \ln(U)) },  \quad t \in [0,1], 
\end{equation}
is a Yule-Simon process with parameter $1/p$. Its law in $D[0,1]$, the space of $\mathbb{R}$-valued  rcll functions in the unit interval endowed with the Skorokhod topology,  will be denoted by $\mathbb{Q}$. It readily follows from the definition that this is a time-inhomogeneaous Markov process, with time-dependent birth rates  given at time $t$ by $\lambda_0(t) = 1/(1-t)$ and $\lambda_k(t)= pk/t$ for $k \in \{1,2,\dots  \}$. Remark as well that we have  $\proba{Y(t) \geq 1}{} = t$.  
In our work, only $p \in (0,1)$ will be used, and it always corresponds to the reinforcement parameter. The Yule-Simon process with parameter $1/p$ is  closely related  to the \textit{Yule-Simon distribution} with parameter $1/p$, i.e. the probability measure  supported on $\{1,2,\dots  \}$ with probability mass function given in terms of the Beta function B$(x,y)$ by 
\begin{equation} \label{formula:densityYuleSimon}
    p^{-1} \text{B}( k,1/p + 1 ) = p^{-1} \int_0^1  u^{p} (1-u)^{k-1} \dd u, \quad \quad \text{ for }  k \geq 1.
\end{equation}
The relation with the Yule process is simply that $Y(1)$ is distributed Yule-Simon with parameter $1/p$. In this work, we refer to $p\in (0,1)$ as a reinforcement or memory parameter, for reasons that will be explained shortly. In the following lemma we state for further use the conditional self-similarity property of the Yule-Simon process, a key feature that will be used frequently. 
\begin{lemma} \label{lemma:yuleSimonProcess} \emph{\cite[Corollary 2.3]{BertoinNRLP}}\\
Let $Y$ be a Yule-Simon process with parameter $1/p$ and fix  $t \in (0,1]$. Then, the process $(Y(rt))_{r \in [0,1]}$ conditionally on $\{ Y(t) \geq 1 \}$ has the same distribution $\mathbb{Q}$ as  $Y$.
\end{lemma}
 In particular, conditionally on $\{ Y(t) \geq 1 \},\,  Y(t)$ is distributed Yule-Simon with parameter $1/p$ and it follows that for every $t \in [0,1]$, $Y(t)$ has finite moments only of order $r < 1/p$. 
Moreover, by the previous lemma and the Markov property of the standard Yule process $Z$,   we deduce that  if $Y$ is a Yule-Simon process with parameter $1/p$ with $p \in (0,1)$ and $k \geq 1$,  we have 
\begin{equation} \label{formula:expectationYuleSimon}
    \esp{Y(t)}{} = (1-p)^{-1} t \quad \quad \text{ and }  \quad \quad   \esp{Y(t) | Y(s) = k }{} = k (t/s)^p \quad  \text{ for  any }   0 < s \leq t \leq 1, 
\end{equation}
while if $1/p > 2$, 
\begin{equation} \label{formula:covarYuleSimon}
   \esp{Y(s) Y(t)}{} = \frac{1}{(1-p)(1-2p)} s^{1-p}t^p.
\end{equation}
More details on these statements can be found in Section 2 of \cite{BertoinNRLP}. 
\subsection{Noise reinforced Lévy processes}\label{subsection:preliminariesNRLPs}
Now, we turn our attention to the main ingredients involved in the construction of NRLPs. For the rest of the section, fix  a real valued Lévy process $\xi$ of  characteristic triplet $(a,q^2,\Lambda)$, where $\Lambda$ is the Lévy measure, and recall that its characteristic exponent $\Psi(\lambda) := \log \esp{e^{i \lambda \xi_1}}{}$ is given by  the Lévy-Khintchine formula
\begin{equation} \label{formula:charExponentLevy}
    \Psi(\lambda) = i a \lambda - \frac{q^2}{2} \lambda^2 +  \int_{\mathbb{R}} \left( e^{i\lambda x} - 1 - i x \lambda \indic{ \{|x| \leq 1\}} \right) \Lambda(\dd x). 
\end{equation}
The constraints on the reinforcement parameter $p$ are  given in terms of the following  two indices introduced by Blumenthal and Getoor:  the  Blumenthal-Getoor (upper) index $\beta(\Lambda)$ of the Lévy measure $\Lambda$ is defined as 
\begin{equation} \label{index:BGIndexMeasure}
    \beta(\Lambda) := \inf \big\{ r >0 : \int_{[0,1]} |x|^r \Lambda(\dd x) < \infty \big\}, 
\end{equation} 
while the Blumenthal-Getoor index $\beta$ of the Lévy process $\xi$ is defined  by the relation
\begin{equation} \label{index:BGIndex}
    \beta :=
    \begin{cases}
    \beta(\Lambda) \quad \text{ if } q^2 = 0\\
    2 \hspace{9mm} \text{ if } q^2 \neq 0.
    \end{cases}
\end{equation}
 When $\xi$ has no Gaussian component, we have $\beta = \beta(\Lambda)$ and both notations will be used indifferently. We say that a memory parameter $p \in (0,1)$ is  admissible for the triplet $(a,q^2,  \Lambda)$ if $ p \beta < 1$. Now, fix $p$ an admissible memory parameter for $\xi$.  If $(S^{(n)}_k)$ is the $n$-skeleton of the Lévy process $\xi$, the sequence of reinforced versions with parameter $p$, 
 \begin{equation*}
      (\hat{S}^{(n)}_{\lfloor nt \rfloor })_{t \in [0,1]},   \hspace{8mm} n \geq 1,   
 \end{equation*}
  converge in the sense of finite dimensional distributions, as the mesh-size tends to 0, towards a process whose law was identified in \cite{BertoinNRLP} and called the noise reinforced Lévy process $\hat{\xi}$ of characteristics $(a,q^2,\Lambda,p)$. In the sequel, when considering a NRLP with parameter $p$, it will be implicitly assumed that $p$ is admissible for the corresponding triplet. For instance, when working with a memory parameter $p\geq 1/2$ it is implicitly assumed that $q = 0$. It was shown in  \cite[Corollary 2.11]{BertoinNRLP} that the finite-dimensional distributions of $\hat{\xi}$ can be expressed in terms of the Yule-Simon process  $Y$ with parameter $1/p$ and the characteristic exponent  $\Psi$ as follows: 
\begin{equation} \label{equation:fddsNRLPs01}
    \esp{ \exp \left\{ i \sum_{i=1}^k \lambda_i \hat{\xi}_{s_i} \right\}  }{}
    = 
     \exp \left\{ (1-p) \esp{\Psi\left( \sum_{i=1}^k \lambda_i Y(s_i) \right) }{}   \right\},  
\end{equation}
for $0 < s_1 < \dots < s_k \leq 1$. Now we turn our attention at defining NRLPs in $\mathbb{R}^+$. Notice that the construction given in the unit interval  in $\cite{BertoinNRLP}$ can not be directly extended to the real line since it relies on Poissonian sums of Yule-Simon processes, and these are only defined on the unit interval. 
\begin{proposition} \label{lemma:fddsExtendidas} \emph{(NRLPs in $\mathbb{R}^+$)}\\
Let $(a,q^2,\Lambda)$ be the triplet of a Lévy process of exponent $\Psi$ and consider an admissible memory parameter $p\in (0,1)$. There exists a process $\hat{\xi} = (\hat{\xi}_s)_{s \in \mathbb{R}^+}$ whose finite dimensional distributions satisfy that,  for any $0 < s_1 < \dots < s_k \leq t$, 
\begin{equation}  \label{equation:fddsNRLPs}
      \esp{ \exp \left\{ i \sum_{i=1}^k \lambda_i \hat{\xi}_{s_i} \right\}  }{} 
     = \exp \left\{ (1-p)  t \esp{\Psi\left( \sum_{i=1}^k \lambda_i Y( s_i/t) \right) }{}   \right\},  
\end{equation}
where the right-hand side does not depend on the choice of $t$. The process $\hat{\xi}$  is called a noise reinforced Lévy process with characteristics $(a,q^2,\Lambda, p)$. 
\end{proposition}
\begin{proof}
First, let us show that the right-hand side of (\ref{equation:fddsNRLPs}) does not depend on $t$. To prove this, pick another arbitrary $T>t$ and  write $r_i = s_i/t \in [0,1]$.  From conditioning on $\{ Y_{t/T} \geq 1 \}$,  an event with probability $t/T$,  by Lemma \ref{lemma:yuleSimonProcess} we get 
\begin{align} \label{equation:nodepenT}
     T \esp{\Psi\left( \sum_{i=1}^k \lambda_i Y( s_i/T) \right) }{} 
     &= t (T/t) \esp{\Psi\left( \sum_{i=1}^k \lambda_i Y( r_i \cdot (t/T)) \right) }{} \nonumber  \\
     &=  t  \esp{\Psi\left( \sum_{i=1}^k \lambda_i Y( r_i \cdot (t/T)) \right) \Big| Y(t/T) \geq 1 }{}  \nonumber \\
     &= t \esp{\Psi\left( \sum_{i=1}^k \lambda_i Y( s_i/t) \right) }{},  
\end{align}
proving our claim, and where in the second equality  we used  that $\Psi(0) = 0$. Now, let us  establish the existence of a process with finite-dimensional distributions characterised by \eqref{equation:fddsNRLPs}. Remark that by Kolmogorov's consistency theorem, it suffices to show that for arbitrary $1 \leq S < T$, there exists processes $\hat{X}^S= (\hat{X}^S_t)_{t \in [0,S]}$,  $\hat{X}^T:= (\hat{X}^T_t)_{t \in [0,T]}$ with finite dimensional distributions characterised  by the identity (\ref{equation:fddsNRLPs}) for $(s_i)$ in $[0,S]$, $t = S$ and $(s_i)$ in  $[0,T]$, $t = T$ respectively -- and hence satisfying that  $(\hat{X}^{T}_t)_{t \in [0,S]} \overset{\scr{L}}{=}(\hat{X}^{S}_t)_{t \in [0,S]}$. Write  $\hat{\xi}^S = (\hat{\xi}^S_t)_{t \in [0,1]}$  the reinforced version of the Lévy process $(\xi_{tS})_{t \in [0,1]}$, remark that the latter  has characteristic exponent $S \Psi$,  and set $(\hat{X}^S_t)_{t \in [0,S]} := (\hat{\xi}^S _{t/S})_{t \in [0,S]}$. From the identity (\ref{equation:fddsNRLPs01}), we deduce that, for any  $0 < s_1 < \dots < s_k$ in the interval $[0,S]$, we have:  
\begin{align}\label{equation:eq1-existancelawNRLP}
     \esp{ \exp \left\{ i \sum_{i=1}^k \lambda_i \hat{X}^S({s_i}) \right\}  }{} 
     = \exp \left\{ (1-p) S \esp{  \Psi\left( \sum_{i=1}^k \lambda_i Y( s_i/S) \right) }{}   \right\}. 
\end{align}
In particular $\hat{X}^S$ restricted to the interval $[0,1]$ has the same distribution as  $(\hat{\xi}_t)_{t \in  [0,1]}$ by the first part of the proof and \eqref{equation:fddsNRLPs01}.  If we consider the restriction of $(\hat{X}^T)_{t \in [0,T]}$ to the interval $[0,S]$, we obtain similarly and by applying  (\ref{equation:nodepenT})  that, for any $0 < s_1 < \dots < s_k \leq S$,
\begin{align*}
    \esp{ \exp \left\{ i \sum_{i=1}^k \lambda_i \hat{X}^T({s_i}) \right\}  }{} 
     &= \exp \left\{ (1-p) T \esp{  \Psi\left( \sum_{i=1}^k \lambda_i Y( s_i/T) \right) }{}   \right\} \\
     &= \exp \left\{ (1-p) S \esp{  \Psi\left( \sum_{i=1}^k \lambda_i Y( s_i/S) \right) }{}   \right\},  
\end{align*}
and it follows that $\hat{X}^T$ restricted to $[0,S]$ has the same distribution as $\hat{X}^S$. Since this holds for any $1 \leq S < T$, we deduce by Kolmogorov's consistency theorem the existence of a process  satisfying for any $0 < s_1 < \dots < s_k \leq t$, the identity (\ref{equation:fddsNRLPs}). In particular, from taking the value  $t = 1$, it follows that this process  satisfies that its restriction to $[0,1]$  has the same law as $\hat{\xi}$ by  (\ref{equation:fddsNRLPs01}). 
\end{proof}
For later use, notice from  (\ref{equation:fddsNRLPs}) that for any fixed  $t \in \mathbb{R}^+$, we have  the following equality in law 
\begin{equation} \label{identity:extension}
    (\hat{\xi}_{st})_{s \in [0,1]} \overset{\scr{L}}{=} (\reallywidehat{\xi_{\cdot t}})_{s \in [0,1]}
\end{equation}
where the right-hand side stands for the noise-reinforced version of the Lévy process $(\xi_{st})_{s \in [0,1]}$. In particular,  $(\hat{\xi}_{st})_{s \in [0,1]}$ is the NRLP associated to the exponent $t \Psi$ with same reinforcement parameter. 

\subsection{Building blocks: noise reinforced Brownian motion and noise reinforced compound Poisson process}\label{subsection:preliminaries,examplesNRLPs}
The characteristic exponent $\Psi$ can be naturally decomposed in tree terms, 
\begin{equation} \label{introduction:decomp_expo}
    \Psi(\lambda) = \big( i a \lambda  - q\frac{\lambda^2}{2} \big)+ \Phi^{(2)}(\lambda)  + \Phi^{(3)} (\lambda), 
\end{equation}
where respectively, we write 
\begin{equation*}
 \Phi^{(2)}(\lambda) := \int_{\{ |x| \geq  1 \}} \left( e^{i \lambda x} - 1  \right)  \Lambda(\dd x)
 \quad \text{ and } \quad 
 \Phi^{(3)}(\lambda) := \int_{\{ |x| < 1 \}} \left( e^{i \lambda x} - 1 - i \lambda x  \right) \Lambda(\dd x). 
\end{equation*}
This decomposition yields that the Lévy process $\xi$ can be written as the sum of tree independent Lévy process of radically different nature. Namely, we have   $\xi_t = (at + qB_t) +  \xi^{(2)}_t + \xi^{(3)}_t $, for $t \geq 0$,  where $B$ is a Brownian motion,     $\xi^{(2)}$ is a compound Poisson process with exponent $\Phi^{(2)}$ and $\xi^{(3)}$ is the so-called compensated sum of jumps with characteristic exponent $\Phi^{(3)}$. In the reinforced setting, it readily follows from  the identity \eqref{equation:fddsNRLPs}  that an analogous decomposition holds for NRLPs. More precisely,  the NRLP $\hat{\xi}$ of  characteristics $(a,q^2,\Lambda,p)$   can be written as a sum of three independent NRLPs,
\begin{equation} \label{descomposicionNRLP}
    \hat{\xi}_t \overset{\scr{L}}{=} (at + q  \hat{B}_t) + \hat{\xi}^{(2)}_t + \hat{\xi}^{(3)}_t, \quad t \geq 0,  
\end{equation}
 the equality holding in law, and 
where we denoted respectively by $\hat{B}$,  $\hat{\xi}^{(2)}$, $\hat{\xi}^{(3)}$,   independent  reinforced versions of the Lévy processes $B$, ${\xi}^{(2)}$, ${\xi}^{(3)}$. Notice that  their 
 respective  characteristics are given by   $(a,q^2, 0 ,p)$,  $(0,0, \indic{(-1,1)^c}\Lambda,p)$ and $(0,0, \indic{(-1,1)}\Lambda,p)$.  Let us now give  a brief description of these three building blocks separately:  \medskip \\
$\circ$ \textit{Noise reinforced Brownian motion: }Assume  $p< 1/2$, consider  a Brownian motion $B$ and set ${\xi} := {B}$. In that case, we simply have $\Psi(\lambda) = -\lambda^2/2$ and we write $\hat{B}$ for the corresponding noise reinforced Lévy process  $\hat{\xi}$.  The process $\hat{B}$ is the so-called noise  reinforced Brownian motion (abbreviated NRBM) with reinforcement parameter $p$,  a centred Gaussian process with covariance given by:  
\begin{equation} \label{formula:covarianceNRBMoriginal}
    \esp{\hat{B}_t \hat{B}_s}{} = \frac{(t \vee s)^p  (t \wedge s)^{1-p} }{1-2p}.
\end{equation}
 Indeed, recalling (\ref{formula:covarYuleSimon}),  observe first that for any  $0 \leq t,s  < T$  the covariance (\ref{formula:covarianceNRBMoriginal}) can be written in terms the Yule-Simon process $Y$ with parameter $1/p$ as follows: 
\begin{equation} \label{formula:covarianceNRBM}
    \esp{\hat{B}_t \hat{B}_s}{} = (1-p) T \cdot \esp{Y(t/T) Y(s/T)}{}.  
\end{equation}
It is now straightforward to deduce from   (\ref{equation:fddsNRLPs}) with $\Psi(\lambda) = -\lambda^2/2$  that the noise reinforced version of $B$ corresponds to the Gaussian process with covariance (\ref{formula:covarianceNRBMoriginal}). The noise reinforced Brownian motion admits a simple representation as a Wiener integral. More precisely, the  process 
\begin{equation} \label{formula:representationNRBM}
    t^{p} \int_0^t s^{-p}\dd B_s,  \quad \quad t \geq 0, 
\end{equation}
has the law of a noise reinforced Brownian motion with parameter $p$. Remark that when $p=0$, there is no reinforcement and we recover a Brownian motion in (\ref{formula:representationNRBM}). As was already mentioned, noise reinforced Brownian motion plays the role of Brownian motion in the reinforced setting, since it is the scaling limit of noise reinforced random walks under mild assumptions on the law of the typical step. We refer to \cite{BertoinUniversality, InvariancePrinciplesNRRW} for a  detailed discussion. \medskip \\
$\circ$ \textit{Noise reinforced compound Poisson process: }If $\xi$ is a compound  Poisson process with rate $c>0$ and jumps with law $P_X$, then its Lévy measure is just  $\Lambda(\dd x) = c P_X(\dd x)$, and   any $p\in (0,1)$ is admissible. When working in $[0,1]$, the noise reinforced compound Poisson process $\hat{\xi}$ admits a simple representation in terms of Poissonian sums of Yule-Simon processes. In this direction, let $\mathbb{Q}$ be the law of the Yule Simon process with parameter $1/p$ and consider a Poisson random measure $\mathcal{M}$ in $\mathbb{R}^+ \times D[0,1]$ with intensity $(1-p) \Lambda  \otimes \mathbb{Q}$. If we denote its atoms by  $(x_i,Y_i)$,  the process  
\begin{equation} \label{example:reinfPoissonProcess}
    \hat{\xi}_t = \sum_{i} x_i Y_i(t), \quad t \in [0,1], 
\end{equation}
has the law of the noise reinforced version of $\xi$ with reinforcement parameter $p$ -- as can be easily verified by Campbell's formula and was already established in  \cite[Corollary 2.11]{BertoinNRLP}. Notice that \eqref{example:reinfPoissonProcess} is a finite variation process and its jump sizes are dictated by $P_X(\dd x)$. Getting back to  \eqref{descomposicionNRLP}, it readily follows form our discussion that the NRLP $\hat{\xi}^{(2)}$ associated with the exponent  $\Phi^{(2)}$ is a reinforced compound Poisson process and  its  jumps-sizes are greater than one. Finally, notice that if $P_X = \delta_{ 1}$, the Lévy process  $\xi$ is just a Poisson process with rate $c$ and we deduce from the last display a simple representation for the reinforced Poisson process $\hat{N}$ in $[0,1]$. Observe that  it is a counting process, since the atoms $x_i$ are  then identically equal to 1.  
\medskip \\
$\circ$ \textit{Noise reinforced compensated compound Poisson process: } Let us now introduce properly $\hat{\xi}^{(3)}$, viz. the noise reinforced version of the compensated martingale ${\xi}^{(3)}$.    When working in $[0,1]$, this process  also admits a representation in terms of random  series of Yule-Simon processes. In this direction,  consider  ${ \mathcal{M} } := \sum_i \delta_{(x_i, Y_i)}$ a Poisson random measure  with intensity $(1-p)\Lambda \otimes \mathbb{Q}$ and for each $a\in [0,1]$, set  
\begin{equation} \label{prelim:explicitConstructionCompensatedSums}
 \hat{\xi}^{(3)}_{a,1}(t) := \sum_{i} \indic{\{ a \leq |x_i|  < 1\}} x_i Y_i(t) - t \int_{\{ a \leq |x| < 1 \}} x \Lambda(\dd x),   \quad t \in [0,1].
\end{equation}
In the terminology of \cite[ Section 2]{BertoinNRLP}, the process $\hat{\xi}^{(3)}_{a,1}$ is a     Yule-Simon compensated series \footnote{The notation used in \cite{BertoinNRLP} for $\hat{\xi}^{(2)}_t$ and $\hat{\xi}^{(3)}_t$ is respectively $\Sigma_{1,\infty}(t)$ and $\Sigma_{0,1}^{(c)}(t)$. These are respectively referred to as  Yule-Simon series and compensated Yule-Simon series.},  and note that $\mathbb{E}[\hat{\xi}^{(3)}_{a,1}(t)]  = 0$ for every $t \in [0,1]$.  Moreover, the following family indexed by $a \in (0,1)$,  
\begin{equation} \label{notation:nrlpFueraOrigen}
    \hat{\xi}^{(3)}_{a,1}(t),  \quad  \text{ for } t \in [0,1], 
\end{equation}
is a  collection of NRLPs with memory parameter $p$,  Lévy measure $\indic{\{a  \leq |x| < 1\}} \Lambda(\dd x)$ and the corresponding exponent writes: 
\begin{equation*}
    \Phi_a^{(3)}(\lambda) := \int_{\{a  \leq |x| < 1\}} \left( e^{i\lambda x} -1- i \lambda x \right) \Lambda(\dd x). 
\end{equation*}
Notice that for each $a > 0$, the process $\hat{\xi}^{(3)}_{a,1}$ is rcll and with    jump-sizes in  $[a,1]$. Now, the process defined  at each fixed $t$ as the pointwise and $L_1(\mathbb{P})$-limit 
\begin{equation} \label{equation:explicitCompensatedYSseries}
    \hat{\xi}^{(3)}_t := \lim_{a \downarrow 0} \hat{\xi}^{(3)}_{a,1}(t),  
\end{equation}
is a NRLP with characteristics $(0,0,\indic{\{ |x| < 1 \}} \Lambda)$.  In contrast with $\xi^{(3)}$, the noise reinforced version $\hat{\xi}^{(3)}$ is no longer a martingale, we shall discuss this point in the next section in detail. For latter use, we point out from \cite[Section 2]{BertoinNRLP} that the convergence in the previous display also holds in  $L_r(\mathbb{P})$, for $r$ chosen according to 
\begin{equation} \label{condition:choiceLq.}
    r \in 
    (\beta(\Lambda) \vee 1, 1/p), \text{ if } 1/p \leq  2 
    \hspace{5mm} \text{ and } \hspace{5mm}
    r = 2,   \text{ if } 1/p > 2.
\end{equation}
In particular, we have  $\hat{\xi}^{(3)}_t \in L_r(\mathbb{P})$ and $\mathbb{E}[\hat{\xi}^{(3)}_t] = 0$ for every $t$. We refer to \cite{BertoinNRLP} for a complete account on this construction and for a proof of the convergence in \eqref{equation:explicitCompensatedYSseries}.  The convergence in \eqref{equation:explicitCompensatedYSseries}   will be strengthen in the sequel, by showing that it holds uniformly in $[0,1]$. At this point, we have introduced the main ingredients needed for this work.

\addtocontents{toc}{\vspace{\contentsSpacingBefore}}
\section{Trajectorial regularity} \label{section:extensionAndRegularity}
\addtocontents{toc}{\vspace{0 pt}}

The purpose of this short section is  to establish the following  regularity theorem: 
\begin{theorem}\label{regularityNRLP} A noise reinforced Lévy process $\hat{\xi}$  has a rcll modification, that we still denote by $\hat{\xi}$. Moreover, if for $\epsilon \in (0,1)$, $\hat{\xi}^{(3)}_{0,\epsilon}$ denotes  a NRLP with characteristics $(0,0, \indic{\{|x| < \epsilon\}} \Lambda, p)$, then  for any $t > 0$ we have: 
\begin{equation} \label{equation:convergenceSmallJumps}
   \lim_{\epsilon  \downarrow 0} \esp{\sup_{s \leq t} |\hat{\xi}^{(3)}_{0,\epsilon}(s)|  }{} = 0.
\end{equation}
\end{theorem}
Before proving this result, let us explain the role of  \eqref{equation:convergenceSmallJumps}. Working in $[0,1]$ and  with the construction  \eqref{equation:explicitCompensatedYSseries} for $\hat{\xi}^{(3)}$, remark that for any $\epsilon \in (0,1)$ we can write  $\hat{\xi}^{(3)} = \hat{\xi}^{(3)}_{0,\epsilon} + \hat{\xi}^{(3)}_{\epsilon,1}$, where  $|\Delta \hat{\xi}_{\epsilon,1}^{(3)}(t)| \geq \epsilon$ for every jump-time $t \in [0,1]$ by construction. Now, the convergence \eqref{equation:convergenceSmallJumps} shows that in fact, the jumps of $\hat{\xi}$ of size greater than $\epsilon$ are precisely the jumps of $\hat{\xi}^{(2)} + \hat{\xi}^{(3)}_{\epsilon,1}$. Hence, when working in $[0,1]$, the jumps of  $\hat{\xi}^{(3)}$  are precisely the jumps of the weighted Yule-Simon processes $x_i Y_i(t)$ -- heuristically, this is  the continuous-time analogue of the dynamics described for the noise reinforced random walk. This fact will be used in Section \ref{subection:jump_LevyIto-Part1}. Moreover, \eqref{equation:convergenceSmallJumps} allow us to  improve the convergence stated in  \eqref{equation:explicitCompensatedYSseries} towards  $\hat{\xi}^{(3)}$. Namely, it follows that for some subsequence $(a_n)$ with $a_n \downarrow 0$ as $n \uparrow \infty$,  the convergence
\begin{equation*}  
    \lim_{n \rightarrow \infty}(\hat{\xi}_{a_n , 1}^{(3)}(s)  )_{s \in [0,1]} =  (\hat{\xi}^{(3)}_s  )_{s \in [0,1]}, 
\end{equation*}
 holds a.s. uniformly in $[0,1]$.  Remark that the convergence in the previous display was only stated when working in $[0,1]$ since,  so far, the only explicit construction of NRLPs is the one   in the unit interval  we  recalled  from  \cite{BertoinNRLP}.  In Section \ref{subsection:jump_LevyItoPart2} we shall address this point. 
\par 
The rest of the section is devoted to the proof of Theorem \ref{regularityNRLP}.   Recalling  the building blocks introduced in Section \ref{subsection:preliminaries,examplesNRLPs} and the identity in distribution \eqref{descomposicionNRLP},    $\hat{\xi}^{(2)}$ is a reinforced compound Poisson process  and hence has finite variation rcll trajectories, while $\hat{B}$ is continuous.  It is then clear that the only difficulty consists in  establishing the regularity of the process $\hat{\xi}^{(3)}$ and we  rely on a remarkable  martingale associated with centred NRLPs, that we now introduce. This martingale  will play a key role in this work. 
\begin{proposition} \label{proposition:laMartingala}
Consider  a Lévy process $\xi$ with  characteristic exponent $\Psi$ satisfying    $\Psi'(0) = 0$ and Lévy measure fulfilling the integrability condition  $\int_{\{ |x| \geq 1 \}}x \Lambda(\dd x ) < \infty$. Then, the process $M = (M_t)_{t \in \mathbb{R}^+}$ defined as  $M_0 = 0$ and for $t > 0$, as  $M_t = t^{-p}\hat{\xi}_t$,  is a martingale. Consequently, $M$ has a rcll modification. 
\end{proposition}
\begin{proof}  Recall from (\ref{descomposicionNRLP})  that in that case, $\hat{\xi}$ can be written as a sum  of two independent processes $\hat{\xi} = q  \hat{B} + \hat{\xi}^{(3)}$, where $\hat{B}$ is a noise reinforced Brownian motion. Recalling the representation (\ref{formula:representationNRBM}) for $\hat{B}$, it follows that $(t^{-p}\hat{B}_t)_{t \in \mathbb{R}^+}$ is a continuous martingale and we assume therefore that $q = 0$.
\par Turning our attention to $\hat{\xi}^{(3)}$, notice that $M_t$ is  in $L_r(\mathbb{P})$ for $r$ chosen according to \eqref{condition:choiceLq.} and that   $E[M_t] = 0$ since, as we discussed after \eqref{condition:choiceLq.}, we have $\hat{\xi}^{(3)}_t \in L_r(\mathbb{P})$, $\mathbb{E}[\hat{\xi}^{(3)}_t]=0$. Now, it remains to show that $(M_t)_{t \in (0,1]}$ satisfies the martingale property. In this direction it is enough to check that for any $0< t_0 < \dots < t_k<t$ and $\lambda_1, \dots , \lambda_{k-1} \in \mathbb{R}$, we have  
\begin{equation} \label{equation:igualdadAverificarMartingala}
    \esp{t_k^{-p}\hat{\xi}^{(3)}_{t_k} \exp \left\{i \sum_{i=1}^{k-1} \lambda_i \hat{\xi}^{(3)}_{t_i}  \right\}  }{} 
    = \esp{t_{k-1}^{-p}\hat{\xi}^{(3)}_{t_{k-1}} \exp \left\{i \sum_{i=1}^{k-1} \lambda_i  \hat{\xi}^{(3)}_{t_i} \right\}  }{}. 
\end{equation}
On the one hand, under our standing assumptions, the left-hand side of (\ref{equation:igualdadAverificarMartingala}) corresponds to the derivative at $\lambda_k=0$ of (\ref{equation:fddsNRLPs}) multiplied by $-i t^{-p}_k$ and hence equals:  
\begin{equation*}
-i t(1-p) \exp\left\{t (1-p) \esp{\Psi \left( \sum_{j=1}^{k-1} \lambda_j Y(t_j/t) \right) }{} \right\} \cdot  \esp{ H\left( Y(s) \, : s \leq t_{k-1}/t \right) Y(t_k/t)}{} t_k^{-p}, \end{equation*}
for $H$ defined as
\begin{equation*}
   \quad H\left( Y(s) \, : s \leq t_{k-1}/t \right) := \Psi' \left( \sum_{j=1}^{k-1} \lambda_i  Y(t_j/t)  \right).
\end{equation*}
Remark that this is a $\sigma(Y(s) : s \leq t_{k-1}/t )$-measurable random variable.
On the other hand, the right-hand side of (\ref{equation:igualdadAverificarMartingala}) corresponds to the derivative with respect to $\lambda_{k-1}$ of (\ref{equation:fddsNRLPs}) multiplied by $-i t_{k-1}^p$ for $\lambda_{k}=0$ and similarly,  we deduce that the right-hand side of (\ref{equation:igualdadAverificarMartingala}) writes:   
\begin{equation*}
-it(1-p)\exp\left\{ t(1-p) \esp{\Psi \left( \sum_{j=1}^{k-1}  \lambda_j Y(t_j/t)   \right) }{} \right\} \cdot \esp{ H\left( Y(s) \, : s \leq t_{k-1}/t \right) Y(t_{k-1}/t)   }{} t_{k-1}^{-p}. 
\end{equation*}
Now, it only remains to show that: 
\begin{equation} \label{equation:ingualdadAVerificarParteII}
    \esp{  H(Y(s) \, : s \leq t_{k-1}/t)   Y(t_k/t)   }{} t_k^{-p} 
    =\esp{ H\left( Y(s) \, : s \leq t_{k-1}/t \right) Y(t_{k-1}/t)  }{} t_{k-1}^{-p}. 
\end{equation}
 Notice that since $\Psi'(0)=0$ and $Y$ is increasing, $H\left( Y(s) \, : s \leq t_{k-1}/t \right) $ vanishes if $Y({t_{k-1}}/t) =0$. This allows us to restrict the terms inside the expectations in (\ref{equation:ingualdadAVerificarParteII})  to $\{ Y(t_{k-1}/t) \geq 1 \}$ and to apply the   Markov property (\ref{formula:expectationYuleSimon}) at time $t_{k-1}/t$ to get: 
\begin{align} \label{equation:finalpruebaMartingala}
    & \esp{ H(Y(r) : r \leq t_{k-1}/t) Y(t_{k}/t) }{} t^{-p}_{k} \nonumber \\
    & = \sum_{j=1}^{\infty} \esp{ H(Y(r) : r \leq t_{k-1}/t) \esp{Y(t_{k}/t) | Y(t_{k-1} /t) =j }{} \indic{\{ Y(t_{k-1}/t) =j \}}  }{} t^{-p}_{k}  \nonumber \\
    & = \sum_{j=1}^{\infty} \esp{ H(Y(r) : r \leq t_{k-1}/t) \cdot j (t_{k}/t_{k-1} )^p \indic{\{ Y(t_{k-1}/t) =j \}}  }{} t^{-p}_{k}  \nonumber \\
    & = \esp{ H(Y(r) : r \leq t_{k-1}/t) Y({t_{k-1}}/t)   }{} t^{-p}_{k-1}, 
\end{align}
proving the claim. 
\end{proof}

Let us now conclude the proof of Theorem \ref{regularityNRLP}.

\begin{proof}[Proof of Theorem \ref{regularityNRLP}.]
The first assertion  is now a consequence of the following simple observation:  denoting by   $\closure{M}$    the rcll modification  of the martingale $M = (t^{-p} \hat{\xi}^{(3)}_t)_{t \in \mathbb{R}^+}$,  it is then clear that the process   $\hat{J}^{(3)} := t^p \closure{M}_t$, for $t \geq 0$, is a rcll modification of $\hat{\xi}^{(3)}$.  Notice by intergrating by parts that consequently,  the process $\hat{\xi}^{(3)}$ is a semimartingale, this will be needed in Section \ref{subection:jump_LevyIto-Part1}. To prove the second claim, remark that by the observation right after \eqref{identity:extension}, it suffices to work on the time interval $[0,1]$. Moreover,  by  Proposition \ref{proposition:laMartingala},  for each $\epsilon >0$,  the process 
\begin{equation*}
    M^{(\epsilon)} := (s^{-p}\hat{\xi}_{0 , \epsilon}^{(3)}(s))_{s \in (0,1]}
\end{equation*}
with $M^{(\epsilon)}_0 = 0$, is a $L_r(\mathbb{P})$ rcll martingale in $[0,1]$,  for $r$ chosen according to (\ref{condition:choiceLq.}).   
Since $r > 1$, by Doob's inequality at time $t = 1$ we have
\begin{align*}
    \esp{\sup_{s \leq t} |\hat{\xi}^{(3)}_{0,\epsilon}(s)|^r }{} 
     \leq \esp{\sup_{s \leq t} \big| s^{-p} \hat{\xi}^{(3)}_{0,\epsilon}(s) \big|^r}{} 
    \leq C_r  \esp{ |\hat{\xi}^{(3)}_{0,\epsilon}(1)|^r }{}, 
\end{align*}
for some constant $C_r >0 $, and it remains to show that the right-hand side converges to $0$ as $\epsilon \downarrow 0$. However, this is a consequence of \eqref{equation:explicitCompensatedYSseries}. More precisely, recalling the construction detailed in \eqref{prelim:explicitConstructionCompensatedSums}, note that    $\hat{\xi}^{(3)}_t - \hat{\xi}^{(3)}_{\epsilon,1}(t)$ has the same distribution as $\hat{\xi}^{(3)}_{0, \epsilon}(t)$ for every $t \in [0,1]$ and $\epsilon > 0$. Since the convergence  \eqref{equation:explicitCompensatedYSseries} still holds in $L_r(\mathbb{P})$,  the result follows by taking the limit as $\epsilon \downarrow 0$.  
\end{proof}

 \par Now that we have established that a NRLP is a rcll process, in the next section we  study the structure of its jump process $(\Delta \hat{\xi}_t)$.  Since it will share striking similarities with the jump process of a  Lévy process, before concluding the section  we recall well known results on $(\Delta \xi_t)$. Namely, if $\xi$ is a Lévy process with Lévy measure $\Lambda$, its jump measure 
\begin{equation} \label{ppp-Levy}
    \mu ( \dd t , \dd x) = \sum_s \indic{\{ \Delta {\xi}_s \neq 0  \}} \delta_{(s ,  \Delta \xi_s  )} ( \dd t ,  \dd x), 
\end{equation}
is a homogeneous Poisson point process (abbreviated PPP) with characteristic measure $\Lambda(\dd x)$. Such a PPP can be constructed by decorating the point process of jumps of a Poisson process, and  it is classic that \eqref{ppp-Levy} is determined by the following two properties:
\begin{enumerate}
    \item[{(i)}] For any Borelian $A$ with $\Lambda(A) < \infty$,  the counting process of jumps  $\Delta \xi_s \in  A$ occurring until time $t$, defined as 
\begin{equation*}
    N_A(t)  = \# \big\{  (s, \Delta  \xi_s) \in [0,t]\times A \big\}, \quad \quad t \geq 0, 
\end{equation*}
is a Poisson process with rate $\Lambda(A)$. 
    \item[{(ii)}]   If $A_1, \dots A_k$ are disjoint Borelians with $\Lambda(A_i)< \infty$ for all $i \in \{1 , \dots, k\}$,  the processes ${N}_{A_1}, \dots , {N}_{A_k}$ are independent.
\end{enumerate}
In particular, from (i), it follows that  $(N_A(t) - \Lambda(A)t)_{t \in \mathbb{R}^+}$ is a martingale.

\addtocontents{toc}{\vspace{\contentsSpacingBefore}}
\section{Reinforced Lévy-Itô decomposition}\label{section:NRPPP}

This section is devoted to the study of  the jump process $(\Delta \hat{\xi}_s)_{s \in \mathbb{R}^+}$ and the associated jump measure in $\mathbb{R}^+\times \mathbb{R}$, viz. 
\begin{equation} \label{definition:reinforcedPoissonPointProcess}
    \hat{\mu}( \dd t, \dd x) := \sum_s \indic{\{ \Delta \hat{\xi}_s \neq 0  \}} \delta_{(s ,  \Delta \hat{\xi}_s  )} ( \dd t ,  \dd x).
\end{equation} 
In this direction, we  shall  introduce in Definition \ref{proposition:reinforcedPPPbyMarking}  below  a family of random measures in $\mathbb{R}^+ \times \mathbb{R}$ under the name \textit{noise reinforced Poisson point processes} -- abbreviated NRPPPs -- that will play the analogous role of PPPs for the jump measure of Lévy processes. Each element of this family of measures is parametrized by a sigma finite measure $\Lambda$ in $\mathbb{R}$, that we refer to as its characteristic measure, and a positive value $p \in (0,1)$, that we call its reinforcement parameter. The construction of NRPPPs consists essentially in the reinforced version  of the one of PPPs. More precisely, we shall construct them by decorating the point process of jumps of a \textit{reinforced} Poisson process.  The main result of this section is the reinforced version of the celebrated Lévy-Itô decomposition:
\begin{theorem}\emph{(Reinforced Lévy-Itô decomposition)} \label{lemma:PreFormulaCompensacion} \\
The jump measure $\hat{\mu}$ of  $\hat{\xi}$ is a noise reinforced Poisson point process with characteristic measure $\Lambda(\dd x)$ and reinforcement parameter $p$. 
\end{theorem}
\noindent \textit{The rest of the section is organised as follows: }In Section \ref{subsection:JumpsReinforcedPP} we restrict our study to the jump process of reinforced Poisson processes. In Section \ref{subsection:jump_NRPPPdecorated}, we construct NRPPPs by decorating the jump process of reinforced Poisson processes and  then study its basic properties.  For instance,   in Proposition \ref{proposition:exponentialFormulasJumpProcess} we prove a characterisation in the same vein as the one holding for PPPs, recalled at the end of Section \ref{section:extensionAndRegularity}. Finally, in Section \ref{subection:jump_LevyIto-Part1} we prove Theorem \ref{lemma:PreFormulaCompensacion} and in Proposition \ref{theorem:compensationFormula}  we  identify  the predictable compensator of $\hat{\mu}$.  
{   
\subsection{The jumps of noise reinforced Poisson processes} \label{subsection:JumpsReinforcedPP}

\indent Let us start by introducing the basic building block of this section. \smallskip  \\
{\noindent $\circ$ \textit{Noise reinforced Poisson process:} 
When $\xi$ is a Poisson process $N$ with rate $c$, any reinforcement parameter $p\in (0,1)$ is admissible and recall from the discussion following  (\ref{example:reinfPoissonProcess}) that $\hat{N}$ is a counting process. Moreover,  the corresponding noise reinforced Poisson process (abbreviated NRPP) with rate $p$ has finite dimensional distributions characterised, for any $0 < s_1 < \dots < s_k \leq t$ and $\lambda_j \in \mathbb{R}$, by the identity 
\begin{equation}  \label{definition:noiseReinfPoisson}
      \esp{ \exp \left\{ i \sum_{i=1}^k \lambda_i \hat{N}_{s_i} \right\}  }{} 
     = \exp \left\{ (1-p) c  t \esp{\left( \exp \bigg\{ {i \sum_{i=1}^k \lambda_i Y( s_i/t)} \bigg\} -1 \right) }{}   \right\}.  
\end{equation}
A Poisson process with rate $c$ has associated to it the random measure $\dd N_s$, also called its point process of jumps.  This is a Poisson random measure in $\mathbb{R}^+$ with intensity $c \dd t$ and it has a natural reinforced counterpart: namely,  the random measure  $\dd  \hat{N}_s$, that we shall now study in detail.  \smallskip \\ }
\indent To do so, we start by introducing some standard notation for point processes. We shall identify discrete random sets $D= \{ t_1, t_2, \dots  \} \subset \mathbb{R}$ with counting measures $\sum_{t \in D} \delta_{t}$ and for $f: \mathbb{R} \mapsto \mathbb{R}$,  we  use the notation $\langle  D , f\rangle$ for $\sum_{t \in D} f(t)$. The collection of counting measures in $\mathbb{R}$ is denoted by $\mathcal{M}_c$. We will make use of the following two basic transformations: for $x \in \mathbb{R}$, we denote by $\mathcal{T}_xD$ the translated point process $\{  t + x : t \in D  \}$ and  for $f:\mathbb{R} \mapsto \mathbb{R}$, we write   $D \circ f^{-1}$ the push-forwarded  point process $\{  f(t) : t \in D \}$. 
\par
Now,  consider an increasing sequence of random times $0 = T_0<T_1<T_2<...$,  such  that the increments  $(T_{n}-T_{n-1}: n \geq 1)$ are independent and for any $n \geq 1$,  $T_n-T_{n-1}$ is exponentially distributed with parameter $p n$. Write $D := \{ 0, T_1, T_2, ... \}$ the point process associated to this family and we denote its law in $\mathcal{M}_c$ by $\mathbb{D}(\dd \mu)$. From these ingredients, we define a decorated measure as follows: first, consider $\scr{E}$ a Poisson point process with intensity $c (1-p) e^{t} \dd t$ in $\mathbb{R}$ and, for each atom $u \in \scr{E}$, let  $D_u$ be an independent copy of $D$.   Then,  we set 
\begin{equation} \label{definition:reinforcedPPExpScale} 
    \mathcal{L}( \dd s ) :=  \sum_{u \in \scr{E}}  \sum_{t \in D_u} \delta_{ u + t} = \sum_{u \in \scr{E}} \mathcal{T}_{u} D_u. 
\end{equation}
Remark that if $(Z_t)$ is a standard Yule process started from $1$, $D$ has the same law as   the point process induced by  the jump-times of $(Z_{tp})$, with a Dirac mass at $0$. The next proposition shows  that the law of the point process of jumps of a noise reinforced Poisson process with rate $c$ is precisely $\mathcal{L} \circ \exp^{-1}$, the pushforward of $\mathcal{L}$ by the exponential function.  

\begin{proposition} \label{lemma:nrpp_jumpProcess}
The following properties hold:
\begin{enumerate} 
    \item[\emph{(i)}] Let $\hat{N}$ be a noise-reinforced Poisson process with {rate $c$} and write  $\hat{\scr{P}} := \dd \hat{N}_s$ the   point process of its jump-times in $\mathbb{R}^+$. Then, we have the equality in distribution  $\hat{\scr{P}} \overset{\scr{L}}{=} \mathcal{L} \circ \exp^{-1}$. We will still refer to $\hat{\scr{P}}$ as a  reinforced Poisson process with rate $c$ and reinforcement parameter $p$. 
    \item[\emph{(ii)}]  If   $Y$ is a 
 Yule-Simon process with parameter $1/p$, for any $f:\mathbb{R}^+ \mapsto \mathbb{R}^+$ we have 
\begin{equation} \label{identity:laplaceNRPPcompactos}
    \hspace{-23mm} - \log \esp{ \exp \bigg\{ - \langle \hat{\scr{P}} ,  \indic{(0,t]} f \rangle \bigg\} }{} = t c \cdot (1-p) \esp{1-e^{-\int_0^1 f(st)\dd Y(s) }}{}.
\end{equation}
\end{enumerate}
\end{proposition}
 \noindent In particular, from (\ref{definition:reinforcedPPExpScale}) and (i)  we deduce the following   identity in distribution: if $\mathscr{P}$ is a Poisson process in $\mathbb{R}_+$ with intensity $c(1-p) \dd t$, we have
 \begin{equation}
     \label{definition:PPofNRPP_intro_intro}
    \hat{\scr{P}} = \sum_{s\in \mathbb{R}^+} \indic{\{ s : \Delta \hat{N}_s = 1 \}}  \delta_{ s } \overset{\scr{L}}{=}  \sum_{u \in \scr{P}} \sum_{t \in D_u} \delta_{ue^{t}}. 
 \end{equation}
Roughly speaking,  the jumps of $\hat{N}$ consist in  Poissonian jumps $u \in\scr{P}$ which -- in analogy with the discrete setting --  we refer to as innovations,  and each  $u$  has attached to it a family  $\{ue^{t} : t \in D_u, \, t \neq 0\}$ which should be interpreted as repetitions of the original $u$ through time.

\begin{figure}[htbp!]
   \centering
   \scalebox{1.4}{  
  \includegraphics[scale=1.0]{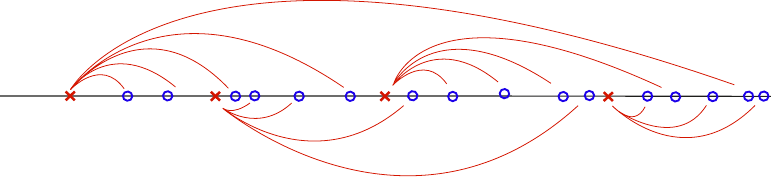}
  }
  \caption{Sketch of the jumps of a noise reinforced Poisson process. We marked by   {\color{red} x} the jumps corresponding to innovations, while each linked  {\color{blue}o} is a repetition of the former.}
\end{figure}

Notice that the time at which $u$ occurs affects the rate of the subsequent repetitions, slowing the rate down as $u$ grows. This is closely related to what happens to the rate at which a step is repeated  in  a step reinforced random walk, depending on its first time of appearance. For later use, remark  that for fixed $u \in \mathbb{R}^+$,  the atoms of  $\sum_{t \in D} \delta_{ue^{t}}$ are distributed as the jump times  of the counting process 
\begin{equation} \label{definition:logtimeChangedYule}
    \indic{\{u \leq s \}} Z_{p( \ln(s) - \ln(u) )}\, ,  \quad \quad s \geq 0.
\end{equation}
\begin{proof}
To establish the identity in distribution stated in (i), we compute the respective Laplace functional of both random measures. Starting with $\hat{\mathscr{P}}$, fix $t \geq 0$ and recall from  the identity in distribution (\ref{identity:extension}) that $(\hat{N}_{ts})_{s \in [0,1]}$ has the same law as a noise reinforced Poisson process with same reinforcement parameter $p$ and rate $tc$, say  $(\hat{N}^{(t)}_s)_{s \in [0,1]}$.  This NRLP is defined in $[0,1]$  and hence admits a simple representation in terms of Poisson random measures: by (\ref{example:reinfPoissonProcess}), if $\sum_{i}\delta_{Y_i}$ is a Poisson random measure in $D[0,1]$ with intensity $t c (1-p) \mathbb{Q}$,  the process  $(\sum_iY_i(s) : s \in [0,1])$ has the same distribution as $(\hat{N}^{(t)}_s)_{s \in [0,1]}$. In particular, we have 
\begin{equation*}
    \int_0^t f(s) \dd \hat{N}_s =  \int_0^1 f(st) \dd \hat{N}_{st} \overset{\scr{L}}{=} \sum_i \int_0^1 f(st) \dd  Y_i(s).
\end{equation*}
Putting everything together, we deduce (\ref{identity:laplaceNRPPcompactos}) by making use of the Laplace formula for integrals with respect to Poisson random measures -- we invite the reader to compare (\ref{identity:laplaceNRPPcompactos}) with the identity   \eqref{equation:fddsNRLPs} for the finite-dimensional distributions of NRLPs -- and it remains to show that the Laplace functional of $\mathcal{L}\circ \exp^{-1}$ coincide with this expression. 

In this direction, recall the observation made in (\ref{definition:logtimeChangedYule}) and denote  by $\mbm{Z}$ the law of the standard Yule process $Z$. 
It follows that the law of  $\langle  \mathcal{L} \circ \exp^{-1} , 1_{ (0,t]  } f  \rangle$ can be expressed in terms of the Poisson random measure $\mathcal{M} := \sum_i \delta_{(u_i, Z^{(i)})}$ in $\mathbb{R}^+ \times D[0,1]$, with intensity $c (1-p)\dd t\otimes \mbm{Z}$, by considering the functional 
\begin{equation*}
    \sum_{i} \int_{(0,t]} f(s) \, \dd \left(  \indic{\{u_i \leq s \}} Z^{(i)}_{p( \ln(s) - \ln(u_i) )} \right),   
\end{equation*} where the integrals in the previous expression are respectively with respect to the Stieltjes  measure associated to the counting process $s \mapsto  \indic{\{u_i \leq s \}} Z^{(i)}_{p( \ln(s) - \ln(u_i) )}$. 
It now follows also by the exponential formula  that 
\begin{align} \label{cuentas:laplacereinforcedPP_parte2}
     -\log \esp{e^{ - \langle \mathcal{L} \circ \exp^{-1} ,  1_{\{\cdot \leq t \} } f  \rangle  }}{} 
    &= (1-p)c \int_{\mathbb{R}} \dd u \, \esp{ 1-\exp\bigg\{-\int_0^t f(s) \dd \left( \indic{\{u \leq s \}} Z_{p(\ln(s) - \ln(u))} \right)   \bigg\} }{}  \nonumber  \\
    &= (1-p)c \int_{0}^t \dd u \,  \esp{ 1-\exp\bigg\{-\int_0^t f(s) \dd \left( \indic{\{u \leq s \}} Z_{p(\ln(s) - \ln(u))} \right)   \bigg\} }{} \nonumber \\
    &=t  (1-p)c  \cdot \mbm{Z}^\bullet \left( 1-\exp\bigg\{\int_0^1 f(st) \dd \left( \indic{\{u \leq st \}} Z_{p(\ln(st) - \ln(u))} \right) \bigg\} \big| u \leq t\right), 
\end{align}
where we denoted in the last line by $\mbm{Z}^\bullet( \, \cdot \,  | u \leq t )$ the integral in  $\mathbb{R}^+ \times D[0,\infty)$ with respect to the probability measure 
\begin{equation*}
    \mbm{Z}^\bullet (\, \cdot \, | u \leq t) := \frac{\indic{\{ u \leq t \} }}{t} \dd u \, \mbm{Z}(  \dd Z  ).
\end{equation*}
Now,  we deduce by  Lemma \ref{lemma:restriccionesRepresentante} - (ii) in the Appendix that (\ref{cuentas:laplacereinforcedPP_parte2}) is precisely (\ref{identity:laplaceNRPPcompactos}).  
\end{proof}

Finally, for later use we  state  the following equivalent expression for the  Laplace functional associated to the random measure $\mathcal{L} \circ \exp^{-1}$.
\begin{lemma}
For any measurable $f : \mathbb{R}^+ \mapsto \mathbb{R}^+$,  we have 
\begin{equation} \label{identity:laplaceNRPPP}
    - \log \esp{ \exp \bigg\{ - \langle  \mathcal{L} \circ \exp^{-1} ,  f  \rangle  \bigg\} }{} =  (1-p)c \int_0^\infty \dd u \int_{\mathcal{M}_c} 1-e^{-\langle T_{\log(u)} \mu , f \circ \, \exp  \rangle } \mathbb{D}(\dd \mu ).  
\end{equation}
\end{lemma}

\begin{proof}
The proof follows from the equality $\langle  \mathcal{L} \circ \exp^{-1} ,  f  \rangle = \langle  \mathcal{L}  ,  f \circ \exp  \rangle$ and the identity: 
\begin{equation*} 
    - \log \esp{ \exp \left\{ - \langle \mathcal{L},h \rangle   \right\}  }{} 
    = (1-p) c \int_0^\infty \dd u \int_{\mathcal{M}_c} 1-e^{- \langle \mathcal{T}_{\log(u)}  \mu , h \rangle  } \mathbb{D}(\dd \mu ), 
\end{equation*}
holding for any  measurable $h:\mathbb{R}^+ \mapsto \mathbb{R}^+$.
The proof of the later is just a straightforward consequence of (\ref{definition:reinforcedPPExpScale}) and the exponential formula for Poisson random measures. 
\end{proof}
}

\begin{remark} 
Notice from   (\ref{definition:PPofNRPP_intro_intro}) that the reinforced Poisson  process with rate $c$    can be interpreted as a Yule-Simon process with immigration: this is, a process modelling the evolution of a population where new independent immigrants arrive according to a Poisson point process with intensity  $(1-p) c  \cdot \dd t$ and reproduce according to  a time changed Yule process, independent of the rest.
\end{remark}

\subsection{Construction of noise reinforced Poisson point processes by decoration} \label{subsection:jump_NRPPPdecorated}
This section is devoted to the  construction of \textit{noise reinforced Poisson point processes} and to establishing their first properties.  From here, we fix $p \in (0,1)$. \medskip \\
$\, \bullet$ \textit{Step 1: }Suppose first that $0 < \Lambda(\mathbb{R}) < \infty$.  
 With the  same notation of Section \ref{subsection:JumpsReinforcedPP},  denote by $\scr{E}$ a Poisson random measure  in $\mathbb{R}$ with intensity $ \Lambda(\mathbb{R})(1-p)e^t \dd t$ and consider the Poisson point process  $\Sigma_{u \in \scr{E}} \delta_{(u, x_u)}$ in $\mathbb{R} \times \mathbb{R}$ with intensity $(1-p) e^t \dd t\otimes \Lambda(\dd x)$. Now, for each $u \in \scr{E}$,  consider an independent copy $D_u$ of $D$ and set 
\begin{equation} \label{definition:expNRPPP}
    \mathcal{L}^{x}( \dd s ,\dd x) :=
    \sum_{u \in \scr{E}} \sum_{t \in D_u} \delta_{ (u + t , x_u)}.
\end{equation}
{This is just  the point process $\mathcal{L}$ from (\ref{definition:reinforcedPPExpScale}) with  $c := \Lambda(\mathbb{R})$},  marked by a collection of i.i.d. random variables with law $\Lambda(\dd x)/\Lambda(\mathbb{R})$. Formula (\ref{definition:expNRPPP}) defines a  random measure in $\mathbb{R} \times \mathbb{R}$ and  if we consider its push forward by $ (t,x) \mapsto (\exp(t) , x )$, that we denote as  $\hat{\mathcal{N}} := \mathcal{L}^{x} \circ (\exp , \text{Id})^{-1}$, we obtain the  measure in $\mathbb{R}^+ \times \mathbb{R}$ given by 
\begin{equation} \label{definition:casofinitoNRPPP}
    \hat{\mathcal{N}}( \dd s ,\dd x) :=  \sum_{u \in \scr{P} } \sum_{t \in D_u} \delta_{( ue^{t} ,x_u) }, 
\end{equation} 
where $\scr{P} := \scr{E} \circ \exp^{-1}$ is a Poisson point process in $\mathbb{R}_+$ with intensity $ \Lambda(\mathbb{R}) (1-p) \dd t$. We refer to the measure in the previous display as a NRPPP with (finite) characteristic measure $\Lambda$ and reinforcement parameter $p$.  \\
$\bullet$ \textit{Step 2: } If we no longer assume $\Lambda(\mathbb{R})<\infty$, we proceed by superposition. More precisely, let $(A_j)_{j \in \mathcal{I}}$ be a disjoint partition of   $\mathbb{R} \setminus \{ 0 \}$ such that $\Lambda(A_j) < \infty$.  Consider a collection of independent NRPPPs  $(\hat{\mathcal{N}}_j( \dd s , \dd x) : j \in \mathcal{I})$ with respective characteristic measures $(\Lambda(\, \cdot \cap A_j) : j \in \mathcal{I})$ constructed as in  (\ref{definition:casofinitoNRPPP}),  respectively in terms of:
\smallskip \\
\indent - independent Poisson random measures $\sum_{u \in \scr{P}_j} \delta_{(u,x_u)}$ with intensities $(1-p) \dd t\otimes \Lambda( \, \cdot \cap A_j)$. \\
\indent - independent collections $(D_u)_{u \in \scr{P}_j}$  of i.i.d. copies of $D$. \smallskip  \\
Finally, set  $\scr{P} := \sum_j \scr{P}_j$.  Now we are in position to introduce  NRPPPs with sigma-finite characteristic measures: 
\begin{definition}\label{proposition:reinforcedPPPbyMarking} \emph{(Noise Reinforced Poisson Point Process - NRPPP)}\\ 
The random measure $\hat{\mathcal{N}}( \dd s , \dd x) := \sum_{j \in \mathcal{I}} \hat{\mathcal{N}}_j(\dd s , \dd x )$ is called a  reinforced Poisson point process with reinforcement (or memory) parameter $p$ and characteristic measure  $\Lambda$. Moreover,  $\hat{\mathcal{N}}$  writes 
\begin{equation} \label{definition:decoratedNRPP}
   \hspace{-8mm} \hat{\mathcal{N}}( \dd s ,\dd x) =  \sum_{u \in \scr{P} } \sum_{t \in D_u} \delta_{( ue^{t} ,x_u)}. 
\end{equation}
\end{definition}
From the identity in the previous display and  recalling that the first element of $D$ is just $0$, the measure $\hat{\mathcal{N}}$ naturally decomposes as $\hat{\mathcal{N}} = \mathcal{N}' + {\mathcal{N}}''$, where  $\mathcal{N}'$ is a PPP with intensity $(1-p) \dd t \otimes \Lambda$. Moreover, the following properties readily follow from our construction:  
\begin{lemma} \label{proposition:propiedadesNRPPP}
 Let $\hat{\mathcal{N}}$ be a NRPPP with characteristic measure  $\Lambda$ and reinforcement parameter $p$. 
 \begin{enumerate}
     \item[\emph{(i)}] If $A \in \scr{B}(\mathbb{R})$,  the restriction $\indic{A}(x) \hat{\mathcal{N}}( \dd s,\dd x)$ is a NRPPP with characteristic measure $ \indic{A} \Lambda$ and  parameter $p$.
     \item[\emph{(ii)}]  If $A_1, A_2 \in \scr{B}(\mathbb{R})$ are disjoints, then $\indic{A_1}(x)\hat{\mathcal{N}}( \dd s,\dd x)$, $\indic{A_2}(x) \hat{\mathcal{N}}( \dd s,\dd x)$  are independent.
     \item[\emph{(iii)}]  If $\hat{\mathcal{N}}_1$, $\hat{\mathcal{N}}_2$ are independent NRPPPs with respective characteristic measures $\Lambda_1$, $\Lambda_2$ and same reinforcement parameter $p$, then $\hat{\mathcal{N}}_1+\hat{\mathcal{N}}_2$ is a NRPPP with characteristic measure $\Lambda_1+\Lambda_2$ and  parameter $p$. 
 \end{enumerate}
\end{lemma}

\noindent The following lemma shows that the intensity measure of a NRPPP with characteristic measure $\Lambda$ and parameter $p$, coincides with the one of a PPP with characteristic measure $\Lambda$.  

\begin{lemma} \label{proposition:intensidadNRPPP}
 Let $\hat{\mathcal{N}}$ be a NRPPP with characteristic measure $\Lambda$ and reinforcement parameter $p$. For any measurable  $f:\mathbb{R}^+ \times \mathbb{R} \mapsto \mathbb{R}^+$, we have   $\mathbb{E}[ \langle  f ,   \hat{\mathcal{N}} \rangle  ]
      =  \int_0^\infty \dd s \int_\mathbb{R} \Lambda(\dd x) f(s , x).$
\end{lemma}
\begin{proof}
Suppose first that $\Lambda(\mathbb{R})< \infty$ and  recall from (\ref{definition:logtimeChangedYule}) that for  fixed $u \in \mathbb{R}_+$, the atoms of the  measure $\sum_{t \in D} \delta_{ue^{t}}$ are precisely the jumps of the time-changed Yule process (\ref{definition:logtimeChangedYule}). 
Hence, if $\sum_{u \in \scr{P}}\delta_{(u, x_u)}$ is a Poisson random measure with intensity $(1-p) \dd t \otimes \Lambda(\dd x)$ and $(Z^{(u)})_{u \in \scr{P}}$ is an independent collection with law $\mbm{Z}$,  it is then clear from our construction in the finite case (\ref{definition:casofinitoNRPPP}) that we can write 
\begin{equation*}
    \esp{\hat{\mathcal{N}} (0,T]\times A}{} 
    = \esp{\sum_{u \in \scr{P}}  \indic{\{u \leq T \}} Z^{(u)}_{\{ p (\ln(T) - \ln(u)) \}} \indic{\{x_u \in A \}} }{},  
\end{equation*}
where the random measure $\sum_{u \in \scr{P}} \delta_{(u,x_u, Z^{(u)})}$ is Poisson with intensity $(1-p)\dd t \otimes \Lambda \otimes \mbm{Z}$. Consequently, recalling that $E[Z_t] = e^t$, by Campbell's formula we obtain that 
\begin{equation*}
    \esp{\hat{\mathcal{N}} (0,T]\times A}{}  = T \cdot \Lambda(A) , 
\end{equation*}
 and we deduce that the intensity measure of $\hat{\mathcal{N}}$ is given by $\dd t \otimes \Lambda$. When $\Lambda(\mathbb{R})= \infty$, we can proceed by superposition.
\end{proof}

We now identify the law of $\hat{ \mathcal{N} }$ by computing its exponential functionals. 

\begin{proposition} \label{proposition:exponentialFormulasJumpProcess} Let $\hat{\mathcal{N}}$ be a NRPPP  with characteristic measure $\Lambda$ and reinforcement parameter $p$. 
\begin{enumerate}
    \item[\emph{(i)}] For every measurable  $f:\mathbb{R}^+ \times \mathbb{R} \mapsto \mathbb{R}^+$ and $t \geq 0$ we have  
\begin{equation} \label{formula:exponencialLaplace} 
    \hspace{-3mm}  \esp{\exp \bigg\{ - \int_{(0,t] \times \mathbb{R} } f(s , x ) \hat{\mathcal{N}}( \dd s ,\dd x)  \bigg\} }{}
    = 
    \exp \bigg\{- t (1-p)
    \int_{\mathbb{R}} \Lambda(\dd x) \esp{   1-\exp\left(- \int_0^1 f(st,x)\dd Y(s) \right)      }{} 
    \bigg\}.
\end{equation}
    \item[\emph{(ii)}] If we no longer assume that $f$ is non-negative, under the condition $\int_0^t \dd s \int_\mathbb{R} \Lambda(\dd x) |f(s,x)| < \infty$ we have:
\begin{equation}\label{formula:exponencialFourier}
    \hspace{-3mm} \esp{\exp \bigg\{ i \int_{(0,t] \times \mathbb{R}} f(s , x ) \hat{\mathcal{N}}( \dd s ,\dd x) \bigg\} }{}
    = 
    \exp \bigg\{ t (1-p)
    \int_{\mathbb{R}} \Lambda(\dd x) \esp{   \exp\left(i \int_0^1 f(st,x)\dd Y(s) \right) -1 }{} 
    \bigg\}.
\end{equation}
\end{enumerate}
\end{proposition}
\begin{proof}
(i) We start by considering the finite case $\Lambda(\mathbb{R})< \infty$ and we make use of the notations introduced in \eqref{definition:expNRPPP}; for instance, recall that $\langle \hat{\mathcal{N}} , f  \rangle = \langle \mathcal{L}^{x} , f \circ (\exp , \rm{ \text{Id}}) \rangle $.  We start showing the result for $f$ of the form $f(s,x)= h(s)g(x)$, for non-negatives $h:\mathbb{R}^+\mapsto \mathbb{R}^+$ and $g:\mathbb{R} \mapsto \mathbb{R}^+$, in which case we can write
\begin{align}
    \langle \mathcal{L}^{x} , (h \circ \exp) g  \rangle 
    = \sum_{u \in \scr{E}} \sum_{t \in D_u} h\circ \exp( u + t) g(x_u) 
    = \sum_{u \in \scr{E}}  g(x_u) \langle  \mathcal{T}_{u} D_u ,  h \circ \exp \rangle.  \label{equation:integralL}
\end{align}
Now,  we deduce from the formula for the Laplace transform of Poisson integrals and a change of variable that
\begin{equation*}
    - \log \esp{ e^{- \langle \mathcal{L}^{x},(h \circ \exp) g \rangle } }{} = (1-p)   \int_{\mathbb{R}}\Lambda(\dd x) \int_{\mathbb{R}^+} \dd u \int_{\mathcal{M}_c}  1-e^{-g(x) \langle  \mathcal{T}_{\log(u)}\mu , h \circ \exp \rangle} \mathbb{D}(\dd \mu ).
\end{equation*}
If we now replace $h$ by $h \indic{\{ \cdot \leq t \}}$, making use of the  equivalent identities (\ref{identity:laplaceNRPPP}) and  (\ref{identity:laplaceNRPPcompactos}),  we obtain that the previous display writes: 
\begin{equation*}
   t \cdot (1-p)   \int_{\mathbb{R}}\Lambda(\dd x) \esp{ 1- e^{-g(x) \int_0^1 h(st)\dd Y(s)} }{}, 
\end{equation*}
proving the claim. Now, still under the hypothesis $\Lambda(\mathbb{R}^+) < \infty$, fix arbitrary $\alpha_{i,j} \in \mathbb{R}^+$,   consider  $0 = t_1 <  \dots < t_{k+1} < t$ as well as disjoint subsets  $A_1,  \dots , A_n$  of $\mathbb{R}^+$. Further,  suppose that $f$ is of the form 
\begin{equation} \label{cuentas:apeoximation}
    f(s,x) := \sum_{j=1}^n \sum_{i=1}^k \alpha_{i,j} \indic{(t_i, t_{i+1}]}(s) \indic{A_j}(x) \quad \quad \text{ and write } \quad  g_j(s,x) := \sum_{i=1}^k \alpha_{i,j} \indic{(t_i, t_{i+1}]}(s) \indic{A_j}(x). 
\end{equation}
Recall from Lemma \ref{proposition:propiedadesNRPPP}  that  the restrictions $\indic{A_1}\hat{\mathcal{N}}, \dots , \indic{A_n}\hat{\mathcal{N}}$ are independent NRPPPs with respective characteristic measures $\Lambda( \cdot \cap A_i)$. By independence and applying the previous case to each $g_j$, we deduce that 
\begin{align*}
    \esp{\exp \bigg\{ - \langle \hat{ \mathcal{N}} , \indic{\{ \cdot \leq t \}} f \rangle   \bigg\} }{}
    &= \prod_{j=1}^n \esp{\exp \bigg\{ - \langle \hat{ \mathcal{N}} , \indic{\{ \cdot \leq t \}} g_j \rangle   \bigg\} }{} \\
    &= \prod_{j=1}^n \exp \bigg\{ t(1-p) \Lambda(A_j) \esp{1- \exp\bigg\{ - \int_0^1 \sum_{i=1}^k \alpha_{i,j} \indic{(t_i < st \leq  t_{i+1}]}(s)\dd Y(s)  \bigg\} }{}  \bigg\} \\
    &=  \exp \bigg\{ t(1-p) \int_\mathbb{R} \Lambda(\dd x) \esp{1- \exp\bigg\{ - \int_0^1 f(st,x)\dd Y(s)  \bigg\} }{}  \bigg\},  
\end{align*}
and once again we recover (\ref{formula:exponencialLaplace}). Finally, if $f$ is non-negative and bounded with support in $[0,t] \times \mathbb{R}$,  it can be approximated by a bounded sequence of  functions $(f_n)$ of the form (\ref{cuentas:apeoximation}), the convergence holding $\dd t \Lambda(\dd x)$ a.e. For each $n$,  we have 
\begin{equation}\label{cuentas:formulaexpo}
    \esp{\exp \bigg\{ - \langle \hat{ \mathcal{N}} ,  f_n \rangle   \bigg\} }{}
    = 
    \exp \bigg\{ t(1-p) \int_\mathbb{R} \Lambda(\dd x) \esp{1- \exp\bigg\{ - \int_0^1 f_n(st,x)\dd Y(s)  \bigg\} }{}  \bigg\},  
\end{equation}
and   by Lipschitz-continuity, it follows that 
\begin{align*}
\esp{ | \exp \big\{ - \langle \hat{ \mathcal{N}} ,  f \rangle   \big\} 
- 
\exp \big\{ - \langle \hat{ \mathcal{N}} ,  f_n \rangle   \big\}|
}{} 
&\leq  \esp{  \int_{[0,t] \times \mathbb{R} } |f(s,x) - f_n(s,x) |\hat{\mathcal{N}}( \dd s ,\dd x) }{} \\
&=  \int_0^t\dd s \int_\mathbb{R} \Lambda(\dd x) | f(s,x) - f_n(s,x)  | \rightarrow 0 \,  \text{ as } n \uparrow \infty.  
\end{align*}  In the last equality we used Lemma  \ref{proposition:intensidadNRPPP}. From the same arguments we also obtain that 
\begin{align*}
   \int \Lambda(\dd x)  \esp{  | e^{-\int_0^1 f(st , x)\dd Y(s) }  - e^{-\int_0^1 f_n(st , x)\dd Y(s) }  | }{}  
    &\leq \int \Lambda(\dd x) \esp{  \int_0^1 |f(st,x) - f_n(st,x) | \dd Y(s) }{} \\ 
    &= (1-p)^{-1} \int_0^1\dd s \int_\mathbb{R} \Lambda(\dd x) | f(st,x) - f_n(st,x)  | \rightarrow 0 \,  
\end{align*}
as  $n \uparrow \infty$. Now,  we deduce from taking the limit as $n \uparrow \infty$ in (\ref{cuentas:formulaexpo}) that  the identity (\ref{formula:exponencialLaplace}) also holds for $f$. 
\par 
If we suppose that $\Lambda(\mathbb{R}) = \infty$, the proof follows by superposition.  Namely,  with the same notation used  for constructing   (\ref{definition:decoratedNRPP}),     the random measures $(\hat{\mathcal{N}} _j)_{j \in \mathcal{I}}$ are   independent NRPPPs with respective finite characteristic measures  $\Lambda(\, \cdot \, \cap A_j)$   and by definition we have $\hat{\mathcal{N}} = \sum_j  \hat{\mathcal{N}}_j$. From  the formula for the Laplace transform we just proved in the finite case and independence it follows that 
\begin{align*}
    \esp{ e^{- \langle \hat{\mathcal{N}},f \indic{\{ \cdot \leq t \}} \rangle } }{} 
    &= \prod_{j \in \mathcal{I}} \esp{ e^{- \langle \hat{\mathcal{N}}_j,f \indic{\{ \cdot \leq t \}} \rangle } }{} \\
    &= \prod_{j \in \mathcal{I}} \exp \bigg\{ - t \cdot (1-p)
    \int_{{A}_j} {\Lambda(\dd x)} \esp{\left( 1- e^{-\int_0^1 f(st,x)\dd Y(s)} \right)  }{} \bigg\},
\end{align*}
proving (i).  Now  (ii) follows from similar arguments, by making use  of the formula for the characteristic function for Poissonian integrals and the inequality $|e^{ib} - e^{ia}| \leq |a-b|$ for $a,b \in \mathbb{R}$,  we omit the details. 
\end{proof}

The following result is the reinforced analogue of the well known characterisation result for Poisson point processes. The arguments we use are similar to the ones in the non-reinforced case.    
\begin{proposition} \label{caractherisation:nrppp}
Let   $\hat{\mathcal{N}}$ be a point process in $\mathbb{R}^+ \times \mathbb{R}$ and for any Borelian $A \subset \mathbb{R}$,  set  
\begin{equation*}
    \hat{N}_A(t) := \hat{\mathcal{N}}([0,t] \times A), \quad t \geq 0.
\end{equation*}
Then, $\hat{\mathcal{N}}$ is a noise reinforced Poisson point process with characteristic measure $\Lambda$ and parameter $p$ if and only if the two following conditions are satisfied:
\begin{enumerate}
    \item[\emph{(i)}] For any Borelian $A$ with $\Lambda(A) < \infty$, the process $\hat{N}_A$ is a noise reinforced Poisson process with rate $\Lambda(A)$ and reinforcement parameter $p$.
    \item[\emph{(ii)}]  If $A_1, \dots A_k$ are disjoint Borelians with $\Lambda(A_i)< \infty$ for all $i \in \{1 , \dots, k\}$,  the processes $\hat{N}_{A_1}, \dots , \hat{N}_{A_k}$ are independent. 
\end{enumerate}
\end{proposition}
\begin{proof}
First, let us prove that NRPPP do satisfy (i) and (ii). Remark that (ii) is just a consequence of Lemma \ref{proposition:propiedadesNRPPP} - (ii) and we focus on (i).   Fix $A$ as in (i) as well as times  $0 <  t_1 < \dots < t_k \leq  t$, and we proceed by computing the characteristic function  of the finite dimensional distributions of $\hat{N}_A$. This can now be done by  considering the function  $f(s,x) := \sum_{i=1}^k \lambda_i \indic{\{ s \leq t_i \}} \indic{A}(x)$ and applying the exponential formula (\ref{formula:exponencialFourier}), yielding  
\begin{align*}  
      \esp{ \exp \left\{ i \sum_{i=1}^k \lambda_i \hat{N}_{A}(t_i) \right\}  }{} 
      &= \exp \bigg\{ t (1-p)
     \int_{\mathbb{R}} \Lambda(\dd x) \esp{   \exp\left(i \sum_{i=1}^k \int_0^1  \lambda_i \indic{\{ s t \leq t_i \}} \indic{A}(x)\dd Y(s) \right) -1 }{} 
     \bigg\}.\\ 
     &= \exp \left\{ t (1-p) \Lambda(A)   \esp{\left( \exp \bigg\{ {i \sum_{i=1}^k \lambda_i Y( t_i/t)} \bigg\} -1 \right) }{}   \right\}.  
\end{align*}
Recalling the identity  (\ref{definition:noiseReinfPoisson}), we deduce that $\hat{N}_A$ is a noise reinforced Poisson process with rate $\Lambda(A)$ and reinforcement $p$. 

Now, we argue that if $\hat{\mathcal{N}}$ is a random measure satisfying (i) and (ii), then it is a NRPPP. We will establish this claim by showing that $\hat{\mathcal{N}}$ satisfies the exponential formula (\ref{formula:exponencialFourier}). First, observe that (i) implies that  {$\mathbb{E} [ \hat{N}_A(t)]  = t \Lambda(A)$}, for example by making use of  Lemma  \ref{proposition:intensidadNRPPP} and the fact that if $\hat{\mathcal{M}}$  is a NRPPP with characteristic measure $\Lambda$ and parameter $p$, then $(\hat{\mathcal{M}} ([0,t]\times A) :  t \geq 0)$ is a reinforced Poisson process with rate $\Lambda(A)$ and parameter $p$. We deduce by a monotone class argument that $\hat{\mathcal{N}}$ satisfies, for any measurable $f: \mathbb{R}^+ \times \mathbb{R} \mapsto \mathbb{R}^+$, the identity: 
\begin{equation} \label{cuentas:intensidadCaracterizacion}
    \esp{  \int_{[0,t] \times \mathbb{R}} f(s,x) \hat{\mathcal{N}}( \dd s,\dd x)  }{} = \int_0^t\dd s \int \Lambda(\dd x) f(s,x).
\end{equation}
Still for $A$ as in (i) and for an arbitrary collection of times $0 = t_1 < t_2 < \dots  < t_{k+1} < t$,  we  set
\begin{equation}
    g(s,x) := \sum_{i=1}^k \alpha_i \indic{ (t_i, t_{i+1}] } (s) \indic{A}(x).
\end{equation}
Since by hypothesis $(\hat{N}_A(t))_{t \in \mathbb{R}^+}$ is a NRPP with rate  $\Lambda(A)$, by  the formula (\ref{definition:noiseReinfPoisson}) for the characteristic function of the finite dimensional distributions of reinforced Poisson processes, we obtain that
\begin{align*}
  \esp{\exp \bigg\{ i \langle \hat{ \mathcal{N}} , \indic{\{ \cdot \leq t \}} g \rangle   \bigg\} }{} 
&=  \esp{\exp \bigg\{ i  \sum_{i=1}^k \alpha_i( N_A(t_{i+1})-N_A(t_{i}))    \bigg\} }{} \\
&= \exp \bigg\{ t (1-p)
      \Lambda(A) \esp{   \exp\left(i \sum_{i=1}^k \alpha_i(   Y(t_{i+1}/t) - Y(t_{i}/t) ) \right) -1 }{} 
     \bigg\}\\ 
&= \exp \bigg\{ t (1-p)
     \int_{\mathbb{R}} \Lambda(\dd x) \esp{   \exp\left(i  \int_0^1 \sum_{i=1}^k  \alpha_i \indic{\{ t_i < s t \leq t_{i+1} \}} \indic{A}(x)\dd Y(s) \right) -1 }{} 
     \bigg\}.  
\end{align*}
Remark that this is precisely the identity (ii) of  Proposition \ref{proposition:exponentialFormulasJumpProcess}  for our choice of $g$. Making use of  the independence hypothesis of  $\hat{N}_{A_1}, \dots \hat{N}_{A_k}$ for disjoints $A_1, \dots , A_k$ with $\Lambda(A_i) < \infty$,  we can also show that 
the identity holds for $f$ as in (\ref{cuentas:apeoximation}) for such collection of sets. 
Now, if $f$ is non-negative, bounded and supported on $[0,t] \times A$ with $\Lambda(A) < \infty$, making use of  (\ref{cuentas:intensidadCaracterizacion}),  we can  proceed  as in (\ref{cuentas:formulaexpo}) for the proof of  Proposition \ref{proposition:exponentialFormulasJumpProcess}, approximating $f$ by a bounded sequence of the form (\ref{cuentas:apeoximation}),  and show that  the exponential formula (\ref{formula:exponencialFourier}) still holds.  The general case follows  by sigma finiteness of $\Lambda$ and we deduce that $\hat{\mathcal{N}}$  is a NRPPP with the desired parameters.  
\end{proof}

\subsection{Proof of Theorem \ref{lemma:PreFormulaCompensacion} and compensator of the jump measure}\label{subection:jump_LevyIto-Part1} 

Let us now establish Theorem \ref{lemma:PreFormulaCompensacion}. Remark that paired with  Proposition \ref{caractherisation:nrppp}, it entails that the role of the counting process of jumps $\Delta \hat{\xi}_s \in A$ for fixed   $A \in \mathscr{B}(\mathbb{R})$ is played precisely by noise-reinforced Poisson processes,  in analogy with the non-reinforced setting.

\begin{proof}[Proof of Theorem \ref{lemma:PreFormulaCompensacion}]
The result will follow as soon as we establish (i) and (ii) of Proposition \ref{caractherisation:nrppp} for 
\begin{equation} \label{definition:countingProcessJumps}
    \hat{N}_A(t) := \# \big\{  (s, \Delta \hat{\xi}_s) \in [0,t]\times A \big\}, \quad \quad t \geq 0, 
\end{equation} 
where $A$ is an arbitrary Borelian satisfying $\Lambda(A) < \infty$. By  the identity in distribution (\ref{identity:extension}), we can restrict our arguments to the  unit interval  and hence we can make use of the explicit construction of NRLPs in $[0,1]$ that we recalled in Section \ref{subsection:preliminaries,examplesNRLPs}, in terms of Yule-Simon series.

Denote by ${\mathcal{M}} := \sum_{i} \delta_{(x_i, Y_i)}$ the  Poisson random measure  with intensity $(1-p) \Lambda \otimes \mathbb{Q}$ and recall the discussion following Theorem \ref{regularityNRLP}. If  $(x_i, Y_i)$ is an atom of ${\mathcal{M}}$,  then at time $U_i = \inf\{t \geq 0 : Y_i(t) = 1 \}$, the process $\hat{\xi}$  performs the jump $x_i$ for the first time, i.e.  $\Delta\hat{\xi}_{U_i}= x_i$  and this precise jump $x_i$ is repeated in the interval $[0,1]$ at each jump time of $Y_i$.  It follows that for any $f: \mathbb{R}\mapsto \mathbb{R}^+$ we  have: 
\begin{equation} \label{equation:saltosFiniteVar_2}
     \sum_{s \leq t} f(\Delta \hat{\xi}_s)  = \sum_{i}  f(x_i) Y_i(t), 
\end{equation}
and in particular, we get:
\begin{equation*}
     \hat{N}_A(t)  = \sum_{i}  \indic{\{ x_i \in A \}} Y_i(t).  
\end{equation*}
Hence, by the independence property of Poisson random measures, the processes $\hat{N}_{A_1}, \dots , \hat{N}_{A_n}$ are independent as soon as $A_i\cap A_j = \emptyset$ for all $i \neq j$.  Now, if we fix  $\lambda_1, \dots , \lambda_k \in \mathbb{R}$, $0 \leq t_1 < \dots < t_k \leq 1$, we deduce from the formula for the characteristic function for Poisson integrals the equality:  
\begin{equation*}
 \esp{ \exp \bigg\{ i \sum_{j=1}^k \lambda_j \hat{N}_A({t_j}) \bigg\} }{}  
  = 
    \exp \bigg\{   (1-p) \Lambda(A) 
    \esp{   \exp \bigg\{i \sum_{j=1}^k \lambda_j Y(t_j) \bigg\} -1 }{}
    \bigg\}.
\end{equation*}
Comparing with $(\ref{definition:noiseReinfPoisson})$, we get that the right-hand side in the previous display is precisely the characteristic function of the finite dimensional distributions at times $t_1, \dots , t_k$ of a reinforced Poisson process with rate  $\Lambda(A)$ and parameter $p$.  
\end{proof}

 Recalling the explicit construction of NRPPPs from  Definition  \ref{proposition:reinforcedPPPbyMarking},  we stress that Theorem \ref{lemma:PreFormulaCompensacion} formalises the idea that the jumps of NRLPs are jumps that are repeated through time, similarly to the dynamics of noise reinforced random walks -- we refer to the beginning of Section \ref{section:jointlaw} for a brief introduction to the later. Our terminology and  notation  for the reinforced measure  $\hat{\mu}$ can now be  justified by the following: if $\mu$ is the jump measure of $\xi$, the counting process  $(\mu([0,t] \times A) : t \geq 0 )$ is a Poisson process with rate $\Lambda(A)$ while  $(\hat{\mu}([0,t] \times A) : t \geq 0 )$ is a reinforced Poisson process with rate $\Lambda(A)$. Said otherwise,  the following   identity holds in distribution: 
\begin{equation}
   \hat{\mu}([0,\cdot \, ] \times A) \overset{\scr{L}}{=}  \reallywidehat{ {\mu}([0,\cdot \,] \times A)}.  
\end{equation}

\par Now that the key result of the section has been established, we continue our study of the jump process of NRLPs. In this direction,  we start by briefly recalling notions of semi-martingale theory that will be needed. Let $X$ be a semimartingale defined on a probability space $(\Omega , \scr{F}, (\scr{F}_t), \mathbb{P}).$ Its jump measure $\mu_X$ is an integer valued random measure on $(\mathbb{R}^+ \times \mathbb{R}, \scr{B}(\mathbb{R}^+) \otimes \scr{B}(\mathbb{R}))$, in the sense of   \cite[Chapter II-1.13]{JacodShiryaev_LimitTheorems}.
Denote  the predictable   sigma-field  on $\Omega\times \mathbb{R}^+$ by  $\pred$. If $H$ is a  $\pred \otimes \scr{B}(\mathbb{R}^+)$-measurable function, we simply  write $H * \mu_X$ for the process defined at each $t \in \mathbb{R}^+$ as 
\begin{equation} \label{equation:integrationMeasuresJacod}
    (H * \mu_X)_t(\omega) :=  \int_{(0,t] \times \mathbb{R}} \mu_X(\omega ; \dd s, \dd x) H_s(\omega;  x),  \quad \quad \text{ if } \quad \int_{(0,t] \times \mathbb{R}}  \mu_X (\omega ; \dd s, \dd x) |H_s(\omega;  x) | < \infty,
\end{equation}
and $\infty$ otherwise. Both notations for the integral will be used indifferently. Further, we denote by $\scr{A}^+$ the class of increasing, adapted rcll finite-variation processes $(A_t)$, with $A_0 = 0$ such that $\esp{A_\infty}{} < \infty$, and by $\scr{A}_{loc}^+$ its localisation class. The jump measure $\mu_X$ posses a predictable compensator, this is, a random measure $\mu_X^{\tt{p}}$ on $(\mathbb{R}^+ \times \mathbb{R}, \scr{B}(\mathbb{R}^+) \otimes \scr{B}(\mathbb{R}))$ unique up to a $\mathbb{P}$-null set, characterised  by being the unique predictable random measure (in the sense of \cite[Chapter II-1.6]{JacodShiryaev_LimitTheorems}) satisfying that for any non-negative $H \in \pred \otimes \scr{B}(\mathbb{R})$,  the equality
\begin{equation*}
    \esp{(H * \mu_X)_\infty }{} = \esp{(H * \mu_X^{\tt{p}})_\infty }{}
\end{equation*}
holds. Equivalently, for any $H \in \pred \otimes \scr{B}(\mathbb{R})$ such that $|H|*\mu_X \in \scr{A}_{loc}^+$, the process $|H|*\mu_X^{\tt{p}}$   belongs to $\scr{A}_{loc}^+$ and $H * \mu_X^{\tt{p}}$ is the predictable compensator of $H*\mu_X$. Said otherwise,  $H * \mu_X - H * \mu_X^{\tt{p}} $
is a local martingale.\par 
 Recall that by  Proposition \ref{proposition:laMartingala}, the process $\hat{\xi}$ is a semimartingale. Hence,  we can consider  $\hat{\mu}^{\tt{p}}$, the predictable compensator of its jump measure $\hat{\mu}$, and our purpose is to identify explicitly   $\hat{\mu}^{\tt{p}}$.  In contrast, it might be worth mentioning that if $\xi$ is a Lévy process with Lévy measure  $\Lambda$, the compensator of its jump measure $\mu$ is just the deterministic measure $\mu^{\tt{p}} =\dd t\otimes \Lambda(\dd x)$. The first step consists in observing the following: 
\begin{lemma}
 Let $A \in \scr{B}(\mathbb{R})$ be a Borel set  that doesn't intersect some open neighbourhood of the origin. If we denote by $(\scr{F}_t^A)$ the natural filtration of $\hat{N}_A$, then the process  $M_A = (M_A(t))_{t \in \mathbb{R}^+}$ defined as $M_A(0) = 0$ and
\begin{equation*}
    M_A(t) = t^{-p} \left( \hat{N}_A(t) - t \Lambda(A) \right),  \quad \quad t \geq 0,
\end{equation*}
is a finite variation $(\scr{F}^A_t)$-martingale.
\end{lemma}

Remark that this is just a special case of Proposition \ref{proposition:laMartingala} for a Lévy measure of the form $ \Lambda(A) \delta_{1}$ with $q = 0$.  Now we can state:  

\begin{toexclude} 
In this direction, since $\esp{Y(r)}{} = r (1-p)^{-1}$, if we denote $\Phi^{(c)} (u) = e^{iu} -1 -iu$ we can write  
\begin{align*}
    \esp{ \exp \left\{i\sum_{i=1}^k \lambda_i \hat{N}_A^{(c)}(t_i) \right\} }{} 
    &= 
    \exp \bigg\{ t \cdot (1-p) \Lambda(B) 
    \esp{   \exp \bigg\{i \sum_{j=1}^k \lambda_j Y(t_j/t) \bigg\} -1 - i\sum_{j=1}^k \lambda_j Y(t_j/t)  }{}
    \bigg\} \\
    &= 
    \exp \bigg\{ t \cdot C 
    \esp{ \Phi^{(c)} \left( \sum_{j=1}^k \lambda_j Y(t_j/t)\right)  }{}
    \bigg\}\\
    &=: F(\lambda_1,\dots , \lambda_k, ) 
\end{align*}
for $C = (1-p) \Lambda(B)$. Now, the left hand side of (\ref{equation:martingaleCountingProcEq1}) is 
\begin{align*}
   & -i t_{k+1}^{-p}  \left. {\frac{\partial}{\partial \lambda_{k+1}}} \right|_{\lambda_{k+1} = 0} F(\lambda_1, \dots , \lambda_k ,\lambda_{k+1}) 
   \\
   & \hspace{15mm} = 
    -i \exp \bigg\{ t \cdot C 
    \esp{ \Phi^{(c)} \left( \sum_{j=1}^k \lambda_j Y(t_j/t)\right)  }{}
    \bigg\}
    \esp{ \Phi^{(c)'} \left( \sum_{j=1}^k \lambda_j Y(t_j/t)\right) Y(t_{k+1}/t) }{} t_{k+1}^{-p}  t \cdot C
\end{align*}
while the right hand side of (\ref{equation:martingaleCountingProcEq1}) can be written as 
\begin{align*}
    & -i t_k^{-p}   \frac{\partial}{\partial \lambda_{k} }  F(\lambda_1, \dots , \lambda_k) \\
    &\hspace{15mm} = 
   -i \exp \bigg\{ t \cdot C 
    \esp{ \Phi^{(c)} \left( \sum_{j=1}^k \lambda_j Y(t_j/t)\right)  }{}
    \bigg\}
    \esp{ \Phi^{(c)'} \left( \sum_{j=1}^k \lambda_j Y(t_j/t)\right) Y(t_{k}/t) }{} t_k^{-p} t \cdot C
\end{align*}
and it remains to show the equality 
\begin{equation*}
    \esp{ \Phi^{(c)'} \left( \sum_{j=1}^k \lambda_j Y(t_j/t)\right) Y(t_{k+1}/t) }{} t_{k+1}^{-p}  
    =\esp{ \Phi^{(c)'} \left( \sum_{j=1}^k \lambda_j Y(t_j/t)\right) Y(t_{k}/t) }{} t_k^{-p}. 
\end{equation*}
Recalling that  $\Phi^{(c)'}(0) = 0$ and $Y$ is increasing, the result follows by the same arguments as in (\ref{equation:finalpruebaMartingala}). 
\end{toexclude}

\begin{proposition} \label{theorem:compensationFormula} \emph{(Compensation formula)}\\
Denote by $(\scr{F}_t)$ the natural filtration of  $\hat{\xi}$ and by ${\hat{\mu}}$ its jump measure. The predictable compensator ${\hat{\mu}}^{\tt{p}}$ of ${\hat{\mu}}$ is given by 
\begin{equation}\label{definition:predCompensator}
    \hat{\mu}^{\tt{p}}(\omega;\dd t, \dd x) = (1-p)\dd t\otimes \Lambda(\dd x) +  p\frac{\dd t}{t}  \scr{E}_t (\omega; \dd x), 
\end{equation}
where $\scr{E}_t (\dd x) = \sum_{s < t} \delta_{\Delta \hat{\xi}_s}(\dd x) $ is the empirical measure of jumps that occurred strictly before time $t$. 
\end{proposition}
Consequently, for any predictable process $H \in \pred\otimes \scr{B}(  \mathbb{R})$  such that $|H|*\hat{\mu} \in \scr{A}_{loc}^+$, we have  $|H|*\hat{\mu}^{\tt{p}} \in \scr{A}_{loc}^+$ and the following process is a local martingale:  
\begin{equation} \label{formula:compensationFormula}
  M_t =  \sum_{s \leq t} H_s ( \cdot \, ,  \Delta \hat{\xi}_s  ) - (1-p) \int_0^t\dd s \int_{\mathbb{R}}  \Lambda(\dd x) H_s(\cdot \, ,x)    - \int_0^t \sum_{r <  s} H_{ s } ( \cdot \, ,\Delta \hat{\xi}_r )  \frac{p}{s}\dd s, \quad t \geq 0. 
\end{equation}
The first compensating term appearing in (\ref{formula:compensationFormula}) is compensating innovations, i.e.  atoms appearing for the first time, while the second one should be interpreted as the compensator of the memory part of $\hat{\mu}$. Notice that Proposition \ref{theorem:compensationFormula} holds if $p=0$. Indeed, in that case  $\hat{\xi}$ is a Lévy process and its jump process $\mu$ is the Poisson point process (\ref{ppp-Levy}). The compensator (\ref{definition:predCompensator}) is just the deterministic compensator   $ \dd t \otimes \Lambda(\dd x)$ for the  Poisson point processes  with  characteristic measure $\Lambda$ and in (\ref{formula:compensationFormula}) we recover the celebrated  compensation formula, see e.g.   \cite[Chapter 1]{Bertoin_LevyProcessesBook}. Remark that since the intensity of both $\mu$ and $\hat{\mu}$ is $\dd t\otimes  \Lambda(\dd x)$, we have, for both $X$ a Lévy process and its associated NRLP, the equality $\esp{ \sum_{s \leq t} f(s,\Delta X_s) }{} = \int_0^t\dd s \int_\mathbb{R}\Lambda(\dd x) f(s,x)$ for any $f:\mathbb{R}\times \mathbb{R}\mapsto \mathbb{R}^+$.  When $X := \xi$, by the compensation formula, this identity  holds also if we replace $f$ by a non-negative predictable process $H \in \pred \otimes \scr{B}(\mathbb{R})$, viz. 
\begin{equation} \label{formula:masterFormula}
    \esp{\sum_{s \leq t} H_s( \cdot , \Delta \xi_s) }{} = \esp{\int_0^t\dd s  \int_{\mathbb{R}} \Lambda(\dd x) H_s(\cdot , x) }{}. 
\end{equation}
However, we point out that if we replace in (\ref{formula:masterFormula}) the Lévy process by its reinforced version $\hat{\xi}$,  the identity no longer holds. Indeed, if such formula was satisfied,  the exact same proof for the exponential formula of PPPs of XII-1.12 in \cite{RevuzYor} would hold in our reinforced setting, and since random measures are characterised by their Laplace functional, this would lead us to the conclusion that the law of $\hat{\mu}$ coincides with the law of $\mu$. 

\begin{toexclude}
As a simple application of (ii) bringing us back to the original construction in terms  of series of Yule-Simon processes in \cite{BertoinNRLP}, pick $f(x) = x \indic{\{a \leq |x| < b\}}$ for some real numbers $0<a<b$. Then, $\esp{ \sum_{s \leq t} f(\Delta \hat{\xi}_ s ) }{} = t \int_{\{a \leq |x| < b\}} x \Lambda(\dd x)$ in accord with Section 2.2 of \cite{BertoinNRLP} since with the notation $(\Sigma_{a,b}(t))_{t \in [0,1]}$ used in \cite{BertoinNRLP}  for the NRLP of characteristics $(0,0,\indic{\{a \leq  |x| < b\}})\Lambda,p)$, we have   $\Sigma_{a,b}(t) = \sum_{s \leq t} \Delta \hat{\xi}_s \indic{\{a \leq  |\Delta \hat{\xi}_s | < b\}}$. 
\end{toexclude}
\begin{proof} (i) In order to establish (\ref{definition:predCompensator}), by (i) of Theorem II-1.8 of \cite{JacodShiryaev_LimitTheorems},  it suffices to show that for any nonnegative predictable process $H \in \pred \otimes \scr{B}(\mathbb{R})$, 
\begin{equation} \label{equation:pruebaCompensacion}
    \esp{(H * \hat{\mu})_\infty}{} = \esp{(H * \hat{\mu}^{\tt{p}})_\infty}{}, 
\end{equation}
and the first step consists in showing the result for deterministic $H_s(\omega , x) = \indic{B}(x)$ for $B \in \scr{B}(\mathbb{R})$. Maintaining the notation introduced in Lemma \ref{lemma:PreFormulaCompensacion} for the process $\hat{N}_B$, consider $B$ an arbitrary interval not containing a neighbourhood of the origin as well as the associated martingale, 
\begin{equation*}
    M_B(t) = t^{-p} \hat{N}_B^{(c)}(t) = t^{-p} \left( \hat{N}_B(t) - t \Lambda(B) \right).  
\end{equation*}
Integrating by parts, we get 
\begin{equation*}
    t^{p} M_B(t) = \int_0^t s^p \dd M_B(s) + \int_0^t p M_B(s) s^{p-1}\dd s, 
\end{equation*}
and consequently, 
\begin{align*}
    \hat{N}_B(t) - t \Lambda(B)  
    &= \int_0^t s^p \dd M_B(s) + \int_0^t p M_B(s) s^{p-1}\dd s  \\
    &= \int_0^t s^p \dd M_B(s) + \int_0^t  \left( \hat{N}_B(s) - s \Lambda(B) \right)  p s^{-1} \dd s \cdot  
\end{align*}
Said otherwise, 
\begin{equation*}
    \hat{N}_B(t) - t (1-p) \Lambda(B) -  \int_0^t   \hat{N}_B(s)  p s^{-1}\dd s   = \int_0^t s^p \dd M_B(s),  
\end{equation*}
is a martingale. Since $(N_B(\omega;s ))_{s \in \mathbb{R}^+}$ and $(N_B(\omega ; {s-}))_{s \in \mathbb{R}^+}$ differ in a set of null Lebesgue measure, the equality still holds replacing $\int_0^t \hat{N}_B(s) ps^{-1} \dd s$ by $\int_0^t \hat{N}_B(s-) ps^{-1} \dd s$   and we obtain precisely (\ref{formula:compensationFormula}) for $H_s(\omega , x) = \indic{B}(x)$. Now we can  proceed as in the proof of II-2.21 from \cite{JacodShiryaev_LimitTheorems}. Concretely,  pick  any  positive Borel-measurable deterministic function $h=h(x)$, $x \in \mathbb{R}$ such that $h * \hat{\mu} - h * \hat{\mu}^{\tt{p}}$ is a local martingale and let $T$ be an arbitrary stopping time. With the same terminology as in I.1.22 of \cite{JacodShiryaev_LimitTheorems} denote by $\llbracket 0,T \rrbracket$ the subset of $\Omega \times \mathbb{R}^+$ defined by 
\begin{equation*}
    \llbracket 0,T \rrbracket = \{ (\omega , s) :  0 \leq s \leq T(\omega) \}.
\end{equation*}
In particular, $(h * \hat{\mu})^T = \indic{\llbracket 0,T \rrbracket}h * \hat{\mu}$ where the process $\indic{\llbracket 0,T \rrbracket}$ is predictable (since left continuous)    and moreover, by Theorem I 2.2 of \cite{JacodShiryaev_LimitTheorems}, the sigma field generated by the collection
\begin{equation*}
\{   A \times \{ 0 \} \text{ where } A \in \scr{F}_0, \text{ and }  \llbracket 0,T \rrbracket \text{ where } T \text{ is any } (\scr{F}_t)  \text{-stopping time } \}
\end{equation*}
is precisely the predictable sigma field $\pred$.  Then, if $(T_n)$ is a localising sequence for the local martingale  $h * \hat{\mu} - h * \hat{\mu}^{\tt{p}}$, it follows from Doob's stopping theorem that for each $n$,
\begin{equation*}
    \esp{ (h * \hat{\mu})_{\infty}^{T \wedge T_n} }{} =  \esp{ (h * \hat{\mu}^{\tt{p}})_{\infty}^{T \wedge T_n} }{}. 
\end{equation*}
 Consequently, taking the limit as $n \uparrow \infty$, we deduce by monotone convergence that
\begin{equation*}
    \esp{ ( \indic{  \llbracket 0,T  \rrbracket} h * \hat{\mu})_\infty }{} = \esp{ (\indic{ \llbracket 0,T \rrbracket} h * \hat{\mu}^{\tt{p}})_\infty }{}
\end{equation*}
which in turn implies that (\ref{equation:pruebaCompensacion}) holds for any predictable process $H = 1_B 1_{\llbracket 0,T \rrbracket}$ where $B$ is any closed interval not containing the origin and $T$ an arbitrary stopping time. Now the claim  follows by a monotone class argument.
\end{proof}
We close our discussion on the jump process of NRLPs with the property at the heart of the infinite divisibility of $\hat{\xi}$ as a stochastic process, a topic that will be studied in Section \ref{section:IDprocesses}.  We claim that, for $A \in \scr{B}(\mathbb{R})$ with $\Lambda(A) < \infty$ the point process of jumps 
\begin{equation}\label{pointProcess}
    \nu_A( \dd s ) =  \sum_{s} \indic{\{ \Delta \hat{\xi}_s \in A \}} \delta_{s}, 
\end{equation}
is an infinitely divisible point process. More precisely,  the measure $\nu_A$ is a reinforced Poisson point process $\hat{\scr{P}}$ with rate  $\Lambda(A)$ in $\mathbb{R}^+$ and if we consider $n$ independent copies  $\nu_A^{1}, \dots ,\nu_A^{n}$ of the reinforced  Poisson process (\ref{pointProcess}) but with rate $n^{-1} \Lambda(A)$, we have the equality in distribution 
\begin{equation} \label{equation:idpointprocess}
    \nu_A \overset{\scr{L}}{=} \nu_A^{1}+ \dots +\nu_A^{n}.
\end{equation}
To see this, consider  $f: \mathbb{R}^+ \mapsto \mathbb{R}^+$ a positive function with support in $[0,t]$, and   observe  that 
\begin{equation*}
    \langle f , \nu_A \rangle = \sum_{s \leq t} \indic{A}(\Delta \xi_s)f(s). 
\end{equation*}
Now the claim follows by computing the Laplace functional of $\nu_A$, $\nu_A^i$  respectively, by applying the exponential formula  (\ref{formula:exponencialLaplace})  and from comparing with  (\ref{identity:laplaceNRPPcompactos}). For a more detailed discussion on infinitely divisible point processes  we refer to page 5 of \cite{MaillardStablePP}.

\addtocontents{toc}{\vspace{\contentsSpacingBefore}}
\section{Weak convergence of the pair of skeletons} 
\label{section:jointConvergence}
\addtocontents{toc}{\vspace{0 pt}}

Before stating the first result of the section, let us briefly recall the statement  of the Lévy-Itô synthesis for Lévy processes:  a Lévy process $\xi$ with triplet $(a,q^2,\Lambda)$ can be written as $\xi = \xi^{(1)} + \xi^{(2)} + {\xi}^{(3)}$, where  $\xi^{(1)} = ( a t + q  B_t : t \geq 0)$ is a Brownian motion with drift  while $\xi^{(2)} + {\xi}^{(3)}$ is a purely discontinuous process that can be explicitly built from the jump measure $\mu$ defined in (\ref{ppp-Levy}). More precisely, if we denote by $\mu^{(sc)}$ the compensated  measure of jumps $\mu^{(sc)} = \mu -\dd t\Lambda(\dd x)$,  we can write 
\begin{equation} \label{equation:LevyIto}
    \xi_t = a t + q B_t + \int_{[0,t]\times (-1,1)^c} x \mu({\dd s,\dd x}) +  \int_{[0,t]\times (-1,1)} x \mu^{(sc)}({\dd s,\dd x}), \quad \quad t \geq 0.
\end{equation}
 The  reinforced Lévy-Itô synthesis, which is the first main result of the section,  states that the analogous result holds for NRLPs where now, the PPP $\mu$ in \eqref{equation:LevyIto} has been replaced by the reinforced version $\hat{\mu}$, and the Brownian motion $B$ by its reinforced version $\hat{B}$ (if $p < 1/2)$.  More precisely,  after properly defining the "space-compensated" measure $\hat{\mu}^{(sc)}$, we prove: 
\begin{theorem}\label{corollary:constructionNRLPs} \emph{(Reinforced Itô's synthesis)} \\
Let $\hat{\mu}$ be the jump measure of a NRLP $\hat{\xi}$ of characteristics $(a,q^2,\Lambda, p)$. Then,  a.s. we have 
\begin{equation*} 
    \hat{\xi}_t = at + q\hat{B}_t + \int_{[0,t] \times (-1,1)^{c}} x \hat{\mu}( \dd s,\dd x) + \int_{[0,t] \times (-1,1)} x \hat{\mu}^{(sc)}( \dd s,\dd x), \quad \quad t \geq 0,
\end{equation*}
for some noise reinforced Brownian motion $\hat{B}$, with the convention that 
 if $p \geq  1/2$ the process $\hat{B}$ is null. Moreover, the integrals in the previous display are  NRLPs with respective characteristics    $(0,0, \indic{ (-1,1)^c }\Lambda ,p)$, $(0,0, \indic{ (-1,1) }\Lambda,p)$. 
\end{theorem}

\begin{remark} Beware of the notation,  $\hat{\mu}^{(sc)}$ stands for the \textit{space}-compensated  jump measure $\hat{\mu}$ and should not be confused with the \textit{time}-compensated measure $(\mu-\mu^p)$  in the sense of   \cite[Chapter II-1.27]{JacodShiryaev_LimitTheorems}. For instance, we stress that  $\hat{\xi}^{(3)}$ is not a local martingale. Remark that for Lévy processes, the time and space compensation of its jump measure coincide, since the compensating measure is the same.
\end{remark}

After proving this result, we start settling the ground for the main result of the section. First, making use of Theorem \ref{corollary:constructionNRLPs},  we   define  the joint law, of a Lévy process and its reinforced version, by introducing an appropriate coupling $(\xi , \hat{\xi})$.  We then   characterise its law by computing the characteristic function of its finite dimensional distributions: 

\begin{proposition}\label{theorem:leyConjunta}
There exists a pair $(\xi, \hat{\xi})$, where $\hat{\xi}$ has the law of a NRLP   with characteristics $(a,q^2,\Lambda, p)$, with law determined by the following: for all $k \geq 1$,  $\lambda_1, \dots , \lambda_k$, $\beta_1, \dots \beta_k$  real numbers, and $0<t_1< \dots < t_k \leq t$, we have 
\begin{align}
    &\esp{\exp\bigg\{ i \sum_{j=1}^k \big( \lambda_j  \xi_{t_j} + \beta_j  \hat{\xi}_{t_j}\big)  \bigg\}  }{} 
    = \nonumber \\
    &\exp \bigg\{t \cdot p  \esp{ \Psi \Bigg( \sum_{j=1}^k \lambda_j \indic{\{ U \leq t_j/t  \}} \Bigg) }{} +  t \cdot (1-p) \esp{\Psi \Bigg( \sum_{j=1}^k (\lambda_j \indic{\{ Y(t_j/t) \geq 1 \}} + \beta_i Y(t_j/t)) \Bigg) }{} \bigg\},  \label{eq:leyConjunta}
\end{align}
where $U$ is a uniform random variable in $[0,1]$. A pair of processes with such distribution will always be denoted by $(\xi, \hat{\xi})$.
\end{proposition}
  \par Now, we  connect  the distribution of the pair  $(\xi, \hat{\xi})$  with the discrete setting.  In this direction, consider  the Lévy process  $\xi$    and for each fixed $n \in \mathbb{N}$ we set 
 \begin{equation} \label{definition:incrementsSkeletton}
     X^{(n)}_k := \Delta^{(n)}\xi_{k} = \xi_{k/n} - \xi_{(k-1)/n}, \quad \text{ for }  k \geq 1.
 \end{equation}
 For each $n$, the sequence  $(X^{(n)}_k)$  is identically distributed with  law $\xi_{1/n}$ and the random walk $S^{(n)}_{k} = X^{(n)}_1 + \dots + X^{(n)}_k  \text{ for } k\geq 1$, $\, S^{(n)}_0 = 0$  built from these increments for a mesh of length $1/n$ is  referred to as  the $n-$skeleton of the Lévy process $\xi$. This process consists in the  positions of  $\xi$ observed at discrete time intervals and, if we write $D(\mathbb{R}_+)$ for the space of $\mathbb{R}_+$ indexed rcll functions into $\mathbb{R}$ with the Skorokhod topology,  we have  $S^{(n)}_{\lfloor n\cdot \rfloor} \overset{D(\mathbb{R}_+)}{\rightarrow} \xi$ as $n \uparrow \infty$. Now,  fix a memory parameter $p \in (0,1)$ and for each $n$, consider the associated noise reinforced random walk $(\hat{S}^{(n)}_k)$ with parameter $p$ built from the same collection of increments: 
 \begin{equation} \label{definition:LevyReinforcedSkeletton}
     \hat{S}_k^{(n)} := \hat{X}_1^{(n)} + \dots + \hat{X}_k^{(n)},  \quad \text{ for } k \geq 1, 
 \end{equation}
 where we set   $\hat{S}^{(n)}_0 := 0$. For a detailed  account on the noise reinforced random walk, we refer to the beginning of Section \ref{section:jointlaw}. The main result in \cite{BertoinNRLP} states that $\hat{S}_{\lfloor n \cdot  \rfloor} \overset{f.d.d.}{\rightarrow} \hat{\xi}$, the convergence holding in the sense of finite-dimensional distributions, and we shall now   strength this result.
 To simplify notation, write  $D^2(\mathbb{R}_+)$ the product space $D(\mathbb{R}_+) \times D(\mathbb{R}_+)$ endowed with the product topology.  Now we can state the main result of the section: 
 \begin{theorem} \label{thm:convergenciaConjunta}
Let $\xi$ be a Lévy process with characteristic triplet $(a,q^2,\Lambda)$, fix  $p \in (0,1/2)$ an admissible memory parameter and for each $n$, let 
$(S^{(n)}_k,\hat{S}^{(n)}_k)$ be the pair of the $n$-skeleton of $\xi$ and its reinforced version.  Then, there is weak convergence in  $D^2(\mathbb{R}_+)$ as $n \uparrow \infty$
\begin{equation} \label{equation:joint convergence}
    \Big(  S^{(n)}_{\lfloor n \cdot  \rfloor} ,   \hat{S}^{(n)}_{\lfloor n \cdot  \rfloor} \Big) \overset{\scr{L}}{ \longrightarrow} (  \xi , \hat{\xi} ), 
\end{equation}
where  $(\xi, \hat{\xi})$ is a pair of processes with law  (\ref{eq:leyConjunta}).  
\end{theorem}

\noindent \textit{The section is organised as follows:}   In Section \ref{subsection:jump_LevyItoPart2}, after  introducing the (space) compensated integral with respect to NRPPPs, we shall establish Theorem   \ref{corollary:constructionNRLPs}. Making use of this result, in Section  \ref{section:jointlaw}  we define the joint law of a Lévy process and its reinforced version $(\xi, \hat{\xi})$. More precisely,  by Lévy-Itô Synthesis and its  reinforced  version of Theorem   \ref{corollary:constructionNRLPs}, it will suffice to define the joint law of $(\mu , \hat{ \mu})$ and $(B , \hat{B})$.  This is  respectively the content of the construction detailed in \ref{section:joint_NN} and Definition \ref{lemma:FourierParejaGaus}. The construction of $\hat{\mu}$ is  done explicitly in terms of  the jump measure of $\xi$ by a procedure that should be interpreted as the continuous-time reinforcement analogue of the reinforcement algorithm for random walks. We then  introduce the joint law $(\xi, \hat{\xi})$ in Definition \ref{definition:Levy_LevyR} and  prove Proposition \ref{theorem:leyConjunta}. Finally, Section \ref{subsection:PruebajointConvergence} is devoted to the proof of Theorem \ref{thm:convergenciaConjunta}. 

\subsection{ Proof of Theorem \ref{corollary:constructionNRLPs} } \label{subsection:jump_LevyItoPart2}
Let us start by  introducing  the (space)-compensated integral with respect to NRPPPs. Recall  the identity of  Lemma  \ref{proposition:intensidadNRPPP} for the intensity measure of NRPPPs and  for fixed $t \in \mathbb{R}$, let  $f:\mathbb{R}^+ \times \mathbb{R}\mapsto \mathbb{R}$ be a  measurable function satisfying, for all $0 < a < b$,  the integrability condition
\begin{equation*}
      \int_{(0,t] \times \{ a \leq | x | <b \} } |f(s,x)|\dd s \Lambda(\dd x) < \infty. 
\end{equation*}
 Next, we set 
\begin{equation} 
    \int_{[0,t]\times \{a \leq |x| < b\}} f(s,x) \hat{\mathcal{N}}^{(\text{sc})}( \dd s ,\dd x) := \int_{[0,t]\times \{a \leq |x| < b\}} f(s,x) \hat{\mathcal{N}}( \dd s ,\dd x) -  \int_{(0,t]\times \{a \leq |x| < b\}} f(s,x)\dd s\Lambda(\dd x).
\end{equation}
This is a centred  random variable and if we denote it by  $\Sigma_{a,b}^{(c)}(f,t)$,  from Proposition \ref{proposition:reinforcedPPPbyMarking} - (ii)  we deduce that $(\Sigma_{e^{-r},b}^{(c)}(f,t))_{r \in [- \log(b),\infty)}$ has independent increments, and hence is a martingale.  When the limit of this martingale exists, we will write 
\begin{equation} \label{equation:compensatedIntegral_}
     \int_{[0,t]\times (-b ,b )} f(s,x) \hat{\mathcal{N}}^{(\text{sc})}( \dd s ,\dd x) := \lim_{r \uparrow \infty}  \int_{[0,t]\times \{e^{-r}\leq |x|<b\}} f(s,x) \hat{\mathcal{N}}^{(\text{sc})}( \dd s ,\dd x). 
\end{equation}
 Recall that the characteristics of a NRLP are being considered with respect to the cutoff function $x\indic{\{ |x|< 1 \}}$ as well as the notation $f * \hat{\mathcal{N}}$ from \eqref{equation:integrationMeasuresJacod}. The following lemma shows that the sums of atoms of NRPPPs are precisely  purely discontinuous NRLPs:
\begin{lemma}\label{lemma:NRPPasIntegralsOfNRPPP} Fix a Lévy measure $\Lambda$, a parameter $p \in (0,1)$ such that $\beta(\Lambda)p<1$ and let $\hat{\mathcal{N}}$ be a  NRPPP with characteristic measure $\Lambda$ and reinforcement parameter $p$.
\begin{enumerate}
    \item[\emph{(i)}] For any $0< a<b$, the process $ \indic{\{ a \leq |x| < b \}} x * \hat{\mathcal{N}}$ is a noise reinforced compound Poisson process with characteristics $( \Lambda(\indic{\{a \leq |x| < 1 \}} x), 0, \indic{\{a \leq  |x| < b \}} \Lambda, p)$.
    \item[\emph{(ii)}] For each  $t \in \mathbb{R}^+$ the compensated integral 
\begin{equation} \label{equation:convergenceCompesantedIntegralNRLP}
    \int_{[0,t]\times (-1,1)} x \hat{\mathcal{N}}^{(\text{sc})}( \dd s ,\dd x) 
    :=   
    \lim_{r \uparrow \infty}  \int_{[0,t]\times \{e^{-r}\leq |x|<1\}} x \hat{\mathcal{N}}^{(\text{sc})}( \dd s ,\dd x)
\end{equation}
exists. The process $\indic{(-1,1)} x * \hat{\mathcal{N}}^{(sc)}$ is a NRLP with characteristics $(0,0,\indic{(-1,1)} \Lambda,p)$ and hence has a rcll modification.  Moreover, the convergence (\ref{equation:convergenceCompesantedIntegralNRLP}) holds towards its rcll modification  uniformly in compact intervals for some subsequence $(r_n)$, and we shall consider it and denote it in the same way without further comments. 
\end{enumerate}
\end{lemma}
\begin{proof}
(i) If we consider  $\hat{\xi}$ a reinforced compound Poisson process with such characteristics and $\hat{\mu}$ is its jump measure, it is a pure jump process and we can write it as the sum of its jumps.  Our claim can now be  proved directly from the identity $\hat{\xi} = (x * \hat{\mu}) \overset{\scr{L}}{=}  (\indic{(a \leq |x| < b)} x * \hat{\mathcal{N}})$, since by Proposition  \ref{proposition:propiedadesNRPPP} - (i), the restriction $\indic{(a \leq |x| < b)} \hat{\mathcal{N}}$ has the same distribution as $\hat{\mu}$. Alternatively, this can be established by means  of the exponential formulas we obtained in Proposition \ref{proposition:exponentialFormulasJumpProcess},  by  fixing $0<t_1< \dots < t_k<t$ and computing the characteristic function of the finite-dimensional distributions at times $t_1, \dots ,  t_k$ of $\indic{\{a \leq |x|<b\}} x * \hat{\mathcal{N}}$, noticing  that for   $f(s,x) := \left( \sum_{j=1}^k \lambda_j \indic{\{ s \leq t_j\}}\right) x \indic{\{a \leq |x| < b \}}$ we have 
\begin{equation*}
    \sum_{j=1}^k \lambda_j ( \indic{\{ a \leq  |x| < b\}} x * \hat{\mathcal{N}})_{t_j} = \int_{[0,t] \times\mathbb{R}} f(s,x) \hat{\mathcal{N}}( \dd s ,\dd x). 
\end{equation*}
The claim follows by  comparing with the identity for the characteristic function of the finite-dimensional distributions (\ref{identity:extension})  of $\hat{\xi}$ . \\
(ii) Recall the notation introduced before (\ref{equation:compensatedIntegral_}) for the martingale $(\Sigma_{e^{-r} ,1}^{(c)}(f,t))_{r \geq 0}$. In our case, we have $f(s,x) = x$ and we just write $(\Sigma_{e^{-r} ,1}^{(c)}(t))_{r \geq 0}$. The fact that the martingale $(\Sigma_{e^{-r} ,1}^{(c)}(t))_{r \geq 0}$ converges as $r \uparrow \infty$ and that the limit is a NRLP with characteristics $(0,0,\indic{(-1,1)} \Lambda )$ can be achieved  by similar arguments as in \cite{BertoinNRLP} after a couple of observations. Starting with the former, recall the definition of $\hat{\mathcal{N}}$ from (\ref{definition:decoratedNRPP}), and remark that for each $r>0$ we have 
\begin{equation*}
    \int_{[0,t] \times \mathbb{R} } \indic{\{ e^{-r}\leq |x| < 1\}} x \hat{\mathcal{N}}( \dd s ,\dd x) 
    = \sum_{u \in \scr{P}} \indic{\{ u \leq t \}} \indic{\{ e^{-r}\leq |x_u| < 1\}} x_u \cdot \# \big\{ \{ ue^{s} : s \in D_u \} \cap [0,t] \big\}. 
\end{equation*}
From the discussion right after Proposition \ref{proposition:reinforcedPPPbyMarking}, we infer that if  we we consider  $(Z^u)_{u \in \mathscr{P}}$ an independent collection of independent, standard  Yule processes,  the family  $\{ ue^{s} : s \in D_u \}$ has the same distribution as the collection of jump times of the counting process $\indic{\{ u \leq t \}} Z^{u}_{p(\ln(t) - \ln(u))}$, $t \geq 0$. Hence the previous display  can also be written as 
\begin{equation*}
    \sum_{u \in \scr{P}}\indic{\{ e^{-r}\leq |x_u| < 1\}} x_u \indic{\{ u \leq t \}} Z^u_{p(\ln(t) - \ln(u))},  
\end{equation*}
and now the proof of the convergence as $r \uparrow \infty$ of  $(\Sigma_{e^{-r} ,1}^{(c)}(t))_{r \geq 0}$  follows by the same arguments as in  \cite[Lemma 2.6]{BertoinNRLP}. Alternatively, one can make use of   \eqref{identity:extension} to restrict our arguments to the interval $[0,1]$ and apply \cite[Lemma 2.6]{BertoinNRLP}.  Next,  to see that the process $\indic{(-1,1)}x * \hat{\mathcal{N}}^{(sc)}$ defines a NRLP with characteristics $(0,0,\indic{(-1,1)}\Lambda)$, fix $0< t_1< \dots < t_k < t$ and for $\epsilon>0$,   $\lambda \in \mathbb{R}$ set 
\begin{equation*}
    \Phi_{\epsilon,1}^{(3)}(\lambda) = \int_{\{\epsilon \leq |x| < 1\}} \left( e^{i\lambda x} -1- i\lambda x  \right)\Lambda(\dd x).
\end{equation*}
Recalling the formula (\ref{formula:exponencialFourier}) for the characteristic function of integrals with respect to NRPPPs, we deduce from considering the function $f(s,x) := \left( \sum_{j=1}^k \lambda_j \indic{\{ s \leq t_j\}}\right) x \indic{\{\epsilon \leq |x| < 1 \}}$ that we have 
\begin{equation*}
     \esp{ \exp \bigg\{ i \sum_{j=1}^k  \lambda_j ( \indic{\{ \epsilon \leq  |x| < 1\}} x * \hat{\mathcal{N}}^{(sc)})_{t_j}) \bigg\} }{} = \exp \bigg\{ t (1-p) \esp{ \Phi_{\epsilon,1}^{(c)}\left( \sum_{j=1}^k \lambda_j Y(t_j/t) \right) }{} \bigg\}.
\end{equation*}
Now we can apply the exact same reasoning as in the proof of Corollary 2.8 in   \cite{BertoinNRLP} by writing $s_j = t_j/t \in [0,1]$ and taking the limit as $\epsilon \downarrow 0$.  The uniform convergence in compact intervals towards the rcll modification of $\indic{(-1,1)} x * \hat{\mathcal{N}}^{(sc)}$ follows from the second statement of Theorem \ref{regularityNRLP},  since for every $\epsilon \in (0,1)$, the process 
\begin{equation*}
    \int_{[0,t]\times \{0 \leq |x|< \epsilon\}} x \hat{\mathcal{N}}^{(\text{sc})}( \dd s ,\dd x), \quad t \geq 0, 
\end{equation*}
is a NRLP with characteristics $(0,0, \indic{\{|x| < \epsilon\}}\Lambda  )$. 
\end{proof}
\comentario{{\color{red} Para justiifcar unforme, usando el scalling basta en (0,1) y luego usar igualdad en ley con construcction en \cite{BertoinNRLP} ?}}
It immediately follows from the previous lemma that if  $\hat{\mathcal{N}}$ is a  NRPPP with characteristic measure $\Lambda$, parameter $p$ and, if $p < 1/2$, we consider    $\hat{W}$  an independent NRBM with same parameter,  then
\begin{equation} \label{equation:constructionOfNRLPS}
    \hat{X}_t = at + q\hat{W}_t + \int_{[0,t] \times (-1,1)^{c}} x \mathcal{\hat{N}}( \dd s,\dd x) + \int_{[0,t] \times (-1,1)} x \mathcal{\hat{N}}^{(sc)}( \dd s,\dd x), \quad \quad t \geq 0,
\end{equation}
defines a NRLP with characteristics $(a,q^2,\Lambda, p)$. To obtain the a.s. statement of Theorem \ref{corollary:constructionNRLPs} we still need a short argument. 
\begin{proof}[Proof of Theorem \ref{corollary:constructionNRLPs}.]
 The result will be deduced from the equality in distribution $\hat{\xi} \overset{\scr{L}}{=} \hat{X}$ for $\hat{X}$ defined as in (\ref{equation:constructionOfNRLPS}) with same characteristics as $\hat{\xi}$. In this direction, wlog we assume $p < 1/2$,  $q=1$  and we set
\begin{align*} 
    \hat{\xi}^{\leq 1} := \hat{\xi}_t - \sum_{s \leq t} \indic{\{ |\Delta \hat{\xi}_s| \geq 1 \}} \Delta \hat{\xi}_s 
    \hspace{5mm} \text{ and } \hspace{5mm}
    \hat{\xi}^{< \epsilon}_t := \hat{\xi}_t^{\leq 1} - \Big( \sum_{s \leq t} \indic{\{\epsilon \leq |\Delta \xi_s| < 1\}} \Delta \hat{\xi}_s - t \int_{\{ \epsilon \leq |x| < 1\}} x \Lambda(\dd x) \Big). 
\end{align*}
Notice that for every $\epsilon >0$, we can write 
\begin{equation} \label{equation:decompItoPrueba}
    \hat{\xi}_t = \hat{\xi}_t^{< \epsilon} +  \Big( \sum_{s \leq t} \indic{\{\epsilon \leq |\Delta \xi_s| < 1\}} \Delta \hat{\xi}_s - t \int_{\{ \epsilon \leq |x| < 1\}} x \Lambda(\dd x) \Big)  + \hat{\xi}^{(2)}_t.
\end{equation}
Since $\hat{\mu}$ is a reinforced PPP, by Lemma \ref{lemma:NRPPasIntegralsOfNRPPP} the process  (\ref{equation:decompItoPrueba})  converges uniformly in compact intervals for some subsequence $(\epsilon_n)$ as $\epsilon_n \downarrow 0$ towards $\hat{C} + \hat{\xi}^{(2)} + \hat{\xi}^{(3)}$, for some process $\hat{C}:= \hat{\xi} - \hat{\xi}^{(2)} - \hat{\xi}^{(3)}$ continuous by construction. Since $\hat{\mu}$ is a reinforced PPP, by the   independence  properties of its restriction we know that $\hat{\xi}^{(2)}, \hat{\xi}^{(3)}$ are independent. Hence,  it remains to show that $(\hat{\xi}^{(2)}, \hat{\xi}^{(3)})$ is independent of $\hat{C}$ and that $\hat{C} - at =: \hat{B}$ is a NRBM.
Fix arbitrary $0<u<v \leq \infty$ and   maintain the notation for $\hat{W}$, $\hat{\mathcal{N}}$ used in the  representation (\ref{equation:constructionOfNRLPS}). Since $\hat{\mathcal{N}}$ is the  clearly the jump measure of $\hat{X}$,  we have the equality in distribution:  
\begin{equation} \label{equation:representationsNRLPequality}
    \Big( \hat{\xi}, \sum_{s \leq \cdot} \indic{\{ u \leq  |\Delta \hat{\xi}_s| < v \}} \Delta \hat{\xi}_s  \Big) \overset{\scr{L}}{=} \left( \hat{X} , \indic{\{u \leq  |x| < v \}} x * \mathcal{\hat{N}} \right).
\end{equation}
Moreover, since $\hat{W}$ is independent of $\hat{\mathcal{N}}$, from the independence of restrictions of NRPPP and (\ref{equation:representationsNRLPequality}) we deduce that  $\indic{\{ \epsilon \leq |x| < 1\}}x * \hat{\mu}^{(sc)} + \indic{\{ 1 \leq |x| \}}x * \hat{\mu}$ and $\hat{\xi} - \indic{\{ \epsilon \leq |x| < 1\}}x * \hat{\mu}^{(sc)} - \indic{\{ 1 \leq |x| \}}x * \hat{\mu}$ are independent, 
the later having the same distribution as $at +  \hat{W}_t + \indic{(-\epsilon, \epsilon)}x * \hat{\mathcal{N}}$. Now the claim follows by  taking the limit as $\epsilon \downarrow 0$. 
\end{proof}

{
\subsection{The joint law $(\xi, \hat{\xi})$ of a Lévy process and {its} reinforced version}
\label{section:jointlaw}
}
In this section we  construct explicitly, for an arbitrary fixed Lévy process $\xi$, the process $\hat{\xi}$ in terms of $\xi$ that will be referred to as {the} noise reinforced version of $\xi$. This will yield a definition for the joint law $(\xi, \hat{\xi})$. Our construction will be justified by the weak convergence of Theorem \ref{thm:convergenciaConjunta}. Let us start by recalling the discrete setting, since our construction is essentially the continuous-time analogue of the dynamics that we now describe. \medskip \\
$\circ$ \textit{The noise reinforced random walk.} Given a collection of identically distributed random variables $(X_n)$ with law $X$, denote by  
 $S_n := X_1 + \dots + X_n, \text{ for }  n \geq 1$ the corresponding random walk. We  construct, simultaneously to $(S_n)$, a noise reinforced version  using the same sample of random variables and performing the reinforcement algorithm at each discrete time step. In this direction, consider $(\epsilon_n)$ and $(U[n])$ independent sequences of  Bernoulli random variables with parameter $p \in (0,1)$ and uniform random variables on $\{1,\dots ,n \}$ respectively. Set $\hat{X}_1 := X_1$ and, for $n \geq 1$, define 
 \begin{equation*}
     \hat{X}_{n+1} :=  X_{n+1}\indic{\{ \epsilon_{n+1} = 0 \}} + \hat{X}_{U[n]} \indic{\{ \epsilon_{n+1} = 1 \}}. 
 \end{equation*}
Finally, we denote the corresponding partial sums  by $\hat{S}_n := \hat{X}_1 + \dots + \hat{X}_n, \,  n \geq 1$. The process $(\hat{S}_n)$ is the so-called noise reinforced random walk with memory parameter $p$, and  we  refer to this particular construction of  $(\hat{S}_n)$ as the noise reinforced version of $(S_n)$. The process $(\hat{S}_n)$ can be written in terms of the individual contributions made by each one of the steps. In this direction, let us   introduce  a counting process keeping track of the number of times each step $X_k$ is repeated up to time $n$. Since if the law of $X$ has atoms, we have  $\mathbb{P}({X}_1 = X_2) > 0$, and we need to perform a slight modification to our algorithm.  Namely, for each $n \geq 1$ we write $X_n' := (X_n,n)$  and we perform the reinforcement algorithm to the pairs $(X'_n)$. This yields a sequence that, with a slight abuse of notation, we denote by  $(\hat{X}'_n)$. If for every $k, n \geq 1$ we set:
\begin{equation} \label{definition:countingRepsNRRW}
    N_k(n) := \# \{ 1 \leq i \leq n : X_k' = \hat{X}_i' \}, 
\end{equation}
we can write:  
\begin{equation} \label{identity:definition:countingRepsNRRW}
    \hat{S}_n = \sum_{k=1}^\infty N_k(n) X_k, \quad \text{ for } n \geq 1.
\end{equation}
For convenience, we always set $S_0 = 0 = \hat{S}_0$, and  when working with pairs of the form  $(S, \hat{S})$ it will always be implicitly assumed that the noise reinforced version has been constructed by the algorithm we described.  For instance, it is clear that at each discrete time step $n$,  with probability $1-p$, ${S}_n$  and $\hat{S}_n$ share the same increment, while with complementary probability $p$, they perform different steps.  
\medskip \\
\indent Roughly speaking, in the continuum,  the steps $(X_n)$ are replaced by jumps $\Delta \xi_s$ of the Lévy process $\xi$. With probability $1-p$, the jump is  shared with its reinforced version $\hat{\xi}$ while with complementary probability $p$ it is  discarded and remains independent of $\hat{\xi}$. The jumps that are not discarded by this procedure are then repeated at each jump time of an independent  counting process that will be attached to it. The process of discarding jumps with probability $p$ is traduced in a  thinning of the jump measure of $\xi$. Let us now give a formal description of this heuristic discussion. \vspace{0.25mm}  \\
\subsubsection{Construction of the pair  $(\mathcal{N} , \hat{\mathcal{N}} )$ }\label{section:joint_NN}
For the rest of the section, we fix a Lévy process $\xi$ with non-trivial Lévy measure $\Lambda$, denote the set of its jump times  by $\scr{I} := \{ u \in \mathbb{R}^+ : {\Delta }\xi_u \neq 0 \}$ and  let 
\begin{equation*}
    \mathcal{N}( \dd s ,\dd x) := \sum_{ u \in \scr{I}} \delta_{(u,\Delta \xi_u)}, 
\end{equation*}
be its jump measure. By the Lévy-Itô decomposition, this is a PPP with characteristic measure $\Lambda$ and   we can write  $\xi = \xi^{(1)} + J$, where $\xi^{(1)}$ is a continuous process while  $J$ is a  process that can be explicitly recovered from $\mathcal{N}$, as we recalled in (\ref{equation:LevyIto}). 

If $\hat{\xi}$ has the law of its reinforced version, by Theorem  \ref{corollary:constructionNRLPs} it can also be written as $\hat{\xi} = \hat{\xi}^{(1)} + \hat{J}$, where  $\hat{J}$  is a functional of a  NRPPP $\hat{\mathcal{N}}$ with characteristic measure $\Lambda$. Hence,   the main step for defining the law of the pair $(J,\hat{J})$ consists in appropriately defining $(\mathcal{N}, \hat{\mathcal{N}})$. However, recalling the construction of NRPPPs by superposition detailed before Definition \ref{proposition:reinforcedPPPbyMarking},  this can be achieved as follows: first, set  $A_0 := \{ 1 \leq |x| \}$ and for each $j \geq 1$, let $A_j := \{ 1/(j+1) \leq |x| < 1/j \}$. Next, for $j \geq 0$ consider the point process 
 \begin{equation*}
     \scr{I}_j := \{ u \in \mathbb{R}^+ : \Delta {\xi}_u \in A_j \}, 
 \end{equation*}
remark that  $\scr{I}_j$ is a PPP with intensity $\Lambda(A_j) \dd t$ 
and write  $\scr{I}:= \cup_j \scr{I}_j$. Maintaining the notation of Section \ref{section:NRPPP}, consider  $(D_u)_{u \in \scr{I}}$   a collection of i.i.d. copies of  $D$ and for each $j \geq 0$ we set 
\begin{equation*}
    \mathcal{N}_j(\dd s , \dd x) := \sum _{u \in \mathscr{I}_j} \sum_{t \in D_u} \delta_{(ue^{t}, \Delta \xi_u ) }.  
\end{equation*}
The measure $\mathcal{N}_j$ is a NRPPP with characteristic measure $(1-p)^{-1} \Lambda( \,  \cdot \cap A_j )$, 
and we can now   proceed as in Section \ref{subsection:jump_NRPPPdecorated} to construct  the following NRPPP with parameter $p$ by superposition of $(\mathcal{N}_j)_{j \geq 1}$, 
\begin{equation} 
     \sum_{u \in \scr{I} } \sum_{t \in D_u} \delta_{( ue^{t} , \Delta {\xi}_u)}.
\end{equation}
Notice however that its  characteristic measure is  $(1-p)^{-1} \Lambda$. In this direction, we consider a sequence of independent Bernoulli random variables $(\epsilon_u)_{u \in \scr{I}}$ with parameter $1-p$ and apply a thinning:  \begin{equation} \label{definition:thereinfocedversion_ofPPP}
    \hat{\mathcal{N}}( \dd s ,\dd x) :=   \sum_{u \in \scr{I} } \indic{\{ \epsilon_u = 1 \}}\sum_{t \in D_u} \delta_{( ue^{t} , \Delta {\xi}_u) }.
\end{equation}
Now, $\hat{\mathcal{N}}$ is a NRPPP with characteristic measure $\Lambda$ and reinforcement parameter $p$  built explicitly from the jump process of ${\xi}$. From the construction, if a jump $\Delta \xi_u$ occurs at time $u$, with probability $1-p$ it is kept  and repeated at each $ue^{t}$ for $t \in D_u$,  while with complementary probability $p$, it is discarded  and remains independent of $\hat{\mathcal{N}}$. From now on, we always consider the pair $(\mathcal{N}, \hat{\mathcal{N}})$ constructed by this procedure. Then, by definition of $\mathcal{N}$ we can write  
\begin{equation*}
    \hspace{-10mm} J_t = {\xi}^{(2)}_t + {\xi}^{(3)}_t = \int_{[0,t] \times \{ |x|\geq 1 \}} x \mathcal{N}( \dd s ,\dd x) +  \int_{[0,t] \times (-1,1)} x \mathcal{N}^{(sc)}( \dd s ,\dd x),  \quad \quad t \geq 0, 
\end{equation*}
while on the other hand, by Theorem  \ref{corollary:constructionNRLPs} the process defined as 
\begin{equation} \label{definition:THeresinforcedJumpLevy}
    \hat{J}_t = \hat{\xi}^{(2)}_t + \hat{\xi}^{(3)}_t := \int_{[0,t] \times \{ |x|\geq 1 \}} x \hat{\mathcal{N}}( \dd s ,\dd x) +  \int_{[0,t] \times (-1,1)} x \hat{\mathcal{N}}^{(sc)}( \dd s ,\dd x),  \quad \quad t \geq 0, 
\end{equation}
is a NRLP with characteristics $(0,0,\Lambda, p)$. 
From our construction, the random measures  $\mathcal{N}$, $\hat{\mathcal{N}}$ can be encoded in terms of a  single Poisson random measure $\sum_{u\in \scr{I}} \delta_{(u,\Delta \xi_u, D_u , \epsilon_u)}$, allowing us to  compute explicitly the characteristic function of the finite dimensional distributions of  $(\xi^{(2)},\hat{\xi}^{(2)})$ and $(\xi^{(3)},\hat{\xi}^{(3)})$. In this direction, for $\lambda \in \mathbb{R}$ recall the notation 
\begin{equation} \label{equation:charExpoSaltosGrandes}
    \Phi^{(2)}(\lambda) := \int_{\{|x| \geq 1\}} (e^{i \lambda x}-1) \Lambda(\dd x). 
\end{equation} 
\begin{lemma} \label{lemma:FourierPareja(1)}
For all $k \geq 1$, let $\lambda_1, \dots , \lambda_k$ and $\beta_1, \dots \beta_k$ be real numbers and fix times $0<t_1< \dots < t_k<t$. Then, we have 
\begin{align}
    &\esp{\exp\bigg\{ i \sum_{j=1}^k \lambda_j \big( \xi_{t_j}^{(2)} + \beta_j\hat{\xi}^{(2)}_{t_j}  \big) \bigg\}  }{} = 
     \nonumber \\
    &\exp \bigg\{ t p \esp{\Phi^{(2)}\left( \sum_{j=1}^k \lambda_j \indic{\{ Y(t_j/t) \geq 1 \}} \right) }{} + t (1-p) \esp{\Phi^{(2)}\left( \sum_{j=1}^k ( \lambda_j \indic{\{ Y(t_j/t) \geq 1 \}} + \beta_i Y(t_j/t)) \right) }{} \bigg\} \label{eq:leyConjunta(1)}, 
\end{align}
where we denote by $Y$ a Yule-Simon process with parameter $1/p$. 
\end{lemma}
Let us briefly comment on this expression. The first exponential term in (\ref{eq:leyConjunta(1)}) corresponds to the characteristic function of the  finite dimensional distributions of a Lévy process with law  $(\xi_{pt}^{(2)})_{t \in \mathbb{R}^+}$, viz. 
\begin{equation*}
    \esp{\exp\bigg\{ i \sum_{j=1}^k \lambda_j \xi_{p t_j}^{(2)} \bigg\}  }{} = \exp \bigg\{ t p \esp{\Phi^{(1)}\Big( \sum_{j=1}^k \lambda_j \indic{\{ U \leq  t_j/t \}} \Big) }{} \bigg\}, 
\end{equation*}
where $U$ is a uniform random variable in $[0,1]$ (recall that the first jump time of a Yule-Simon process is uniformly distributed in $[0,1]$). More precisely, this Lévy process is  built from the discarded jumps $\sum_{u} \indic{\{ \epsilon_u = 0 \}}\delta_{(u, \Delta \xi_u)}$ and consequently is independent of $\hat{\xi}^{(2)}$ and  $\sum_{u} \indic{\{ \epsilon_u = 1 \}}\delta_{(u, \Delta \xi_u)}$, which explains the form of the identity  \eqref{eq:leyConjunta(1)}. 
\begin{proof}
We can assume that $t_k < 1$ by working with  $t_1/t < \dots < t_k /t$  and with the pair $(\xi_{s t} , \hat{\xi}_{s t})_{s \in [0,1]}$, which now has Lévy measure $t \Lambda$. {Now, the proof follows by a rather long but straightforward application of the formula for the characteristic function of  integrals with respect to Poisson random measures.}  
\end{proof}

We now turn our attention to the characteristic function of the finite dimensional distributions of $(\xi^{(3)} , \hat{\xi}^{(3)} )$. In this direction, for $\lambda \in \mathbb{R}$, recall the notation 
\begin{equation} \label{equation:charExpoSaltosChicos}
    \Phi^{(3)}(\lambda) := \int_{\{|x| < 1 \}} (e^{i \lambda x} - 1 - i \lambda x ) \Lambda(\dd x).  
\end{equation}

\begin{lemma} \label{lemma:FourierPareja(c)}
For all $k \geq 1$, let $\lambda_1, \dots , \lambda_k$ and $\beta_1, \dots \beta_k$ be real numbers and fix times $0<t_1< \dots < t_k<t$. Then, we have 
\begin{align}
    &\esp{\exp\bigg\{ i \sum_{j=1}^k \big( \lambda_j \xi_{t_j}^{(3)} + \beta_j \hat{\xi}^{(3)}_{t_j} \big) \bigg\}  }{} 
    = \nonumber \\
    &\exp \bigg\{t p \esp{\Phi^{(3)}\left( \sum_{j=1}^k \lambda_j \indic{\{ Y(t_j/t) \geq 1 \}} \right) }{} + t (1-p) \esp{\Phi^{(3)}\left( \sum_{j=1}^k (\lambda_j \indic{\{ Y(t_j/t) \geq 1 \}} + \beta_i Y(t_j/t)) \right) }{} \bigg\}, \label{eq:leyConjunta(c)}
\end{align}
where we denote by $Y$ a Yule-Simon process with parameter $1/p$.
\end{lemma}
\begin{proof}
By the usual scaling argument we can suppose that $t_k < 1 =t$. Now, the proof is  similar to the one of Corollary 2.8 in \cite{BertoinNRLP}. In this direction,  notice that  the processes ${\xi}^{(3)} = \indic{(-1,1)} x  * {\mathcal{N}}^{(sc)}$ and $\hat{\xi}^{(3)} = \indic{(-1,1)} x  * \hat{\mathcal{N}}^{(sc)}$ are respectively  the limit as $\epsilon \downarrow 0$   of 
\begin{align} \label{definition:compensatedYSseries_1}
   \xi^{(3)}_{\epsilon,1} := \indic{\{ \epsilon \leq |x| < 1\}} x  * {\mathcal{N}} - \indic{\{ \epsilon \leq |x| < 1\}} x  *\dd t\otimes \Lambda,  
\end{align}
\begin{equation}\label{definition:compensatedYSseries_2}
     \hat{\xi}^{(3)}_{\epsilon ,1} := \indic{\{\epsilon \leq |x| < 1\}} x  * \hat{\mathcal{N}} - \indic{\{\epsilon \leq |x| < 1\}} x  *\dd t\otimes \Lambda, 
\end{equation}
     the convergence holding uniformly in compact intervals. The characteristic function of the finite-dimensional distributions of the pair $(\indic{\{\epsilon \leq |x| < 1\}} x  * {\mathcal{N}} , \indic{\{ \epsilon  \leq |x| < 1\}} x  * \hat{\mathcal{N}})$  can be computed by the same arguments as in Lemma \ref{lemma:FourierPareja(1)} and we obtain for each $0 < \epsilon < 1$ that 
\begin{align}
    &\esp{\exp\bigg\{ i \sum_{j=1}^k ( \lambda_j \xi_{\epsilon,1}^{(3)}(t_j) + \beta_j \hat{\xi}^{(3)}_{\epsilon,1}{(t_j)}) \bigg\}  }{} 
      \nonumber \\
    &= \exp \bigg\{ p \esp{\Phi^{(3)}_\epsilon\left( \sum_{j=1}^k \lambda_j \indic{\{ Y(t_j) \geq 1 \}} \right) }{} + (1-p) \esp{\Phi^{(3)}_\epsilon\left( \sum_{j=1}^k \big( \lambda_j \indic{\{ Y(t_j) \geq 1 \}} + \beta_i Y(t_j) \big) \right) }{} \bigg\}. \label{eq:EncaminoaleyConjunta(1)}
\end{align}
In order to establish that this expression converges as $\epsilon \downarrow 0$ towards  (\ref{eq:leyConjunta(c)}), we recall that since $|e^{ix} -1-ix|$ is $O(|x^2|)$ as $|x| \downarrow 0$ and $O(|x|)$ as $|x| \uparrow \infty$, for any $r \in (\beta(\Lambda) \vee 1  , 1/p \wedge 2)$ if $\beta(\Lambda) < 2$ and  $r=2$ if $\beta(\Lambda) = 2$, we have 
\begin{equation*}
    C := \sup_{x \in \mathbb{R}} |x|^{-r} |e^{ix}-1-ix| < \infty.
\end{equation*}
It follows that for all $0 < \epsilon < 1$, $\lambda \in \mathbb{R}$, we can bound 
\begin{equation*}
    |\Phi_\epsilon^{(3)}(\lambda)| 
    \leq \int_{\{ |x|<1 \}} |e^{i\lambda x} -1- i\lambda x|\Lambda(\dd x) \leq C  |\lambda|^r \int_{\{ |x|<1 \}} |x|^r \Lambda(\dd x).
\end{equation*}
Moreover,  by the remark following Lemma \ref{lemma:yuleSimonProcess},  the random variable  $Y(t) \in L_r(\mathbb{P})$ for any $r < 1/p$ and it follows that  the term  
$$\sum_{j=1}^k \big(\lambda_j \indic{\{ Y(t_j) \geq 1 \}} + \beta_i Y(t_j) \big), $$
is in $L_r(\mathbb{P})$. Hence,  by dominated convergence,   (\ref{eq:EncaminoaleyConjunta(1)}) converges towards (\ref{eq:leyConjunta(c)})  as $\epsilon \downarrow 0$. On the other hand, since $( \xi_{\epsilon,1}^{(3)}(t_j),  \hat{\xi}_{\epsilon ,1}^{(3)}({t_j})) \rightarrow ( \xi_{t_j}^{(3)} ,  \hat{\xi}^{(3)}_{t_j} )$ as $\epsilon  \downarrow 0$, we obtain the desired result. 
\end{proof}
\subsubsection{The distribution of $(B, \hat{B})$ and proof of Proposition \ref{theorem:leyConjunta} }\label{section:joint_LevyLevyR}
The last ingredient needed to define the joint distribution of $(\xi , \hat{\xi})$  is the joint distribution of a Brownian motion $B$ and its reinforced version $\hat{B}$, that we denote as $(B, \hat{B})$. Recall from \cite{BertoinUniversality} that $\hat{B}$ has the same law as the solution to the SDE
\begin{equation} \label{eq:SDEReforzada}
    \dd X_t = \dd B_t + \frac{p}{t} X_t \dd t, 
\end{equation}
and that $X$ can be written explicitly in terms of the stochastic integral (\ref{formula:representationNRBM}) with respect to the driving Brownian motion $B$.
We also recall from (\ref{formula:covarianceNRBM}) that for $0<s,t <T$ the covariance of $\hat{B}$ can be expressed in terms of the Yule Simon process as follows:  
\begin{equation} \label{formula:covarRefrozado}
    \esp{\hat{B}_{t} \hat{B}_s }{} = \frac{(t \vee s)^p (t \wedge s)^{1-p}}{1-2p} = T (1-p) \esp{Y(t/T) Y(s/T)}{},
\end{equation}
and for later use, we observe that 
\begin{equation}\label{formula:covarRefrozadoMB}
    {(t \wedge s)^{1-p}s^p} = T (1-p) \esp{\indic{\{ Y(t/T) \geq 1 \}} Y(s/T) }{}. 
\end{equation}
We stress that the right-hand side in the previous display do not depend on the choice of $T$. The proof  of this identity is a consequence of the representation  (\ref{formula:RepresentacionYuleSimon}) of $Y$ in terms of a standard Yule process and an independent uniform random variable. 
\begin{definition} \label{lemma:FourierParejaGaus}
 Let $(B,\hat{B})$ be a pair of Gaussian processes and fix a parameter $0<p< 1/2$. We  say that the pair $(B,\hat{B})$ has the law of a Brownian motion with its reinforced version if the respective covariances are given by 
\begin{equation} \label{definition:covariances}
   \esp{B_tB_s}{} = (t \wedge s), \quad \quad  \quad  \esp{{B}_t \hat{B}_{s}   }{} = (t \wedge s)^{1-p}s^p, \quad \quad  \quad   \esp{\hat{B}_{t} \hat{B}_s }{} = \frac{(t \vee s)^p (t \wedge s)^{1-p}}{1-2p}, 
\end{equation}
for any $s, t \in \mathbb{R}^+$. 
\end{definition}
Let us briefly explain where this definition comes from: for fixed $p$, by   \cite[Theorem 1.1]{InvariancePrinciplesNRRW} the law of the pair $(B , \hat{B})$ is universal, in the sense that it is the weak joint scaling limit of random walks paired with its reinforced version with parameter $p$ for $p<1/2$, when the typical step is in $L_2(\mathbb{P})$.  For more details, we refer to \cite{BertoinUniversality,InvariancePrinciplesNRRW}. \par 
Given a fixed Brownian motion $B$, it is clear that we can not expect to have an explicit  construction of the reinforced version  $\hat{B}$ in terms of $B$  similar to the one performed for  $(J,\hat{J})$.  However,  we can make use of the SDE (\ref{eq:SDEReforzada}) to get an explicit construct of  $(B,\hat{B})$  with the right covariance structure. This can be easily achieved as follows: first, let $W$ be an independent copy of $B$; if we set
\begin{equation} \label{construccionExplicitaNRBM}
    \beta_t := (1-p) B_t + \sqrt{1-(1-p)^2} W_t, 
\end{equation}
then, $B$ and $\beta$ are two Brownian motions with $ \esp{B_t \beta_s}{} = (1-p) (t \wedge s)$. If we let $\hat{B}$ be the solution to the SDE,  
\begin{equation} \label{formula:laSdeNRBM}
    \dd \hat{B}_t = \dd \beta_t + \frac{p}{t}\hat{B}_t \dd t,
\end{equation}
$\hat{B}$ has the law of a noise reinforced Brownian motion with reinforcement parameter $p$, and  can be written explicitly as 
$\hat{B}_t = t^{p} \int_0^t s^{-p} \dd \beta_s$. Moreover, it readily follows  that the covariance of the pair of Gaussian processes $(B, \hat{B})$ satisfies (\ref{definition:covariances}). The decorrelation applied for constructing $\beta$ is playing the role of the thinning in the construction of $(J, \hat{J})$. \par 
Finally, we will need for the proof of Proposition \ref{theorem:leyConjunta} the following representation of the characteristic function of the finite-dimensional distributions of the pair $(B,\hat{B})$ in terms of the Yule-Simon process:  
\begin{lemma} \label{lemma:charFunMB_NRMB}
Let $(B, \hat{B})$ be a Brownian  motion with its reinforced version for a memory parameter $p < 1/2$. For all $k \geq 1$,  $\lambda_1, \dots , \lambda_k, \beta_1, \dots \beta_k$  real numbers and $0<t_1< \dots < t_k<t$, we have 
\begin{align}\label{eq:leyConjuntaMB}
    &\esp{\exp\bigg\{ i \sum_{j=1}^k \big( \lambda_j  B_{t_j} + \beta_j  \hat{B}_{t_j} \big)\bigg\}  }{} 
    = \nonumber \\
    &\exp \bigg\{- { t \cdot p}  \esp{ \frac{q^2}{2} \Bigg( \sum_{j=1}^k \lambda_j \indic{\{ Y(t_j/t) \geq 1 \}} \Bigg)^2 }{} - t \cdot  {(1-p)} \esp{ \frac{q^2}{2} \Bigg( \sum_{j=1}^k \big( \lambda_j \indic{\{ Y(t_j/t) \geq 1 \}} + \beta_j Y(t_j/t) \big) \Bigg)^2 }{} \bigg\}. 
\end{align}  
\end{lemma}
\begin{proof}
Since  NRBM satisfies the same scaling property of Brownian Motion (see page 3 of \cite{BertoinUniversality}), from (\ref{definition:covariances}) we deduce $(B_{tc}, \hat{B}_{tc} )_{t\in \mathbb{R}^+} \overset{\scr{L}}{=}(c^{1/2} B_{t}, c^{1/2} \hat{B}_{t} )_{t\in \mathbb{R}^+}$. Hence, as usual we can suppose that $t_k < 1$ and we take  $t:=1$. To simplify notation we also suppose that $q = 1$. Now, the left hand side of \eqref{eq:leyConjuntaMB} writes 
\begin{align*} 
& \esp{\exp\bigg\{ i \sum_{j=1}^k \big( \lambda_j  B_{t_j} + \beta_j  \hat{B}_{t_j} \big) \bigg\}  }{} \nonumber \\
&= \exp\bigg\{ -\frac{1}{2} \sum_{i,j} \lambda_i \lambda_j Cov(B_{t_i}, B_{t_j}) - \frac{1}{2}\sum_{i,j} \beta_i \beta_j Cov(\hat{B}_{t_i} , \hat{B}_{t_j} ) -  \sum_{i,j} \lambda_i \beta_j Cov(B_{t_i}, \hat{B}_{t_j} ) \bigg\} \\
&= \exp \bigg\{ 
    -\frac{1}{2} \esp{ \left( \sum_{j=1}^k \lambda_j \indic{\{Y(t_j) \geq 1 \}}  \right)^2  + (1-p) \left( \sum_{j=1}^k \beta_j Y(t_j) \right)^2 + 2 {(1-p)}\sum_{ i,j } \lambda_i \beta_j \indic{\{ Y(t_i) \geq 1\}} Y(t_j)  }{}
    \bigg\},  
\end{align*}
where we used respectively for each one of the covariances in order of appearance that: the first jump time of a Yule-Simon process is uniformly distributed, (\ref{formula:covarRefrozado}) and (\ref{formula:covarRefrozadoMB}). However, this is precisely the right hand side of (\ref{eq:leyConjuntaMB}). 
\begin{toexclude}
If $U$ is a uniform random variable in $[0,1]$, we can write 
\begin{equation*}
    \frac{1}{2} \sum_{i,j} \lambda_i \lambda_j Cov(B_{t_i}, B_{t_j})  = \frac{1}{2} \esp{ \left( \sum_{i=1}^k \lambda_i \indic{\{ U \leq t_i \}} \right)^2 }{}
\end{equation*}
while by (\ref{formula:covarRefrozado}) 
\begin{equation} \label{formula:CovarianzaNRBM}
    \frac{1}{2} \sum_{i,j} \beta_i \beta_j Cov(\hat{B}_{t_i} , \hat{B}_{t_j} ) =  \frac{(1-p)}{2}\esp{ \left( \sum_{i=1}^k \beta_i Y(t_i)  \right)^2 }{}
\end{equation}
and finally, by (\ref{formula:covarRefrozadoMB})
\begin{equation*}
    \sum_{i,j} \lambda_i \beta_j Cov(B_{t_i} , \hat{B}_{t_j}) = {(1-p)} \esp{ \sum_{i,j} \lambda_i \beta_j \indic{\{Y(t_i) \geq 1) \}} Y(t_j) }{}.
\end{equation*}
\end{toexclude}
\end{proof}
Now that all the ingredients have been introduced, we  define the law of $(\xi, \hat{\xi})$.\medskip \\
$\Game$ \textit{Recipe for reinforcing Lévy processes: }consider a starting Lévy process $\xi$ with triplet $(a,q^2,\Lambda)$ and denote by $\xi_t = at + q B_t + J_t$ for $t \geq 0$ its Lévy Itô decomposition, where $B$ and $J$ are respectively a Brownian motion and a Lévy process with triplet $(0,0, \Lambda)$. Further,  fix $p \in (0,1)$ an admissible parameter for the triplet, denote the  jump measure of $\xi$ by $\mathcal{N} = \sum \delta_{(u, \Delta \xi_u)}$   and consider the NRPPP $\hat{\mathcal{N}}$ with characteristic measure  $\Lambda$ and reinforcement parameter $p$ as constructed in (\ref{definition:thereinfocedversion_ofPPP}) in terms of $\mathcal{N}$. Denote by $\hat{J} := \indic{(-1,1)}x * \hat{\mathcal{N}}^{(sc)} +   \indic{(-1,1)^c}x * \hat{\mathcal{N}}$ the corresponding NRLP of characteristics $(0,0,\Lambda,p)$ and finally, consider a NRBM $\hat{B}$ independent of $(J,\hat{J)}$, such that  $(B,\hat{B})$ has the law of a Brownian motion with its reinforced version -- for example by proceeding as in (\ref{formula:laSdeNRBM}).

\begin{definition}\label{definition:Levy_LevyR}
 We call the noise reinforced Lévy process  $\hat{\xi}_t := at + q   \hat{B}_t + \hat{J}_t$ for $t \geq 0$ of characteristics $(a,q^2 ,\Lambda, p)$  the noise reinforced version of ${\xi}$, the unicity only holding in distribution. From now on, every time we consider a pair $(\xi, \hat{\xi})$, it will be implicitly assumed that $\hat{\xi}$ has been constructed by the procedure we just described in terms of $\xi$.
\end{definition}

Let us now conclude the proof of Proposition \ref{theorem:leyConjunta}.  

\begin{proof}[Proof of Proposition \ref{theorem:leyConjunta}]
If $\Psi$ is the characteristic exponent of $\xi$, we can write 
\begin{equation*}
    \Psi(\lambda) = i a \lambda - \frac{1}{2} q^2 \lambda^2 + \Phi^{(2)}(\lambda) + \Phi^{(3)}(\lambda), 
\end{equation*}
for  $\Phi^{(2)}$, $\Phi^{(3)}$ defined respectively by (\ref{equation:charExpoSaltosGrandes}) and  (\ref{equation:charExpoSaltosChicos}). Recalling  the independence between the pairs $(B,\hat{B})$, $(\xi^{(2)}, \hat{\xi}^{(2)})$, $(\xi^{(3)}, \hat{\xi}^{(3)})$,   the proof of Proposition \ref{theorem:leyConjunta}  now follows from  Lemmas  \ref{lemma:FourierPareja(1)}, \ref{lemma:FourierPareja(c)}, \ref{lemma:charFunMB_NRMB} and the previous decomposition for the characteristic exponent $\Psi$.
\end{proof}

\indent From  the construction of $({\mathcal{N}} , \hat{\mathcal{N}})$,  we can sketch a sample path of $(\xi, \hat\xi)$, where the jumps that are not appearing on the path of $\hat{\xi}$ are precisely the ones deleted by the thinning: 
\begin{figure}[htbp!]
   \centering
   \includegraphics[scale=0.85]{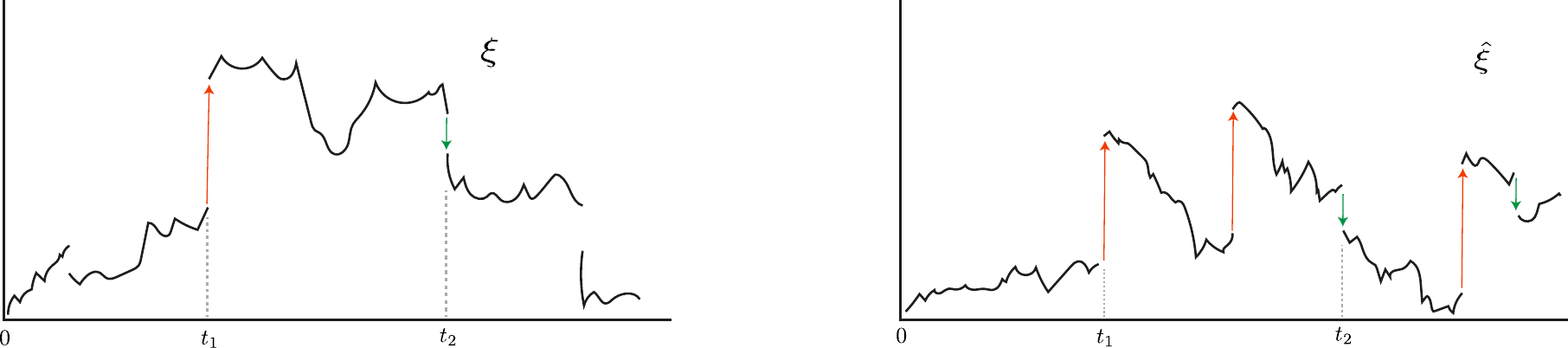}
  \caption{Sample path of a Lévy process and its reinforced version.}
\end{figure}

\subsection{Proof of Theorem \ref{thm:convergenciaConjunta} } \label{subsection:PruebajointConvergence}

 Let us   outline the proof of Theorem \ref{thm:convergenciaConjunta}. First,   by \eqref{identity:extension}, it suffices to prove the convergence in $[0,1]$ and we therefore work with $\xi = (\xi_t)_{t \in [0,1]}$. Next, since we are working in $D^2[0,1]$, it suffices to establish tightness coordinate-wise to obtain tightness for the sequence of pairs.  The first coordinate in (\ref{equation:joint convergence}) converges a.s. towards $\xi$ in $D[0,1]$ (and in particular is tight)  and hence it remains to establish tightness for  the sequence of reinforced $n$-skeletons. This is the content of Section \ref{subsection:tightness} and more precisely, of Proposition \ref{theorem:weakConvergenceOfSkeletons}. This is achieved by means of the celebrated Aldous tightness criterion and our arguments  rely on the discrete counterpart of the remarkable martingale from Proposition \ref{proposition:laMartingala}. This discrete martingale is  introduced in Lemma \ref{LaMartingala} and  we recall from \cite{InvariancePrinciplesNRRW, Bertenghi} its main features. This is the content of  Section \ref{subsection:BercuMartingale}. Finally, the joint convergence in the sense of finite-dimensional distributions towards $(\xi, \hat{\xi})$ is proved in Proposition \ref{thm:convergenciaConjuntaFDD}, by establishing the convergence of the corresponding characteristic functions. 

\subsubsection{The martingale associated with a noise reinforced random walk}\label{subsection:BercuMartingale}
   
 $\circ$ \textit{The elephant random walk and its associated martingale}. Let us start with some historical context. In \cite{Bercu}, Bercu was interested in establishing asymptotic convergence results for a particular random walk with memory, called the elephant random walk. This process is defined as follows: for a fixed $q \in (0,1)$ that we still call the reinforcement parameter, we set $\mathcal{E}_0 := 0$ and let $Y_1$ be a random variable with $Y_1 \in \{-1,1\}$. Then, the position of our elephant at time $n=1$ is given by  $\mathcal{E}_1 = Y_1$ and for $n \geq 2$,  it is  defined recursively by the relation  $\mathcal{E}_{n+1} := \mathcal{E}_{n} + Y_{n+1}$, for $Y_{n+1}$ constructed by  selecting uniformly at random one of the previous increments $\{Y_1 \, \dots \, Y_n \}$, and changing its sign with probability $1-q$. The analysis of  Bercu relies on  a  martingale associated to the elephant random walk, defined as  $M_1 = \mathcal{E}_1$ and for $n \geq 2$, as 
  \begin{equation} \label{Bercu_ElephantMartingale}
      M_n := \hat{a}_n \mathcal{E}_n, \quad \quad \text{ for  }  \hat{a}_n := \frac{\Gamma(n)\Gamma(2q) }{\Gamma(n+2q-1)}, 
  \end{equation}
   and where $\Gamma$ stands for the Euler-Gamma function. 
  This martingale had already made its appearance in the literature in Coletti, Gava, Schütz \cite{GavaSchuetzColetti}.  As was pointed out by Kürsten \cite{Kuersten}, the key is that  when $q \in [1/2,1)$,   the elephant random walk is a version of the noise reinforced random walk when the typical step $X$ has distribution $\proba{X=1}{}= \proba{X=-1}{}=1/2$ with memory parameter $p = 2q-1$.  
\medskip \\   
\mbox{}\hspace{4mm}Getting back to our setting, we maintain the notation introduced at the beginning of Section \ref{section:jointlaw} for the noise reinforced random walk for a memory parameter $p \in (0,1)$. Our first observation is  that the martingale (\ref{Bercu_ElephantMartingale})  associated to the elephant random walk is still a martingale in our setting --  we stress  that the reinforcement parameter $q$ in  \cite{Bercu} corresponds to the  parameter $p= 2q-1$ in our context. This martingale plays a fundamental role in our reasoning, and also played a central role in \cite{InvariancePrinciplesNRRW, Bertenghi}. More precisely, let  {$a_1 := 1$} and for $n \geq 2$ we set
 \begin{equation} \label{equation:coefBercuMartingale}
     a_n :=  \frac{\Gamma(n)  }{\Gamma(n+p)} = \prod_{k=1}^{n-1} \gamma_k^{-1},
 \end{equation}
 for  $\gamma_n := \frac{n+p}{n}$. We write $\scr{F}_n := \sigma (\hat{X}_1, \dots , \hat{X}_n )$ the filtration generated by the reinforced steps. The following lemma is taken from \cite{InvariancePrinciplesNRRW}.
 \begin{lemma}
 \label{LaMartingala}
 \emph{\cite[Proposition 2.1]{InvariancePrinciplesNRRW}}
 Suppose that the typical step $X$ is centred and in $L_2(\mathbb{P})$. Then, the process $M$ defined by $M_0 = 0$ and  $M_n = a_n \hat{S}_n$ for $n \geq 1$ is a square-integrable martingale with respect to the filtration $(\scr{F}_n)$.
 \end{lemma}
In order to establish  tightness for our sequence of  reinforced skeletons, we will make use of the explicit form of the predictable quadratic variation  $\langle M,M \rangle$ of this  martingale, which is the process defined as $\langle M,M \rangle_0 = 0$ and 
\begin{equation*}
     \langle M,M \rangle_n 
    = \sum_{k=1}^n \esp{\left( \Delta M_k \right)^2 | \scr{F}_{k-1} }{}, \quad \quad n \geq 1. 
\end{equation*}
In this direction, we introduce:
\begin{equation*}
    \hat{V}_n :=  \hat{X}_1 ^2 + \dots  +  \hat{X}_n ^2, \quad \quad n \geq 1,
\end{equation*}
with $\hat{V}_0 = 0$. The following lemma is also taken from \cite{InvariancePrinciplesNRRW} and was the main tool for establishing the invariance principles proven in that work. 
\begin{lemma} \label{LaVariacionCuad} 
\emph{\cite[Proposition 2.1]{InvariancePrinciplesNRRW}}
The predictable quadratic variation process $\langle M,M \rangle$ is given by $\langle M , M \rangle_0 =0$  and for $n \geq 1$, 
\begin{equation}
    \langle M,M \rangle_n = \sigma^2 +  \sum_{k=2}^n a^2_k \left( (1-p) \sigma^2 - p^2 \frac{\hat{S}_{k-1} ^2}{(k-1)^2} + p \frac{\hat{V}_{k-1}}{k-1} \right),   \label{quadVar}
\end{equation}
where the sum should be considered  null for $n=1$. 
\end{lemma}

\subsubsection{Proof of tightness}\label{subsection:tightness}
We stress that the f.d.d. convergence of the sequence of reinforced skeletons towards a NRLP $\hat{\xi}$ of characteristics $(a,q^2,\Lambda,p)$ was already established in Theorem 3.1 of \cite{BertoinNRLP}. 
\begin{proposition}\label{theorem:weakConvergenceOfSkeletons} Let $p<1/2$ be an admissible memory parameter for the triplet $(a,q^2,\Lambda)$. Then, 
the sequence of laws associated to the reinforced skeletons 
\begin{equation} \label{weekConvergence1}
    \big\{ \big( \hat{S}^{(n)}_{\lfloor n t \rfloor} \big)_{t \in [0 , 1]} \, : \,  n \in \mathbb{N} \big\} \quad \text{is tight in } D[0,1].
\end{equation}
Therefore, the convergence $\big( \hat{S}^{(n)}_{\lfloor n t \rfloor} \big)_{t \in [0 , 1]} \overset{\scr{L}}{\longrightarrow} (\hat{\xi}_t)_{t \in [0,1]}$ holds in $D[0,1]$. 
\end{proposition}
\noindent 
The reason behind the restriction $p < 1/2$ and why we don't expect our proof to work for $p \geq 1/2$ is explained in Remark \ref{remark:casopgeq2}, at the end of the proof.
\subsubsection*{Proof of Proposition \ref{theorem:weakConvergenceOfSkeletons} for   centred $\xi$ with compactly supported Lévy measure.} \label{subsection:proofPropCasoEspecial}
Until further notice, we   restrict our reasoning to the case when $\xi$ is a centred Lévy process, with Lévy measure $\Lambda$ concentrated in $[-K,K]$ for some $K>0$, and without loss of generality we  suppose that $K=1$. In consequence, $\xi$ has finite moments of any order and we set  $\sigma^2_n := \mathbb{E} [ \xi_{1/n}^2] = \mathbb{E}[(\hat{X}^{(n)}_1)^2]$. Notice that under our standing hypothesis, $\hat{\xi}$ writes 
\begin{equation*}
    t^{p} \int_0^t s^{-p} \dd B_s + \hat{\xi}^{(3)}_t, \quad \quad t \in [0,1], 
\end{equation*}
for  some Brownian motion ${B}$ independent of $\hat{\xi}^{(3)}$.  Further, remark that under our restrictions, the family of discrete skeletons $(\hat{S}^{(n)}_k)$, for $n \in \mathbb{N}$, have typical steps centred  and  in $L_2(\mathbb{P})$.  Consequently, we can make use of Lemma \ref{LaMartingala}. \par
Next,  to establish  Proposition \ref{theorem:weakConvergenceOfSkeletons}  under our current restrictions, we claim that it would be enough  to show that the following convergence holds, 
\begin{equation}
    \left(   n^p a_{\lfloor n t \rfloor } \hat{S}^{(n)}_{\lfloor n t \rfloor} \right)_{t \in [0 , 1]} \overset{\scr{L}}{\longrightarrow} \big( t^{-p} \hat{\xi}_t\big)_{t \in [0, 1]}, 
    \label{weekConvergenciaMartingalas}
\end{equation}
 where now, the sequence on the left-hand side of (\ref{weekConvergenciaMartingalas}) is a sequence of martingales, while the process on the right hand side is  the martingale  introduced in Proposition \ref{proposition:laMartingala}, viz. 
\begin{equation*}
    N_t := \int_0^t s^{-p}\dd B_s +  t^{-p} \hat{\xi}^{(3)}_t, \quad \quad t \in [0,1]. 
\end{equation*}
 Indeed, for each $n$, let $M^n$ be the continuous time version of  martingale of   Lemma \ref{LaMartingala} associated with the $n$-reinforced skeleton  $(\hat{S}^{(n)}_k)_{k \in \mathbb{N}}$, i.e. 
\begin{equation*}
    M^n_{\lfloor nt \rfloor}  = a_{\lfloor nt \rfloor} \hat{S}^{(n)}_{\lfloor nt \rfloor},   \quad t \geq 0, 
\end{equation*}
and remark that by Lemma \ref{LaVariacionCuad}, the predictable quadratic variation of $M^{n}_{\lfloor n \cdot \rfloor}$ is given by
\begin{equation*}
 \langle M^n,M^n \rangle_{\lfloor nt \rfloor} = \indic{\{ 1/n \leq t \}}  \sigma_n^2 + \sum_{k=2}^{\lfloor nt \rfloor} a^2_k \left( (1-p) \sigma^2_n - p^2 \frac{(\hat{S}_{k-1}^{(n)}) ^2}{(k-1)^2} + p \frac{\hat{V}_{k-1}^{(n)} }{k-1} \right).     
\end{equation*}
It follows that for each $n \in \mathbb{N}$, the following process 
\begin{equation*}
    N^n_t :=   n^{p} M^{n}_{\lfloor n t \rfloor} =   n^p a_{\lfloor nt \rfloor} \hat{S}^{(n)}_{\lfloor nt \rfloor}, \quad {t \geq 0}, 
\end{equation*}
is also a martingale, and its  predictable quadratic variation writes: 
\begin{equation}
\label{varquadN}
 \langle N^n,N^n \rangle_{t} =  n^{2p} \langle M^{n}, M^n \rangle_{\lfloor nt \rfloor}  = \indic{\{ 1/n \leq t \}}   n^{2p} \sigma^2_n  +  n^{2p} \sum_{k=2}^{\lfloor nt \rfloor} a^2_k \left( (1-p) \sigma^2_n - p^2 \frac{(\hat{S}_{k-1}^{(n)}) ^2}{(k-1)^2} + p \frac{\hat{V}_{k-1}^{(n)}}{k-1} \right). 
\end{equation}
Moreover, by Stirling's formula, we have 
\begin{equation*}
    a_n = \frac{\Gamma(n)  }{\Gamma(n+p)} \sim n^{-p},  \quad \text{ as } n \uparrow \infty, 
\end{equation*}
which gives:  
\begin{equation*}
      n^p a_{\lfloor n t \rfloor } \sim t^{-p},  \quad \text{ as } n \uparrow \infty.
\end{equation*}
This yields  the claimed equivalence between (\ref{weekConvergence1}) and (\ref{weekConvergenciaMartingalas})  under our current restrictions for $\xi$. For technical reasons, we shall  prove first that  the convergence of the martingales $(N^{n})$ towards $N$ holds in  the interval $[\epsilon , 1]$,  for any $\epsilon > 0$.  This leads us to the following lemma: 
\begin{lemma} \label{lemma:TightnessMartingales} For any $\epsilon >0$, the sequence  $(N^n_t)_{t \in [\epsilon,1]}$ for $n \in \mathbb{N}$ is tight.
\end{lemma}
\begin{proof}
We denote by $(\scr{F}^n_t)$ the natural filtration of $N^n$. By Aldous's tightness criterion (see for e.g. Kallenberg \cite{Kallenberg_FoundationsOfModernProbability} Theorem 16.11), it is enough to show that for any sequence $(\tau_n)$ of (bounded) $(\scr{F}^n_t)$-stopping times in $[\epsilon, 1]$ and any sequence of positive real numbers $(h_n)$ converging to 0, we have  
\begin{equation*}
    \lim_{n \uparrow \infty} |N^n_{\tau_n + h_n} -  N^n_{\tau_n}| = 0, \quad \text{ in probability.} 
\end{equation*}
By Rebolledo's Theorem (see e.g. Theorem 2.3.2 in Joffre and Metivier \cite{JoffreMetivier_ConvergenceOfSemimartingales} ) it's enough to show that the sequence of associated predictable quadratic variations $(\langle N^n, N^n \rangle )$ satisfies Aldous's tightness criterion, i.e. that 
\begin{equation*}
    \lim_{n \uparrow \infty} \langle N^n, N^n \rangle _{\tau_n + h_n} -  \langle N^n, N^n \rangle _{\tau_n} = 0, \quad \text{ in probability}. 
\end{equation*}
In this direction, by (\ref{varquadN}), we have 
\begin{align}
 \langle N^n, N^n \rangle _{\tau_n + h_n} -  \langle N^n, N^n \rangle _{\tau_n} &=  n^{2p} \sum_{k= \lfloor n\tau_n \rfloor +1 }^{\lfloor n (\tau_n + h_n) \rfloor} a^2_k \left( (1-p) \sigma^2_n - p^2 \frac{(\hat{S}_{k-1}^{(n)}) ^2}{(k-1)^2} + p \frac{\hat{V}_{k-1}^{(n)}}{k-1} \right) \nonumber \\
    & \hspace{-20mm} \leq  (1-p)  n^{2p} \sum_{k= \lfloor n\tau_n \rfloor +1 }^{\lfloor n (\tau_n + h_n) \rfloor} a^2_k  \sigma^2_n  +    p \cdot  n^{2p} \sum_{k= \lfloor n\tau_n \rfloor +1 }^{\lfloor n (\tau_n + h_n) \rfloor} a^2_k \frac{\hat{V}_{k-1}^{(n)}}{k-1}, \label{DiferenciaCorchetes}
\end{align}
and it remains to show that both terms in the right hand side converge to 0 in probability as $n \uparrow \infty$. The key now is in the asymptotic behaviour of the series $\sum_{k=1}^n a^2_k$. As was already pointed out in \cite{Bercu}, for $p \in (0,1/2)$, we have   
\begin{equation}
    \label{asimptot1}
    \lim_{n \uparrow \infty} n^{2p-1} \sum_{k=1}^n a^2_k = \frac{1}{1-2p}. 
\end{equation}
Furthermore,  since the Lévy measure of $\xi$ is compactly supported,  it holds that 
\begin{equation}
    \label{asimptotLevy}
    \sigma^2_n = \esp{(\hat{X}^{(n)}_1 )^2}{} = \esp{\xi_{1/n}^2}{} = O(1/n),  \quad \text{ as } n \uparrow \infty.
\end{equation}
Now,  from (\ref{asimptot1}) and (\ref{asimptotLevy}) it follows that 
\begin{equation*}
   \lim_{n \uparrow \infty}  (1-p)  n^{2p-1} \sum_{k= \lfloor n\tau_n \rfloor +1 }^{\lfloor n (\tau_n + h_n) \rfloor} a^2_k = 0 \quad \text{ a.s.}
\end{equation*}
and a fortiori in probability, which entails that the first term in (\ref{DiferenciaCorchetes}) converges in probability to 0 \\ as $n \uparrow \infty$. In order to show that the second term in (\ref{DiferenciaCorchetes}) also converges in probability to 0, we need to proceed more carefully. First, since $\tau_n \in [\epsilon, 1]$, we can bound the second term in (\ref{DiferenciaCorchetes}) by
\begin{equation*}
    n^{2p} \sum_{k= \lfloor n\tau_n \rfloor +1 }^{\lfloor n (\tau_n + h_n) \rfloor} a^2_k \frac{\hat{V}_{k-1}^{(n)}}{k-1}  
    \leq n^{2p} \frac{ \sup_{ \lfloor n \epsilon \rfloor  \leq  k \leq n   } \hat{V}_k^{(n)} }{\lfloor n \epsilon \rfloor } \sum_{k= \lfloor n\tau_n \rfloor +1 }^{\lfloor n (\tau_n + h_n) \rfloor} a^2_k.
\end{equation*}
Next, since $\frac{n^{2p}}{\lfloor n \epsilon \rfloor} \sim n^{2p-1}\epsilon^{-1}$, in order to proceed as before we need to show that 
\begin{equation*}
    \sup_{ \lfloor n \epsilon \rfloor  \leq  k \leq n} \hat{V}_k^{(n)} = O(1),  \quad \text{in probability as }n \uparrow \infty,  
\end{equation*}
i.e. that the sequence is stochastically bounded. To do so we proceed as follows: for each $n$,  notice that the process 
\begin{equation*}
    \hat{V}^{(n)}_k = (\hat{X}^{(n)}_1)^2 + \dots + (\hat{X}^{(n)}_k)^2, \quad k \geq 1, 
\end{equation*}
is the reinforced version of the random walk
\begin{equation*}
    {V}^{(n)}_k = ({X}^{(n)}_1)^2 + \dots + ({X}^{(n)}_k)^2, \quad k \geq 1,
\end{equation*}
where $(\hat{X}^{(n)}_i)^2$ are i.i.d. variables with law $\xi_{1/n}^2$. In order to have a centred noise reinforced random walk, for $k \geq 1$ set  $ \hat{Y}^{(n)}_k :=(\hat{X}^{(n)}_k)^2 - \mathbb{E}[{\xi_{1/n}^2}]$ and we introduce: 
\begin{align*}
    \hat{W}_k^{(n)} &:= \hat{V}^{(n)}_k - k \esp{\xi_{1/n}^2}{} \\
     &= \left(  (\hat{X}^{(n)}_1)^2 - \esp{\xi_{1/n}^2}{}  \right)  + \dots + \left(  (\hat{X}^{(n)}_k)^2 - \esp{\xi_{1/n}^2}{} \right)
     = \hat{Y}^{(n)}_1 + \dots + \hat{Y}^{(n)}_k, 
\end{align*}
 Now, the process $(\hat{W}^{(n)}_k)_{k \in \mathbb{N}}$ is the noise reinforced version of the centred random walk defined for $k \geq 1$ as 
\begin{align*}
     {W}_k^{(n)} &= {V}^{(n)}_k - k \esp{\xi_{1/n}^2}{} \\
     &= \left(  ({X}^{(n)}_1)^2 - \esp{\xi_{1/n}^2}{}  \right)  + \dots + \left(  ({X}^{(n)}_k)^2 - \esp{\xi_{1/n}^2}{} \right)
     = {Y}^{(n)}_1 + \dots + {Y}^{(n)}_k, 
\end{align*}
where $(Y_i^{(n)})_{i \in \mathbb{N}}$ are i.i.d. with law $\xi_{1/n}^2 - \mathbb{E}[{\xi_{1/n} ^2}]$. We can now apply Corollary 4.3 from \cite{BertoinUniversality}  to $W^{(n)}$:  recalling that $\sigma_n^2 = \mathbb{E}[{\xi_{1/n}^2 }] = O(1/n)$, we have 
\begin{align*}
    \esp{\sup_{k \leq n} \left( \hat{V}^{(n)}_k\right)^2 }{} = \esp{\sup_{k \leq n} \left( \hat{W}^{(n)}_k + k \sigma^2_n\right)^2}{} 
    &\leq C \esp{\sup_{k \leq n} \left( \hat{W}^{(n)}_k \right) ^2  }{} + O(1)\\
    & \leq C' \esp{\left( \xi_{1/n}^2 - \esp{\xi_{1/n}^2 }{}  \right)^2 }{} n + O(1). 
\end{align*}
Once again, since $\xi$ has compactly supported Lévy measure, $\mathbb{E}[{\xi_{1/n}^2 }]$ and $\mathbb{E}[{\xi_{1/n}^4 }]$ are both $O(1/n)$ as $n \uparrow \infty$ and we deduce that 
\begin{equation*}
    \esp{\sup_{k \leq n} \left( \hat{V}^{(n)}_k\right)^2 }{}  = O(1), \quad \text{ as } n \uparrow \infty.
\end{equation*}
Hence, by Markov's inequality the sequence $(\sup_{k \leq n} \hat{V}^{(n)}_k  )_n$ is $O(1)$ in probability and we can conclude as before by bounding as follows for $L > 0$: 
\begin{align*}
  & \proba{   n^{2p} \frac{ \sup_{ \lfloor n \epsilon \rfloor  \leq  k \leq n   } V_k^{(n)} }{\lfloor n \epsilon \rfloor } \sum_{k= \lfloor n\tau_n \rfloor +1 }^{\lfloor n (\tau_n + h_n) \rfloor} a^2_k > \eta}{} 
  \\
  & \hspace{30mm}\leq \proba{\sup_{k \leq n}  \hat{V}^{(n)}_k >L}{} + \proba{L  \frac{ n^{2p}} {\lfloor n \epsilon \rfloor } \sum_{k= \lfloor n\tau_n \rfloor +1 }^{\lfloor n (\tau_n + h_n) \rfloor} a^2_k  > \eta}{}.
\end{align*}
\end{proof}
We shall now conclude the proof of Proportion \ref{theorem:weakConvergenceOfSkeletons} under our standing assumptions, and in this direction recall our discussion prior to Lemma \ref{lemma:TightnessMartingales}. To extend the convergence to the interval $[0,1]$ we shall use a truncation argument similar to the one employed in Section 4.3 of \cite{BertoinUniversality}. For each $\epsilon>0$, we have $(N^{n}_t)_{t \in [\epsilon , 1]} \overset{\scr{L}}{\rightarrow} (N_t)_{t \in [\epsilon , 1]}$   and since $(N_{t+\epsilon})_{t \in [0,1] } \rightarrow (N_{t})_{t \in [0,1]}$ by right continuity (extending $N_{\cdot + \epsilon}$ for $t \in [1-\epsilon, 1]$ identically as the constant $N_1$ ),
we deduce by metrisability of the weak convergence that there exists some sequence $(\epsilon(n))_{n \in \mathbb{N}}$, converging to 0 slowly enough as $n \uparrow \infty$ such that $(N^{n}_t)_{t \in [\epsilon(n) ,1]} \overset{\scr{L}}{\rightarrow} (N_t )_{t\in [0,1]}$ and we only need to show: 
\begin{equation} \label{requirement:terminoNeglegible}
    \sup_{s \leq \epsilon(n)} n^p a_{\lfloor ns \rfloor}\hat{S}^{(n)}_{\lfloor ns \rfloor}  \rightarrow 0, \quad \quad \text{ in probability as } n \uparrow \infty. 
\end{equation}
In this direction, notice the inequality 
\begin{equation*}
    \langle N^n,N^n \rangle_s \leq    n^{2p} \sigma^2_n  +  n^{2p} \sum_{k=2}^{\lfloor ns \rfloor} a^2_k \left( (1-p) \sigma^2_n  + p \frac{\hat{V}_{k-1}^{(n)}}{k-1} \right).
\end{equation*}
Since  $\mathbb{E}[{\hat{V}^{(n)}_k}] = k \sigma^2_n$, an application of Doob's inequality and the previous display yield that, for any $\delta >0$, we have 
\begin{align*}
    \proba{ \sup_{s \leq \epsilon(n)} |n^{p} a_{\lfloor ns \rfloor }\hat{S}^{(n)}_{\lfloor ns \rfloor}| \geq \delta  }{} &\leq \delta^{-2} \esp{\langle N^{(n)},N^{(n)} \rangle_{\epsilon(n)} }{}\\
    &\leq \delta^{-2} n^{2p} \sigma^2_n + \delta^{-2}\esp{  n^{2p} \sum_{k=2}^{\lfloor nt \rfloor} a^2_k \left( (1-p) \sigma^2_n  + p \frac{\hat{V}_{k-1}^{(n)}}{k-1} \right)}{}\\
    &\leq \delta^{-2}  \sigma_n^2 n^{2p} \sum_{n=2}^{\lfloor n \epsilon(n) \rfloor } a^2_k.
\end{align*}
From the asymptotics, 
\begin{equation*}
 \lim_{n \uparrow \infty} n^{2p-1} \sum_{k=1}^{n} a_k^2 =  \frac{1}{1-2p}, \quad \quad \text{and } \quad \quad {\sigma_n^2}{} = O(1/n),    
\end{equation*}
 we deduce that, as $n \uparrow \infty$, the convergence (\ref{requirement:terminoNeglegible}) holds and we can conclude by an application of Lemma  3.31 - VI from Jacod and Shiryaev \cite{JacodShiryaev_LimitTheorems}.
 
\begin{remark}\label{remark:casopgeq2}
\emph{Before proceeding, we point out that our proof no longer works for $p \geq 1/2$: indeed, one might notice that the change in the asymptotic behaviour of the series $\sum^n_{k=1} a^{2}_k$ for $p \geq 1/2$  makes the preceding reasoning unfruitful. Let us be more precise: these series possess three different asymptotic regimes depending on $p$ and  are the reason behind the different regimes appearing in the behaviour of the Elephant random walk, see e.g.  \cite{Bercu}.  More generally, they are behind the three regimes appearing in the invariance principles \cite{BertoinUniversality, InvariancePrinciplesNRRW}. 
 When $p\geq 1/2$, there is no Brownian component and the martingale $t^{-p} \hat{\xi}^{(3)}$ is no longer in $L_2(\mathbb{P})$ because $Y(t) \in L_q(\mathbb{P})$ for $q < 1/p$.  Since $N^n$ is converging weakly towards  $t^{-p}\hat{\xi}^{(3)}_t$ by (\ref{weekConvergenciaMartingalas}), working with the sequence of quadratic variations $\langle N^{n}, N^n \rangle$  might not be the right approach to obtain tightness.}
\end{remark}
\subsubsection*{Proof of Proposition \ref{theorem:weakConvergenceOfSkeletons}, general case.}\label{subsubsection:jointLaw_tightnessGeneral}
Let us start by  introducing some notation. First, if $\hat{\mathcal{N}}$ is the jump measure of $\hat{\xi}$, we will shorten our notation for the compensated integrals  and simply write   $\hat{\xi}_{u,v}^{(3)}(t) := (\indic{\{ u\leq |x| <v \}}x * \hat{\mathcal{N}}^{(sc)})_t$, for $0 \leq u < v$. Hence, for $K>1$, we can write 
\begin{equation*}
    \hat{\xi}_{0,K}^{(3)}(t) =  \hat{\xi}^{(3)}_t + \hat{\xi}_{1,K}^{(3)}(t), \quad t \in [0,1]. 
\end{equation*}
{It will also be  convenient to introduce the following notation for the sums of jumps: for fixed $0<a<b$, we write 
\begin{equation} \label{equation:saltosFiniteVar}
    \Sigma_{a,b} (t) := \sum_{s \leq t} \indic{\{a \leq |\Delta \hat{\xi}_s| < b \}} \Delta \hat{\xi}_s  = \sum_{i} \indic{x_i \in [a,b)} x_i Y_i(t),  \quad \text{ for } t \in [0,1],  
\end{equation}
so that in particular we have $\hat{\xi}^{(2)} = \Sigma_{1,\infty}$. } Next,  if $\xi$   can be decomposed into $\xi = L^{(1)} + L^{(2)} $, for independent Lévy processes $L^{(1)},L^{(2)}$,  we denote  its reinforced skeleton by $\hat{S}^{(n)}(\xi) = (\hat{S}^{(n)}_{ \lfloor nt \rfloor  }(\xi))_{t \in [0,1] }$ and we write: 
\begin{equation*}
    \hat{S}^{(n)}(\xi) =  \hat{S}^{(n)}(L^{(1)}) +  \hat{S}^{(n)}(L^{(2)}) , 
\end{equation*}
for the decomposition that is naturally induced. More precisely,  the two noise reinforced random walks in the right-hand side of the previous display are made with  the same sequence of Bernoulli random variables as $\hat{S}(\xi)$, and just result from decomposing each increment as 
\begin{equation*}
    \Delta^{(n)} \xi_{i} = +\Delta^{(n)} L^{(1)}_{i} + \Delta^{(n)} L^{(2)}_{i}.
\end{equation*}
Now, we proceed by lifting progressively our restriction imposed in \ref{subsection:proofPropCasoEspecial} as follows: \medskip  \\
\textbf{\textit{Step 1:} }First,   if $\xi$ satisfies that  $\xi = M^{ \leq K}$ where $M^{\leq K}$ is the sum of a Brownian motion with diffusion $q$ and a compensated martingale with jumps smaller than $K$, by \ref{subsection:proofPropCasoEspecial}  the following convergence  holds in distribution: 
\begin{equation*}
    \hat{S}^{(n)}\left( M^{\leq K} \right) \overset{\scr{L}}{\rightarrow} q  \hat{B} + \hat{\xi}_{0,K}^{(3)}, \quad \text{as } n \uparrow \infty. {\tag*{(\textit{End of Step 1})}}
\end{equation*}
\textbf{\textit{Step 2:} } If $b$ is a deterministic constant, let $b \cdot Id := (b t : t \geq 0)$ and suppose now that $\xi$ can be written as  $\xi = b \cdot Id + M^{\leq K}$. Then, we can write 
\begin{equation*}
    \hat{S}^{(n)}(\xi) = \hat{S}^{(n)}(b \cdot Id)  +\hat{S}^{(n)}(M^{\leq K}), 
\end{equation*}
where the sequence of processes $(\hat{S}^{(n)}(b \cdot Id) : n \geq 1)$ is deterministic and converges uniformly to the continuous function $b\cdot Id$. Indeed, notice that the reinforcement doesn't affect the drift term since $\hat{S}^{(n)}(b \cdot Id)_t  = b \lfloor nt \rfloor / n$. We deduce from   \cite[Lemma 3.33]{JacodShiryaev_LimitTheorems} that, as $n \uparrow \infty$, we still have 
\begin{equation}\label{equation:step2}
    \hat{S}^{(n)}\left( b \cdot Id + M^{\leq K} \right) = \hat{S}^{(n)}(b \cdot Id)  + \hat{S}^{(n)}\left( M^{\leq K} \right) \overset{\scr{L}}{\rightarrow} b \cdot Id + q  \hat{B} + \hat{\xi}_{0,K}^{(3)}. 
\end{equation}
\[ {\tag*{(\textit{End of Step 2})}} \]
\indent From here, we work with the Lévy process $\xi$ with triplet $(a,q^2,\Lambda)$,   with Lévy-Itô decomposition given by: 
\begin{equation*}
    \xi = a \cdot Id + M^{\leq 1} + \xi^{(3)}, 
\end{equation*}
and we  denote    its jump measure  by $\mathcal{N}$ -- in particular, we have  $\xi^{(3)} = \indic{(-1,1)^c}x * {\mathcal{N}}$.  For any $K > 1$, we can rearrange the triplet by compensating and modifying appropriately the drift coefficient, in such a way that we have: 
\begin{equation*}
    \xi = b_KId + M^{\leq K} + \xi^{\geq K}, 
\end{equation*}
where $\xi^{\geq K}:= \indic{(-K,K)^c}x * {\mathcal{N}}$. Before moving  to \textit{Step 3}, let us make the two following remarks. \par
$\bullet$ First,  notice that for each fixed $n$, $S^{(n)} \left(  \xi^{\geq K} \right) \overset{\mathbb{P}}{\rightarrow} 0$ uniformly in probability as $K \uparrow \infty$. Indeed,  we have 
\begin{align*}
    \mathbb{P} \Big( \sup_{t \in [0,1] }| \hat{S}_{\lfloor nt \rfloor}^{(n)}\left( \xi^{\geq K} \right) |> \epsilon \Big)
    & \leq \proba{ \Delta \xi_t^{\geq K} \neq 0 \text{ for some }t \in [0,1] }{} , 
\end{align*}
where the right-hand side can be written in terms of the jump process $\mathcal{N}$ of $\xi$ as
\begin{equation} \label{equation:step3}
    \mathbb{P} \Big( \mathcal{N}\big(\{ (t,x) \in [0,1]\times \mathbb{R} : |x| \geq K \} \big) \geq 1 \Big) = 1- e^{-\left( (-\infty,K ] \cup [K,\infty) \right) }.
\end{equation}
The right-hand side in the previous display converges  to 0 as $K \uparrow \infty$ and notice that the bound does not depend on $n$. \par 
$\bullet$ Let $\hat{\xi}$ be  the  noise reinforced Lévy process of characteristics $(a,q^2,\Lambda, p)$ and write  its jump measure   by $\hat{\mathcal{N}}$.  Again, we can rewrite $\hat{\xi}$, by compensating appropriately and modifying the drift coefficient, as follows: 
\begin{equation*}
    \hat{\xi} = b_K Id + q  \hat{B} +\hat{\xi}_{0,K}^{(3)} + \Sigma_{K,\infty}. 
\end{equation*}
Arguing as before, we  have the  uniform convergence in probability $b_K Id + q  \hat{B} + \hat{\xi}_{0,K}^{(3)} \overset{\mathbb{P}}{\rightarrow} \hat{\xi}$ as $K \uparrow \infty$,   since, by the description of $\hat{\mathcal{N}}$ given in Definition \ref{proposition:reinforcedPPPbyMarking}, we have  
\begin{equation*}
    \mathbb{P} \Big( \sup_{t \in [0,1]} \left| \Sigma_{K,\infty}(t) \right| \geq \epsilon \Big) \leq   \proba{ \hat{\mathcal{N}}\big(\{ (t,x) \in [0,1]\times \mathbb{R} : |x| \geq K \} \big) \geq 1 }{} = 1-e^{-(1-p) \Lambda \left( (-\infty,K ] \cup [K,\infty) \right) }.
\end{equation*}
\textbf{\textit{Step 3:} } To conclude, for $K> 1$, we write respectively the Lévy process and the corresponding NRLP  without their jumps of size greater than $K$ as 
\begin{equation*}
    \xi^{\leq K} := b_K Id + M^{\leq K}, \hspace{10mm} \text{ and } \hspace{10mm} \hat{\xi}^{\leq K} := b_K Id + q  \hat{B} + \hat{\xi}_{0,K}^{(3)}.
\end{equation*}
  In (\ref{equation:step2}), we already proved that for each fixed $K$, we have 
\begin{equation*}
    \hat{S}^{(n)}(\xi^{\leq K}) \overset{\scr{L}}{\rightarrow} \hat{\xi}^{\leq K}, \text{ as } n \uparrow \infty, \hspace{5mm} \text{  while by our second remark,  it holds that  } \hspace{5mm} \hat{\xi}^{\leq K} \overset{\scr{L}}{\rightarrow} \hat{\xi}, \text{ as } K \uparrow \infty.  
\end{equation*}
Since the convergence in distribution is metrisable, there exists an increasing sequence $(K(n): n \geq 1)$   converging to infinity slowly enough as $n \uparrow \infty$, such that   
\begin{equation*}
    \hat{S}^{(n)} \big( \xi^{ \leq K(n) } \big) \overset{\scr{L}}{\rightarrow} \hat{\xi},  \quad  \quad \text{ as } n \uparrow \infty. 
\end{equation*}
Moreover, we can write  
\begin{equation*}
    \hat{S}^{(n)} \left( \xi \right) = \hat{S}^{(n)} \big( \xi^{\leq K(n)} \big) + \hat{S}^{(n)} \big( \xi^{\geq K(n)} \big), 
\end{equation*} where for each $\epsilon>0$, by \eqref{equation:step3} we have: 
\begin{equation*}
     \lim_{n \uparrow \infty } \mathbb{P} \Big( \sup_{t \in [0,1]}|\hat{S}_{\lfloor nt \rfloor}^{(n)} \left( \xi^{\geq K(n)} \right)  | > \epsilon \Big) \leq \lim_{n \uparrow \infty }  1-e^{-\Lambda((-\infty ,K(n) ] \cup [ K(n) ,\infty )) }  = 0. 
\end{equation*}
We can now apply   \cite[Lemma 3.31, Chapter  VI]{JacodShiryaev_LimitTheorems}  to deduce that the convergence $\hat{S}^{(n)} (\xi)\overset{\scr{L}}{\rightarrow} \hat{\xi}$ holds. 
\[ {\tag*{(\textit{End of Step 3})}} \]
With this last result we conclude the proof of Proposition \ref{theorem:weakConvergenceOfSkeletons}.  
\subsubsection{Convergence of finite-dimensional distributions} \label{subsection:fdds}
We maintain the notation and setting introduced at the beginning of Section \ref{section:jointConvergence}. 
\begin{proposition} \label{thm:convergenciaConjuntaFDD}
Let $\xi$ be a Lévy process of characteristic triplet $(a,q^2,\Lambda)$  and  denote  its characteristic exponent by $\Psi$. Fix  $p \in (0,1)$ an admissible memory parameter, and for each $n$, let 
$(S^{(n)}_k,\hat{S}^{(n)}_k)$ be the sequence of $n$-skeletons and its corresponding reinforced versions as defined in (\ref{definition:LevyReinforcedSkeletton}). 
Then, there is the weak convergence in the sense of finite-dimensional distributions, 
\begin{equation}
\Big( \big( S^{(n)}_{\lfloor n t  \rfloor} \big)_{t \in [0,1]} ,  \big( \hat{S}^{(n)}_{\lfloor n t  \rfloor} \big)_{t \in [0,1]} \Big) \overset{f.d.d.}{ \longrightarrow} \Big(  (\xi_t)_{t \in [0,1]} , (\hat{\xi}_t)_{t \in [0,1]} \Big), 
\end{equation}
where we denoted by  $(\xi, \hat{\xi})$ a pair of processes with law characterised  by (\ref{eq:leyConjunta}).  
\end{proposition}
Remark that since the convergence is in the sense of finite dimensional distributions, the restriction $p < 1/2$ is  dropped. Our proof will rely on two results taken respectively from \cite{BertoinNRLP} and   \cite{BlumenthalGetoor_SampleFunctions}. We state them without proof for ease of reading: \medskip \\
\textbf{Corollary 3.7 of \cite{BertoinNRLP} } \textit{ Let $F$ be a continuous functional on counting functions such that $F(0) = 0$ where, with a slight abuse of notation we still write $0$ for the identically $0$ trajectory. Further, suppose that there exists $c>0$ and $1 \leq \gamma < 1/p$ such that $|F(\omega)|\leq c \omega(1)^{\gamma}$ for every counting function $\omega :[0,1] \rightarrow \mathbb{N}$. Then, if $Y$ is a Yule-Simon process with parameter $1/p$, the following convergence holds in $L_1(\mathbb{P})$:}
\begin{equation}
    \lim_{n \rightarrow \infty} \frac{1}{n} \sum_{j=1}^n F\left( N_j(\lfloor n \cdot \rfloor ) \right) = (1-p) \esp{F(Y)}{}.
\end{equation}
The second result concerns the asymptotic behaviour of $\Psi$. \medskip \\
\textbf{Lemma 3.1 of \cite{BlumenthalGetoor_SampleFunctions}} \textit{ The  asymptotic behaviour of the characteristic exponent $\Psi$ as $|z| \uparrow \infty$ is given by:} 
\begin{equation*} 
    |\Psi(z)| = 
    \begin{cases}
    o(|z|^{2 + \eta }) \quad \hspace{4mm} \text{when }q \neq 0\\
    o(|z|^{ \beta(\Lambda) + \eta }) \quad \text{when }q  = 0 \text{ and } \int_{|x|\leq 1}|x| \Lambda(\dd x) = \infty \\
    o(|z|^{ 1 + \eta }) \quad \hspace{4mm} \text{when }q  = 0 \text{ and } \int_{|x|\leq 1}|x| \Lambda(\dd x) < \infty.
    \end{cases}
\end{equation*}
Now we have all the ingredients needed for the proof of Proposition \ref{thm:convergenciaConjuntaFDD}.
\begin{proof}
We fix $k \geq 1$, $0 <\lambda_1 < \dots < \lambda_k \leq 1$, and let  $\beta_1, \dots \beta_k$ be real numbers. In order to establish the  finite dimensional convergence, it suffices to show that
\begin{equation} \label{equation:charFunctionParejaCaminatas}
    \esp{\exp\bigg\{ i \sum_{j=1}^k \big( \lambda_j  S_{\lfloor n t_j\rfloor} + \beta_j  \hat{S}_{\lfloor n t_j\rfloor} \big) \bigg\}  }{},  
\end{equation}
converges as $n \uparrow \infty$ towards (\ref{eq:leyConjunta}). In this direction, for each $n$, we write $(N_\ell^{(n)}(k))_{k \geq 1, \ell \geq 1}$ the counting process of repetitions of $\hat{S}^{(n)}$ introduced in \eqref{definition:countingRepsNRRW}.  Recalling the identity \eqref{identity:definition:countingRepsNRRW}, 
we can write, 
\begin{equation*}
     \hat{S}_{ \lfloor nt \rfloor }^{(n)} 
    =  \sum_{\ell =1}^{n} N_\ell ^{(n)}(\lfloor nt \rfloor ) X_\ell ^{(n)} 
    \hspace{10mm} \text{and} \hspace{10mm}
      {S}_{ \lfloor nt \rfloor }^{(n)} 
    = \sum_{\ell =1}^n \indic{\{ \ell  \leq \lfloor nt \rfloor \}}X_\ell ^{(n)}, 
\end{equation*}
with $\mathbb{E}[{e^{i \lambda X^{(n)}_\ell  }}] = e^{\frac{1}{n} \Psi(\lambda)}$ for every $\ell$. Then, by independence of  the counting processes $(N_\ell ^{(n)}(k))_{k \geq 1, \ell \geq 1}$ from  the  sequence  $(X^{(n)}_\ell)_{\ell \geq 0}$, the characteristic function (\ref{equation:charFunctionParejaCaminatas}) can be written as follows
\begin{align*}
    &\esp{\exp\bigg\{ i \sum_{j=1}^k \big( \lambda_j  S_{\lfloor n t_j\rfloor} + \beta_j  \hat{S}_{\lfloor n t_j\rfloor} \big) \bigg\}  }{} 
    \\
    &=
    \esp{\exp\bigg\{ i  \sum_{\ell =1}^n \left(\sum_{j=1}^k   \left(\lambda_j N_\ell ^{(n)}(\lfloor n t_j\rfloor )  + \beta_j \indic{\{ \ell \leq \lfloor nt_j \rfloor  \}} \right)  \right) X_\ell ^{(n)} \bigg\}  }{} \\
    &= \esp{\exp\bigg\{  -
     \frac{1}{n} \sum_{\ell=1}^n \Psi 
     \left(\sum_{j=1}^k \lambda_j N_\ell ^{(n)}(\lfloor n t_j\rfloor )  + \beta_j \indic{\{ \ell  \leq \lfloor nt_j \rfloor  \}} \right) 
    \bigg\}  
    }{}. 
\end{align*}
 Remark that since the law of $(N_\ell^{(n)}(k))_{k \geq 1, \ell \geq 1}$ doesn't depend on $n$, we can drop the up-script $(n)$ in the last display. Next,  recall that $N_\ell (\lfloor nt \rfloor ) = 0$ for all $t \in [0,1]$ if $\epsilon_\ell = 1$  while on the other hand,  if $\epsilon_\ell =1$,  $N_\ell (\lfloor ns \rfloor) =0$ for $ \lfloor ns \rfloor < l$, and  $N_\ell( \lfloor ns \rfloor ) \geq  1$ if $\lfloor ns \rfloor \geq l$ . Hence,  we have: 
\begin{equation*}
    \indic{ \{ \ell \leq \lfloor ns \rfloor \}} = \indic{ \big\{ N_\ell \lfloor ns \rfloor \geq 1\big\} }, \quad \quad \text{ on }   \{\epsilon_\ell = 0\}.
\end{equation*}
By the previous observations, we can write:  
\begin{align} \label{eq:partirEndosCaracteristica}
    & \frac{1}{n} \sum_{\ell =1}^n \Psi 
     \left(\sum_{j=1}^k \lambda_j N_\ell (\lfloor n t_j\rfloor )  + \beta_j \indic{\{ \ell \leq \lfloor nt_j \rfloor  \}} \right) 
    \\
    &= 
    \frac{1}{n} \sum_{\ell =1}^n \Psi 
     \left(\sum_{j=1}^k \lambda_j N_\ell (\lfloor n t_j\rfloor )  + \beta_j \indic{ \big\{ N_\ell \lfloor nt_j \rfloor \geq 1\big\} } \right) \indic{\{ \epsilon_\ell =0 \}}
     + 
     \frac{1}{n} \sum_{\ell =1}^n \Psi 
     \left(\sum_{j=1}^k \beta_j \indic{\{ \ell  \leq \lfloor nt_j \rfloor  \}} \right) \indic{\{ \epsilon_\ell=1 \}}.  \nonumber
\end{align}
Now, let us  establish the convergence in probability of both terms in the previous display separately. Starting with the first one, we introduce the functional $F : D[0,1] \rightarrow \mathds{C}$ defined as follows: 
\begin{equation}
    F (\omega) := \Psi \Bigg( \sum_{j=1}^k \lambda_j \omega(t_j) + \beta_j \indic{\{ \omega (t_j) \in [1 , \infty] \}}   \Bigg),  
\end{equation}
for $\omega : [0,1] \rightarrow \mathbb{N}$ a generic counting function. This is a $\mathbb{Q}$ - a.s. continuous functional, since $\omega \mapsto \indic{\{ \omega(s) \in [1,\infty] \}}$ can be written as  $\omega \mapsto \omega(s) \wedge 1$, which is a composition of a $\mathbb{Q}$-a.s. continuous functional  with the continuous  mapping $x \mapsto x \wedge 1$. Moreover, we have  $F(0) = 0$, and  notice that we can bound:  
\begin{equation*}
    \left| \sum_{j=1}^k \lambda_j \omega (t_j) + \beta_j \indic{\{ \omega(t_j) \in [1,\infty] \}} \right| \leq w(1) \Bigg(  \sum_{j=1}^k   \left|\lambda_j  \right| + \left|\beta_j  \right| \Bigg),  
\end{equation*}
by monotonicity of $\omega$ and the inequality  $\indic{\{ \omega (s) \in [1,\infty] \}} \leq \omega(s)$.
Now, by Lemma 3.1 of \cite{BlumenthalGetoor_SampleFunctions}, we deduce   that  $F$ satisfies the hypothesis of Corollary 3.7 from \cite{BertoinNRLP}, since 
\begin{equation}
    |F(\omega )| \leq |\omega(1)|^{\gamma} K \Bigg( \sum_{j=1}^k |\lambda_j| + |\beta_j| \Bigg) ^\gamma, \quad \text{with } \quad 
    \begin{cases}
    \gamma \in (2 , 1/p), \quad \text{ if } q \neq 0\\
    \gamma \in (1, 1/p),  \quad \text{ if } q =0, 
    \end{cases}
\end{equation}
for a constant $K$ that only depends on $\beta(\Lambda)$ and $q$. From an application of  Corollary 3.7 of \cite{BertoinNRLP}, we obtain  the following convergence: 
\begin{align} \label{eq:limite1}
     & \lim_{n \rightarrow \infty} \frac{1}{n} \sum_{\ell =1}^n \Psi 
     \Bigg( \sum_{j=1}^k \lambda_j N_\ell (\lfloor n t_j\rfloor )  + \beta_j \indic{ \big\{ N_\ell  \lfloor nt_j \rfloor \geq 1\big\} } \Bigg) \indic{\{ \epsilon_\ell =0 \}} \nonumber  \\
     & = 
     \lim_{n \rightarrow \infty} \frac{1}{n} \sum_{\ell =1}^n \Psi 
     \Bigg( \sum_{j=1}^k \lambda_j N_\ell (\lfloor n t_j\rfloor )  + \beta_j \indic{ \big\{ N_\ell  \lfloor nt_j \rfloor \geq 1\big\} } \Bigg) \nonumber \\
     & = (1-p) \esp{\Psi\Bigg( \sum_{j=1}^k \lambda_j Y( t_j )  + \beta_j \indic{ \big\{ Y(t_j) \geq 1\big\} } \Bigg)  }{}.
\end{align}
Turning our attention to the second term, similarly, we claim that: 
\begin{equation} \label{eq:limite2}
    \lim_{n \rightarrow \infty} \frac{1}{n} \sum_{\ell =1}^n \Psi 
     \Bigg(\sum_{j=1}^k \beta_j \indic{\{ \ell  \leq \lfloor nt_j \rfloor  \}} \Bigg) \indic{\{ \epsilon_\ell =1 \}} = p \esp{\Psi \Bigg( \sum_{j=1}^n \beta_j \indic{ \big\{ U \leq  t_j\big\} } \Bigg)  }{}.
\end{equation}
Indeed, if for each $n$ we denote by $u(n)$ a uniform random variable on $\{1, \dots , n \}$ independent of the i.i.d. sequence  $(\epsilon_n)_n$ of Bernoulli with parameter $p$, we have 
\begin{align} \label{eq:termino2Convergencia}
     \esp{\frac{1}{n} \sum_{\ell =1}^n \Psi 
     \Bigg( \sum_{j=1}^k \beta_j \indic{\{ \ell  \leq \lfloor nt_j \rfloor  \}} \Bigg) \indic{\{ \epsilon_l=1 \}} }{}
     &= \esp{ \Psi 
     \Bigg( \sum_{j=1}^k \beta_j \indic{\{ u(n) \leq \lfloor nt_j \rfloor  \}} \Bigg) \indic{\{ \epsilon_{u(n)}=1 \}} }{} \nonumber \\
     &=
     \esp{ \Psi 
     \Bigg(\sum_{j=1}^k \beta_j \indic{\{ u(n) \leq \lfloor nt_j \rfloor  \}} \Bigg)  }{}p, 
\end{align}
 since  $\epsilon_{u(n)}$ is independent of $u(n)$ for each $n$. Further, since $u(n)/n$ converges in law towards a uniform random variable in $[0,1]$, the sequence of step processes $(\indic{ \{ u(n) \leq \lfloor n \cdot  \rfloor \} })_{n \in \mathbb{N}}$   converges weekly towards $\indic{\{ U \leq \cdot  \}}$. Consequently, as $n \uparrow \infty$, (\ref{eq:termino2Convergencia})  converges towards
 \begin{equation*}
 p \, \esp{ \Psi \Bigg( \sum_{j=1}^k \beta_j \indic{\{ U \leq t_j \}} \Bigg)  }{},  
 \end{equation*}
 where we recall that  $\indic{ \{ U \leq t \}}$  has the same distribution as $\indic{ \{ Y(t) \geq 1 \}}$ by the description \eqref{formula:RepresentacionYuleSimon}. Finally, recall the identity of Proposition \ref{theorem:leyConjunta} for characteristic function of the finite dimensional distributions of the pair $(\xi , \hat{\xi})$. It follows from (\ref{eq:partirEndosCaracteristica}) and the limits (\ref{eq:limite1}), (\ref{eq:limite2})  that as $n \uparrow \infty$,  we have the convergence towards the characteristic function of the finite-dimensional distributions of $(\xi , \hat{\xi})$,  
\begin{align*}
    & \lim_{n \uparrow \infty} \esp{\exp\bigg\{ i \sum_{j=1}^k \big( \lambda_j  S_{\lfloor n t_j\rfloor} + \beta_j  \hat{S}_{\lfloor n t_j\rfloor}\big) \bigg\}  }{} \\
    &=  \exp \bigg\{ p  \esp{ \Psi \Bigg( \sum_{j=1}^k \lambda_j \indic{\{ Y(t_j) \geq 1 \}} \Bigg) }{} +(1-p) \esp{\Psi\Bigg( \sum_{j=1}^k \lambda_j \indic{\{ Y(t_j) \geq 1 \}} + \beta_i Y(s_i) \Bigg) }{} \bigg\}. 
\end{align*}
\end{proof}
\noindent This result paired with the tightness established  in Proposition \ref{theorem:weakConvergenceOfSkeletons} proves Theorem \ref{thm:convergenciaConjunta}.

\begin{toexclude}
{\color{blue}(?) En LevyIto 2 mejor iniciar con nrPPP y construir al correspondiente NRLP?} We will assume from now on that all the Poissonian NRLPs considered in this work have been explicitly constructed by means of a noise reinforced PPP.  In consequence we can now finally state:
\begin{theorem} \label{Theorerm:LevyIto-Part2} \emph{(Lévy-Itô decomposition-Part 2)}\\
Let $\hat{\xi}$ be a NRLP with characteristics $(a,q^2,\Lambda,p)$ and denote by $\hat{\mu}$ its jump measure. Then, $\hat{\mu}$ is a noise reinforced PPP with characteristic $\Lambda$ and parameter $p$ and the following equality holds almost surely for all t, 
\begin{equation} \label{Reinforced-Lévy-Ito}
     \hat{\xi}_t = a t + q\cdot \hat{B}_t + \int_{[0,t]\times (-1,1)^c} x \hat{\mu}({\dd s,\dd x}) +  \int_{[0,t]\times (-1,1)} x \hat{\mu}^{(sc)}({\dd s,\dd x}) \quad \quad t \geq 0.
\end{equation}
\end{theorem}
\begin{remark} Beware of the notation,  $\hat{\mu}^{(sc)}$ stands for the \textit{space}-compensated  jump measure $\hat{\mu}$ and should not be confused with the \textit{time}-compensated measure $(\mu-\mu^p)$  in the sense of II-1.27 of \cite{JacodShiryaev_LimitTheorems}. For instance, the integral with respect to $\mu^{(sc)}$ in (\ref{Reinforced-Lévy-Ito}) is not a local martingale. For Lévy processes, the time and space compensation of its PPP coincide, as the compensating measure is the same.
\end{remark}
We briefly comment on the statements and the connection with the Lévy-Itô decomposition. The Lévy-Itô decomposition for Lévy processes states that any Lévy process $\xi$ with triplet $(a,q^2,\Lambda)$ can be written as $\xi = C + J$, where $C$ is a a Brownian motion with drift while $J$ is a purely discontinuous process that can be explicitly recovered from the jump measure $\mu$ defined in (\ref{ppp-Levy}). More precisely, if we denote as $\mu^{(sc)}$ the compensated PPP of jumps $\mu^{(sc)} = \mu -\dd t\Lambda(\dd x)$,  we can write 
\begin{equation} \label{equation:LevyIto}
    \xi_t = ct + q\cdot B_t + \int_{[0,t]\times (-1,1)^c} x \mu({ \dd s,\dd x}) +  \int_{[0,t]\times (-1,1)} x \mu^{(sc)}({ \dd s,\dd x}) \quad \quad t \geq 0
\end{equation}
where, if we denote as $\xi^{(c)}$  the process corresponding to the compensated integral, 
\begin{equation*}
    {\xi}^{(c)}_t = \lim_{a \downarrow 0} \sum_{s \leq t} \indic{\{a \leq |\Delta \xi_s| < 1\}} \Delta \hat{\xi}_s - t \int_{\{ a\leq |x| < 1\}} x \Lambda(\dd x) \quad \quad  t \geq 0
\end{equation*}
and for any $A \in \scr{B}(\mathbb{R})$ with $\Lambda(A) < \infty$, the process $(\mu([0,t] \times A), t \geq 0)$ is a Poisson process with rate  $\Lambda(A)$. On the other hand, if we consider the reinforced version $\hat{\xi}$ of (\ref{equation:LevyIto}), this is a NRLP with characteristics $(a,q^2,\Lambda,p)$ and we know from (\ref{descomposicionNRLP}) and  Lemma \ref{lemma:NRPPasIntegralsOfNRPPP} that this process can be written as $\hat{\xi} = \hat{C} + \hat{J}$ where $\hat{C}$ is a noise reinforced Brownian motion $\hat{B}$ with drift and $\hat{J}$ is the reinforced version of the Lévy process $J$, obtained by integrating a reinforced PPP. More precisely,  since for $0<a<b$ 
\begin{equation*}
    \sum_{s\leq t} \indic{\{a \leq  |\Delta \xi_s| < b\}} \Delta \hat{\xi}_s  =\int_{\{a \leq |x| <  b\} \times [0,t]} x \hat{\mu}( \dd s , \dd x)
\end{equation*}
 we deduce from Corollary \ref{theorem:laMartingalaNRLP} and Lemma \ref{lemma:NRPPasIntegralsOfNRPPP} (ii) with our current notation that 
\begin{equation*}
    \hat{\xi}^{(c)}_t = \lim_{a \downarrow 0} \sum_{s \leq t} \indic{\{a \leq |\Delta \xi_s| < 1\}} \Delta \hat{\xi}_s - t \int_{\{ a\leq |x| < 1\}} x \Lambda(\dd x) \quad \quad  t \geq 0
\end{equation*}
the  convergence holding uniformly in compact intervals, and we get  (\ref{Reinforced-Lévy-Ito}) by definition of the compensated measure $\hat{\mu}^{(sc)}$ where now, by Theorem \ref{lemma:PreFormulaCompensacion} the process  $(\hat{\mu}([0,t] \times A), t \geq 0)$ is a reinforced Poisson process with rate $\Lambda(A)$. More precisely, it has the law of the noise reinforced version of ${\mu}([0,t] \times A)$ which yields  the equality in distribution, 
\begin{equation}
   \hat{\mu}([0,t] \times A) \overset{\scr{L}}{=}  \reallywidehat{ {\mu}([0,t] \times A)}. 
\end{equation}
From the uniform convergence showed in Corollary \ref{theorem:laMartingalaNRLP} and the structure (\ref{definition:decoratedNRPP}) of the jump process $\hat{\mu}$ we deduce the following decomposition, which comes from splitting the jumps of $\hat{\xi}$ into two groups, the ones occurring for the first time, at times $u \in \scr{P}$ and the ones coming from repetitions, occurring at  times $\{ ue^{s/p} : s \in D_u^Z\}$.   
\end{toexclude}
\addtocontents{toc}{\vspace{\contentsSpacingBefore}}

\section{Applications} \label{section:applications}

We conclude this work with two sections devoted to applications. 

\subsection{Rates of growth at the origin}
\label{section:ratesOfGrowth}
\addtocontents{toc}{\vspace{\contentsSpacing pt}}
In this section we turn our attention to the trajectorial behaviour of noise reinforced Lévy processes at the origin. In this direction let us start by recalling a well known result established by Blumenthal and Getoor \cite{BlumenthalGetoor_SampleFunctions} for Lévy processes.  Let $\xi$ be a Lévy process with characteristic triplet $( a , q^2,  \Lambda)$ with no Gaussian component, viz. $q=0$;  in particular  $\beta(\Lambda) = \beta$. Further,  we make the following hypothesis: \medskip \\
$\bullet$ If $\int_{\{ |x| \leq 1 \}} |x| \Lambda(\dd x) = \infty $, the characteristic exponent can be written as follows: 
\begin{equation*}
    \Psi(\lambda) = \int_{\mathbb{R}} \left( e^{i \lambda x} - 1 - i \lambda x \indic{\{ |x| \leq 1 \}} \right) \Lambda(\dd x). 
\end{equation*}
Observe that in this case,  we have  $\beta(\Lambda) \in [1,2]$. \smallskip \\
$\bullet$ If $\int_{\{ |x| \leq 1 \}} |x| \Lambda(\dd x) < \infty$, which can happen for $\beta(\Lambda) \in [0,1]$, we suppose $\Psi$ takes the following form:
\begin{equation*}
    \Psi(\lambda) = \int_{\mathbb{R}} \left( e^{i \lambda x} - 1 \right) \Lambda(\dd x).
\end{equation*}
This is, when $\int_{[0,1]} |x| \Lambda(\dd x) < \infty$ we are supposing that the Lévy process has no linear drift, the reason being that in that case the behaviour at 0 is dominated by the drift term. We insist in the fact that when $\beta(\Lambda)=1$ the integral $\int_{\{ |x| \leq 1\}} |x| \Lambda(\dd x)$ can be finite or infinite. \par We will be working for the rest of the section under these hypothesis, and we will refer to them as   hypothesis (H). 
It was established by Blumenthal and Getoor in \cite{BlumenthalGetoor_SampleFunctions} that  under (H), the behaviour at zero of a Lévy process is dictated by the Blumenthal-Getoor index of the Lévy measure $\Lambda$. More precisely,  almost surely, we have: 
\begin{equation*}
    \lim_{t \downarrow 0} {t^{-\gamma}} \xi_t = 0 ,  \quad  \text{ if } \beta(\Lambda) < 1/\gamma
  \quad \quad  \quad \text{ and } \quad \quad  \quad 
 \limsup_{t \downarrow 0} {t^{-\gamma}} |\xi_t| = \infty, \quad \text{ if }  \beta(\Lambda) > 1/\gamma.
\end{equation*}
We will show that the same result still holds if we replace the Lévy process $\xi$ by its  noise reinforced version. Concretely, the  main result of the section is  the following:
\begin{proposition}\label{teorema:comportamientoOrigen}
Let $\xi$ be a Lévy process with triplet $(a,q^2,\Lambda)$ satisfying hypothesis \emph{(H)},  and consider   $\hat{\xi}$ its noise reinforced version for an  admissible parameter $p$. Then, 
 almost surely, we have 
\begin{equation} \label{theorem:asimtotOriginlim0}
    \lim_{t \downarrow 0} {t^{-\gamma}} \hat{\xi}_t = 0, \quad \hspace{6mm} \hspace{3mm} \text{ if } \beta(\Lambda) < 1/\gamma, 
\end{equation}
while 
\begin{equation}\label{theorem:asimtotOriginlimsup0}
 \limsup_{t \downarrow 0} {t^{-\gamma}} |\hat{\xi}_t| = \infty, \quad \text{ if }  \beta(\Lambda) > 1/\gamma.
\end{equation}
\end{proposition}
The rest of the section is devoted to the proof of Proposition  \ref{teorema:comportamientoOrigen}  and it is achieved in several steps.  We start by proving the second statement (\ref{theorem:asimtotOriginlimsup0}), in Lemma \ref{lemma:conv0p1} we prove (\ref{theorem:asimtotOriginlim0}) for $\beta(\Lambda) \geq 1$, $\int_{|x| \leq 1} |x| \Lambda(\dd x) = \infty$ and the case $\beta(\Lambda) \leq 1$, $\int_{|x| \leq 1} |x| \Lambda(\dd x) < \infty$  is treated separately  in Lemma \ref{proposition:rateGrowthSubordinator}. 

\begin{proof}[Proof of (\ref{theorem:asimtotOriginlim0}).]
It  suffice to prove that for some $r >0$ and $\epsilon > 0$ a.s. there exists a sequence of jumps occurring in $[0, \epsilon]$ at times, that we denote by  $(t_i)$,  satisfying 
\begin{equation*}
    |\Delta \hat{\xi}_{t_i}| > t_i^{\gamma - r}.
\end{equation*}
 Now, recall from the discussion following (\ref{definition:decoratedNRPP}) that the jump measure $\hat{\mathcal{N}}$ of $\hat{\xi}$ dominates a Poisson point process with intensity $(1-p) ( \dd  u \otimes \Lambda )$, say $\mathcal{N}'$. If we denote the atoms of $\mathcal{N}'$ by  $(u_i,x_i)$, we deduce that
\begin{equation*}
    \# \{ (u_i, x_i) \in \mathcal{N}' : u_i \in [0,\epsilon] \text{ and } |x_i| > 2 u_i^{\gamma - r} \}, 
\end{equation*}
is distributed Poisson  with parameter  
\begin{equation} \label{eq:intensidadPoisson}
   (1-p) \dd u \otimes  \Lambda   \left( (u,x) \in [0,\epsilon]\times \mathbb{R} : |x|^{1/(\gamma - r)}>2 \cdot u  \right) = \int_{\mathbb{R}}\left( 2^{-1}|x|^{1/(\gamma-r)} \wedge \epsilon  \right) \Lambda(\dd x) (1-p) .
\end{equation}
Now, take $r>0$ small enough such that the inequality $1/(\gamma - r) < \beta(\Lambda)$ still holds. For such a choice of $r$, the integral  (\ref{eq:intensidadPoisson}) is infinite by definition of the index $\beta(\Lambda)$ and the claim follows.
\end{proof}
Now we focus on showing that $\lim_{t \downarrow 0} t^{-\gamma} |\hat{\xi}_t|=0$ for $\gamma \in (0,1/\beta(\Lambda))$. In this direction, let us start introducing some notation and with some preliminary remarks. First, notice that since we are interested in the behaviour of $\hat{\xi}$ at the origin, we can   rely   on the original construction in \cite{BertoinNRLP} in terms of Poissonian sums of Yule-Simon processes that we recalled in Section \ref{subsection:preliminaries,examplesNRLPs}. Next, under (H),  $\hat{\xi}$ can be written either as a sum of a compensated  integral $\hat{\xi}^{(3)}$ and a reinforced compound Poisson process $\hat{\xi}^{(2)}$ viz. 
\begin{equation} \label{equation:caso1_comportamientoOrigen}
  \hat{\xi}=  \hat{\xi}^{(3)} + \hat{\xi}^{(2)},   \hspace{6mm}\hspace{15mm} \text{if } \beta(\Lambda) >1,  
\end{equation}
or as an absolutely convergent series of jumps, 
\begin{equation} \label{equation:caso2_comportamientoOrigen}
    \hat{\xi}= \sum_{s \leq t}  \Delta \hat{\xi}_s, \, \, t \in [0,1],  \hspace{10mm} \text{if } \beta(\Lambda) <1. 
\end{equation}
We stress that if $\beta(\Lambda) = 1$, $\hat{\xi}$ takes the form (\ref{equation:caso1_comportamientoOrigen})  or  (\ref{equation:caso2_comportamientoOrigen}) depending respectively on if $\int_{\{|x| \leq 1\}}|x| \Lambda(\dd x)$ is infinite or not, and remark that $\gamma$  can be strictly larger than one only when $\beta(\Lambda) < 1$.  Since $\hat{\xi}^{(2)}$ is a finite sum of weighted Yule processes and $\hat{\xi}^{(2)}_0 = 0$, independently of the value of $\beta(\Lambda)$ it holds that   $\lim_{t \downarrow 0} t^{-\gamma} \hat{\xi}^{(2)}_t =0$ and we can consequently restrict  our study of  (\ref{equation:caso1_comportamientoOrigen}) resp.  (\ref{equation:caso2_comportamientoOrigen})   to the case where   $\hat{\xi} = \hat{\xi}^{(3)}$ resp.  $\hat{\xi}$ has  Lévy measure concentrated in $[0,1]$ -- and hence is a reinforced, driftless subordinator. 
\begin{lemma} \label{lemma:conv0p1} Suppose that $\beta(\Lambda) \geq 1$ and  let $\gamma \in (0,1 \wedge 1/\beta(\Lambda))$.  Then, 
\begin{equation*}
    \lim_{T \downarrow 0} \esp{\sup_{s \leq T} s^{-\gamma} |\hat{\xi}^{(3)}_s| }{} = 0.
\end{equation*}
In particular, if $\beta(\Lambda) \geq 1$ with $\int_{\{|x| \leq 1 \}} |x| \Lambda(\dd x) = \infty$,  we have   $\lim_{t \downarrow 0} t^{-\gamma}|\hat{\xi}^{(3)}_t|=0$  a.s.
\end{lemma}
\begin{proof}
Recall from Proposition \ref{proposition:laMartingala} that $(t^{-p} \hat{\xi}^{(3)}_t)_{t \in [0,1]}$ is a martingale. We start by fixing $s<u$ two times in $[0,1]$ and notice that for any $r \in (\beta(\Lambda)  , 1/p \wedge 2)$ (or $r=2$ if $\beta(\Lambda) = 2$ ), by Doob's inequality in $L_r(\mathbb{P})$ we have  
\begin{align} \label{eq:boundInIntervalSU}
    \esp{\sup_{t \in [s,u]} t^{-\gamma} |\hat{\xi}^{(3)}_t| }{} 
    &\leq s^{-(\gamma-p)} \esp{\sup_{t \in [s,u]} t^{-p} |\hat{\xi}^{(3)}_t| }{} \nonumber \\
    &\leq  s^{-(\gamma-p)} \esp{\sup_{t \in [s,u]} t^{-p \cdot r} |\hat{\xi}_t^{(3)}|^r }{}^{1/r}  
    \leq c \cdot  s^{-(\gamma-p)} u^{-p} \esp{  |\hat{\xi}_u^{(3)}|^r }{}^{1/r}  , 
\end{align}
for some constant $c$. In order to bound the expectation on the right hand side, we recall from the proof of Lemma 2.6 in Bertoin \cite{BertoinNRLP} that the following bound holds
\footnote{The bound was first established for non-atomic Lévy measures $\Lambda$, but  it was later shown that a similar bound holds if $\Lambda$ has atoms by an approximation argument.}
for some constant $C$ large enough:
\begin{equation} \label{equation:cotaCompensacion1}
    \esp{|\hat{\xi}^{(3)}_u |^r}{}^{1/r}  \leq C \esp{\sum_j Y_j(u)^r|x_j|^r }{}^{1/r}. 
\end{equation}
Next,  by Campbell's formula we have 
\begin{equation} \label{equation:cotaCompensacion2}
    \esp{ \sum_j Y_j(u)^r|x_j|^r }{} = \esp{Y(u)^r}{} \int_{\{ |x|<1 \}} |x|^r \Lambda(\dd x) < \infty, 
\end{equation}
and remark that  $\esp{Y(u)^r}{} = u \cdot  \esp{\eta^r}{}$ where $\eta$ stands for a Yule-Simon random variable with parameter $1/p$. It now follows that we can bound 
$\esp{|  \hat{\xi}^{(3)}_u |^r}{}^{1/r} \leq K \cdot u^{1/r}$ for a positive constant $K$ depending only on $r$. This observation  paired with the bound we obtained in (\ref{eq:boundInIntervalSU}), yields:  
\begin{equation} \label{eq:boundInterval2}
    \esp{\sup_{t \in [s,u]} t^{-\gamma} |\hat{\xi}^{(3)}_t| }{} \leq  s^{-(\gamma-p)} u^{-p + 1/r} \cdot K,  
\end{equation}
for a finite constant $K$ that only depends on the choice of $r$.
Now, set   $t_0:=1$, $t_n := 2^{-n}, \text{ for } n \geq 1$ and fix $N\in \mathbb{N}$. Applying  the bound  (\ref{eq:boundInterval2})  to each interval $[2^{-(n+1)},2^{-n} ]$, we get: 
\begin{equation} \label{condition:seriesConvergentes}
    \esp{\sup_{t \leq t_N} t^{-\gamma} |\hat{\xi}^{(3)}_t| }{} \leq \sum_{n \geq N}\esp{\sup_{t \in [t_{n+1},t_n]} t^{-\gamma} |\hat{\xi}^{(3)}_t| }{} \leq 2^{\gamma - p} \sum_{n \geq N} 2^{n(\gamma - 1/r)},  
\end{equation}
and to conclude it suffices to show that, for an appropriate choice of  $r$, the inequality  $\gamma-1/r < 0$ is satisfied. Since  $r \in  (\beta(\Lambda) , 1/p \wedge 2)$, we can  always choose  $\epsilon$ small enough such that  $r:= \beta(\Lambda) + \epsilon$ belongs to $(\beta(\Lambda) , 1/p \wedge 2)$ and $\gamma < 1/ (\beta(\Lambda) + \epsilon)$,   since  we recall that $\gamma < 1/\beta(\Lambda)$. For such a particular choice of $r$, the series (\ref{condition:seriesConvergentes}) converge and we obtain the desired result. In particular, this proves the statement of Proposition  \ref{teorema:comportamientoOrigen}  when $\int_{\{ |x| \leq 1 \}} |x|\Lambda(\dd x) = \infty$,  which is when $\hat{\xi} = \hat{\xi}^{(3)} $.   
\end{proof}
The statement (\ref{theorem:asimtotOriginlim0}) of Proposition \ref{teorema:comportamientoOrigen}  is incomplete only when the Lévy measure fulfils the integrability condition $\int_{\{|x| \leq 1\}} |x| \Lambda(\dd x) < \infty$. Recalling the discussion prior to Lemma \ref{lemma:conv0p1}, we  henceforth assume that the Lévy process is a driftless subordinator with jumps smaller than one, say $(T_t)$,  and we  denote by $(\hat{T}_t)$  the corresponding reinforced version for a memory parameter $p \in (0,1)$. It is then convenient to work with its Laplace transform at time $t \in [0,1]$,
\begin{equation*}
    \esp{  e^{ -\lambda \hat{T}_t}  }{} =  \exp \Big( -  \esp{ \Phi (Y(t) \lambda)  }{} \Big),  \quad \quad \text{ for } \lambda \geq 0, 
\end{equation*}
for  $\Phi(\lambda) := (1-p) \int_{\mathbb{R}^+} \left( 1- e^{-x \lambda } \right) \Lambda(\dd x)$ and $Y$ is a Yule-Simon process with parameter $1/p$.  The following result from  \cite{BlumenthalGetoor_SampleFunctions} will be needed   and  we state it for the reader's convenience: 
\begin{theorem} \emph{[Blumenthal, Getoor]\cite{BlumenthalGetoor_SampleFunctions}}  \label{Teorema:BandGetoorLaplaceAsymptotics}
If $\Phi(\lambda )$ is the Laplace exponent of a driftless subordinator with Lévy measure $\Lambda$, then for any $\epsilon>0$,
\begin{equation*}
    \Phi(\lambda) = o (\lambda ^{\beta(\Lambda) + \epsilon}), \quad \text{ as } \lambda \uparrow \infty.
\end{equation*}
\end{theorem} 
\noindent Let $\epsilon>0$, fix $\lambda >0$ and observe from Theorem \ref{Teorema:BandGetoorLaplaceAsymptotics} that for $t \in (0,1)$, there exists positive constants $K$ and $R$ such that 
\begin{equation*}
    \Phi( \eta \lambda t^{-\gamma} ) \leq 
    \begin{cases}
    K  \quad \hspace{18.5mm} \text{ if } \eta\lambda t^{-\gamma} \leq R, \\
     \left( \lambda \eta t^{-\gamma} \right)^{\beta(\Lambda) + \epsilon}   \quad \text{ if } \eta\lambda t^{-\gamma} > R.\\
    \end{cases}
\end{equation*}
Consequently,  for $t \in (0,1)$ the following bound holds: 
\begin{align} \label{desigualdad:cotaExponenteLaplace}
    t \Phi(\eta \lambda t^{-\gamma} ) \leq t \left( K +  \left( \lambda \eta t^{-\gamma} \right)^{\beta(\Lambda) + \epsilon}   \right)  = t K +  (\lambda  \eta  )^{\beta(\Lambda)+ \epsilon} t^{1 - \gamma \beta(\Lambda) - \gamma \epsilon}.
\end{align}

\begin{lemma} \label{proposition:rateGrowthSubordinator}
Let $\hat{T}$ be a reinforced subordinator of  memory parameter $p$ and Lévy measure $\Lambda$. Then, for any $\gamma \in  \mathbb{R}^+$ such that  $\gamma < 1/\beta(\Lambda)$, 
\begin{equation*}
    \lim_{t \downarrow 0} t^{-\gamma} \hat{T}_t = 0 \quad \text{a.s.}
\end{equation*}
\end{lemma}
\noindent The proof relies on the same techniques  used for  subordinators, see   \cite[Proposition 10 -  III.4]{Bertoin_LevyProcessesBook}.
\begin{proof}
 Consider $t \in [0,1]$ and  fix $a>0$. An application of Markov's inequality  for $g(r) = 1-e^{-r}$ and the inequality $g(r) \leq r$ for $r \geq 0$ yield 
\begin{align*}
    \mathbb{P} \big(  \hat{T}_t > a \big) 
    &\leq (1-e^{-1})^{-1} \left( 1- \exp \left\{ - \esp{ \Phi\left( a^{-1} Y(t) \right) }{}  \right\} \right) \\
    & \leq (1-e^{-1})^{-1} \esp{ \Phi\left( a^{-1} Y(t) \right) }{}.
\end{align*}
Since $\Phi(0) = 0$ and  $Y(t)$ conditioned to $Y(t) \geq 1$ follows the Yule-Simon distribution with parameter $1/p$, for a constant $C$ we deduce the bound: 
\begin{equation} \label{equation:desigualdadSubordinador}
    \mathbb{P} \big(  \hat{T}_t > a \big)  \leq C t \esp{ \Phi\left( \eta/a \right)  }{}, 
\end{equation}
where we denoted  by $\eta$ a Yule Simon random variable with parameter $1/p$. Now, let $h$ be an  increasing   function with $\lim_{t \downarrow 0} h(t) =0$, and consider  $a = h(2^{-n})$, $t = 2^{-(n-1)}$. Then, by  (\ref{equation:desigualdadSubordinador}) and from summing over $n \in \mathbb{N}$, we deduce 
\begin{equation} \label{equation:desigualdadSubordinador2}
    \sum_{n=1}^\infty \proba{\hat{T}_{2^{-(n-1)}} > h(2^{-n}) }{} 
    \leq 2C  \esp{ \sum_{n=1}^\infty 2^{-n}\Phi\left( \eta/h(2^{-n}) \right)  }{}.
\end{equation}
In order to apply a Borel-Cantelli argument, we specialise in our case of interest: we set $h(t) := t^{\gamma}$ and we show that the right hand side of  (\ref{equation:desigualdadSubordinador2}) is finite. From the first inequality in (\ref{desigualdad:cotaExponenteLaplace}) with $\lambda = 1$, we get 
\begin{equation*}
    \sum_{n=1}^\infty 2^{-n}\Phi\left( \eta/h(2^{-n}) \right) \leq K \sum_{n=1}^\infty 2^{-n} +  \eta^{\beta(\Lambda) + \epsilon} \sum_{n=1}^\infty (2^{-n})^{1-\gamma\beta(\Lambda) - \gamma \epsilon}. 
\end{equation*}
For $\epsilon$ small enough, we have both $\eta^{\beta(\Lambda) + \epsilon} \in L_1(\mathbb{P})$ (since $\eta$ is in $L_q(\mathbb{P})$ for any $q < 1/p$ and $\beta(\Lambda) < 1/p$) and $1-\gamma\beta(\Lambda) - \gamma \epsilon >0$, by our standing assumption  $1> \gamma \beta(\Lambda)$. Consequently, we have  
\begin{equation*}
     \sum_{n=1}^\infty \proba{\hat{T}_{2^{-(n-1)}} > (2^{-n})^{\gamma} }{}  < \infty, 
\end{equation*}
which entails by Borel-Cantelli that  $\hat{T}_{2^{-(n-1)}}< (2^{-n})^{\gamma}$ holds for all $n$ large enough, a.s. From a monotony argument, it follows that a.s. $\hat{T}_{t}< t^{\gamma}$ for all $t$ small enough and in consequence $\limsup_{t \downarrow 0} t^{-\gamma} \hat{T}_t \leq 1$. If we now take $h(t) = \delta t^{\gamma}$ for $\delta \in (0,1)$, by the same reasoning we obtain  $\limsup_{t \downarrow 0} t^{-\gamma} \hat{T}_t \leq \delta$ which leads to the desired result. 
\end{proof}
Finally,  our proof of Proposition \ref{teorema:comportamientoOrigen} is complete. 

\subsection{Noise reinforced Lévy processes as infinitely divisible processes} \label{section:IDprocesses}
\addtocontents{toc}{\vspace{\contentsSpacing pt}}
As was already mentioned in Section \ref{subection:jump_LevyIto-Part1}, NRLPs are   \textit{infinitely divisible processes} --  abbreviated ID processes. In this final  section, we study their properties under this new scope. In this direction, we  start by giving a brief overview of the theory;  our exposition   mainly  follows Rosinksi \cite{RosinskiID} and Chapter 3 of Samorodnitsky  \cite{Samorodnitsky_StochProcesses_and_LongRangeDependance(+IDprocesses)}. Then, we   identify the features of NRLPs in this setting and more precisely, we identify the functional triplet of NRLPs, in the sense of ID processes. The objective here is hence  to put Lévy processes and their NRLPs counterparts  in the context of ID  processes and compare then through this new lens.   As an application, making use of the Isomorphism Theorem for ID processes \cite[Theorem 4.4]{RosinskiID}  we establish the following result: 
\begin{proposition} \label{proposition:convOrigenNRLP}
Let $\hat{\xi}$ be a noise reinforced Lévy process with characteristics $(a,0,\Lambda, p)$. Let $f: \mathbb{R} \rightarrow \mathbb{R}^+$ be a bounded, continuous  function with { $f(x) = O(x^2)$ at $0$}. Then,  we have 
\begin{equation*}
    \lim_{h \downarrow 0} h^{-1} \mathbb{E}\big[ {f \big( \hat{\xi}_h \big) } \big] = 
    p^{-1} {(1-p)}
    \int_{\mathbb{R}}  \Lambda(\dd x) \sum_{k=1}^\infty f(kx)  \emph{B}(k,1/p+1). 
\end{equation*}
\end{proposition}
Note that the probability distribution  appearing in the previous display is  the  Yule-Simon distribution \eqref{formula:densityYuleSimon}. For an analogous result in the setting of Lévy processes, we refer to  \cite[Proposition 4.13]{RosinskiID} and we shall use in our proof similar type of arguments. 
To simplify notation, for the rest of the section we  work with NRLPs in $[0,1]$,  but our exposition can be adapted to $\mathbb{R}^+$ with some slight changes. Hence, we can make use of the construction of NRLPs from \cite{BertoinNRLP} in terms of Poissonian Yule-Simon series that we recalled at the end of Section \ref{section:extensionAndRegularity}. This construction will be used for the rest of the section. 
\subsubsection{Preliminaries on infinitely divisible processes}
Let us introduce some standard notation mostly taken from \cite{RosinskiID}.  For $T$ a nonempty set, we  denote by $\mathbb{R}^T$ the set of $\mathbb{R}$-valued functions indexed by $t \in T$. If $S \subset T$ is an arbitrary subset and $e = (e(t))_{t \in T} \in \mathbb{R}^T$, we write $e_S$  for the restriction of $e$ to $S$. Further, let   $\pi_S$ be  the canonical projection $\pi_{S} : \mathbb{R}^T \rightarrow \mathbb{R}^{S}$ from $\mathbb{R}^T$ into $\mathbb{R}^S$, viz. the function defined as $\pi_S(e) := e_S$. For finite subsets of $T$ of the form  $I:= \{t_1, \dots , t_k \} \subset  T$, the space $\mathbb{R}^I$ is identified with  $\mathbb{R}^k$ and we write:  
\begin{equation*}
    e_I  = (e({t_1}), \dots , e({t_k})) \in \mathbb{R}^{I}.
\end{equation*}
As usual, the space $\mathbb{R}^T$ is equipped with the cylindrical sigma field  $\scr{B}^T := \sigma (\pi_t \, : \, t \in T)$ generated by the projection mappings. For any arbitrary $S \subset T$, we denote by $0_S$ the $0$ element of $\mathbb{R}^S$ and we write  $\scr{B}_0^S := \{ A \in \scr{B}^S : 0_S \notin A \}$.  Consequently, 
\begin{equation*}
    \pi^{-1}_S( {0}_S) = \{e \in \mathbb{R}^T : e(t) = 0 \text{ for all } t \in S \} =: \mbm{0}_S \subset  \mathbb{R}^T. 
\end{equation*}
Notice however that this subset is not $\scr{B}^T$ measurable when $S$ is uncountable.    Finally,  for $x \in \mathbb{R}$ we set $\llbracket x \rrbracket := x \indic{\{|x| \leq 1 \}}$ and  if $x=(x_1, \dots , x_k) \in \mathbb{R}^k$, the term  $\llbracket x \rrbracket$ should be interpreted component-wise, viz.  $\llbracket x \rrbracket := \left( \llbracket x_1 \rrbracket, \dots , \llbracket x_k \rrbracket \right)$. Now let us start with the following definition: 
\begin{definition} \label{definition:IDprocess}
An $\mathbb{R}$-valued stochastic process $X = (X_t)_{t \in T}$ is said to be infinitely divisible (in law) if for any $n \in \mathbb{N}$, there exist independent and identically distributed processes  $Y^{(n,1)}, \dots Y^{(n,n)}$ such that 
\begin{equation*}
    X \overset{\scr{L}}{=} Y^{(n,1)} + \dots Y^{(n,n)}.
\end{equation*}
\end{definition}
When $T=\{ 1 \}$ is a singleton, this is just the definition of a real valued  infinitely-divisible random variable, in which case,  the characteristic function of $X_1$ takes the Lévy-Khintchine form: 
\begin{equation*}
    \esp{e^{i \theta X_1 }}{} = \exp \Big\{ i \theta b -   \frac{q^2}{2} \theta^2 + \int_{\mathbb{R}} \left( e^{i \theta x} -1 - i  \theta  \llbracket x \rrbracket  \right)\nu(\dd x)  \Big\}, 
\end{equation*}
for $q, b \in \mathbb{R}$, $\nu$ a Lévy measure. Further, it is well known that the set of infinitely divisible random variables and distributions of Lévy processes  are in bijection and it is clear that if  $X$ is a Lévy process with characteristic  exponent as in the previous display, we have 
\begin{equation*}
    X \overset{\scr{L}}{=} Y^{(n,1)} + \dots + Y^{(n,n)}, 
\end{equation*}
where for each $i \in \{1, \dots , n \}$,  $Y^{(n,i)}$ is an independent copy of a  Lévy process with  characteristic triplet $(b/n,q/n , \nu/n )$. Said otherwise,  Lévy processes are infinitely divisible processes. Moreover, from the formula for the characteristic function of Proposition \ref{lemma:fddsExtendidas}, it is clear that NRLPs  are in turn infinitely divisible.  \par 
Now, recall that a Gaussian process $X = (X_t)_{t \in T}$  is a $T$-indexed process satisfying that, for any $I = \{ {t_1}, \dots , {t_k} \} \subset  T$, the vector
$
   X_I =  (X_{t_1}, \dots , X_{t_k})
$
is Gaussian. In the sequel we  also assume that the Gaussian processes we  work with are centred.  Gaussian  processes are characterised by their covariance function, in the sense that the law of $X$ is completely determined by the semi-definite positive function $\Gamma : T \times T \rightarrow  \mathbb{R}$ defined by  
\begin{equation}
    {\Gamma}(t,s) := \esp{X_tX_s}{}, \quad \quad \text{ for } t, s \in T.
\end{equation}
The following characterisation of infinitely divisible stochastic processes shows that they are the natural generalisation of Gaussian processes: 
\begin{proposition} \emph{[Proposition 3.1.3]\cite{Samorodnitsky_StochProcesses_and_LongRangeDependance(+IDprocesses)}}
An $\mathbb{R}$-valued stochastic process $X = (X_t)_{t \in T}$ is infinitely divisible  if and only if for any finite collection of indices $I=\{t_1, \dots t_k \} \subset T $,  the random vector
$
X_I =  (X_{t_1}, \dots X_{t_k})
$
is infinitely divisible. 
\end{proposition}
Hence, if $X$ is an infinitely divisible process, by Lévy-Kintchine representation and the previous proposition,  for every $I=\{t_1, \dots, t_k \}$  there exists: an $\mathbb{R}^k$-valued  measure  $\nu_{I} (\dd x)$ verifying  
\begin{equation*}
    \int_{\mathbb{R}^k} 1 \wedge |x|^2 \nu_{I} (\dd x) < \infty, \hspace{15mm} \text{ and } \hspace{15mm}  \nu_{I}\left( \{ 0_I \} \right)=0,
\end{equation*}
a semi-definite positive $I \times I$  matrix  $\Gamma_{I}$ and an $\mathbb{R}^k$ vector, that we denote as ${b}({I})$, satisfying for every $\theta \in \mathbb{R}^I$ the identity: 
\begin{align} \label{formula:fddsID}
    \esp{ \exp \left\{ \sum_{t \in I} \theta_t X_{t} \right\} }{} = 
    \exp \left\{ i \langle {b}({I}) , {\theta} \rangle -\frac{1}{2} \langle {\theta} \, \Gamma_{I} , {\theta} \rangle  + \int_{\mathbb{R}^I} \left( e^{i \langle {\theta} , x \rangle} - 1 - i \langle {\theta} , \llbracket x \rrbracket \rangle  \right) \nu_{I}(\dd x) \right\}.
\end{align}
 It is possible to show that one can recover the collection of  triplets $((b(I) , \Gamma_I , \nu_I) : I \subset T, \, |I| < \infty )$ from a so called functional triplet $(b,\Gamma, \bar{\nu})$,  consisting in a path $b \in \mathbb{R}^T$, a covariance function $\Gamma : T \times T \rightarrow \mathbb{R}$ and a path-valued measure $\bar{\nu}$ defined in $\mathbb{R}^T$, satisfying for any finite  $I \subset {T}$, that 
\begin{equation*}
    b(I) = \pi_I(b) \hspace{20mm} \Gamma_I =  \restr{\Gamma}{I} \hspace{20mm} \nu_I (\dd x) = \bar{\nu} \circ \pi_I^{-1}(\dd x)  \quad \text{ in } \scr{B}_0^I, 
\end{equation*}
where $\bar{\nu}$ satisfies some regularity and integrability conditions that we now introduce:
\begin{definition} \label{definicion:condicionesPathLevy}
A measure $\bar{\nu}$ on $\mathbb{R}^T$ is called a path Lévy measure if it satisfies the following two conditions:

\begin{enumerate}
    \item[\emph{(i)}] $\int_{\mathbb{R}^T}  |e(t)|^2 \wedge 1 \, \bar{\nu}(\dd e) < \infty \quad \text{ for all }t \in T. $
    \item[\emph{(ii)}]  $\text{ For every } A \in \scr{B}^T, \text{ there exists a countable subset } T_A \subset T \text{ such that } \bar{\nu}(A) = \bar{\nu}(A \setminus  \pi_{T_A}^{-1}( 0_{T_A})  )$. 
\end{enumerate}
Moreover, we consider the following third condition: 
\begin{enumerate}
    \item[\emph{(iii)}] $\text{There exists a countable subset } T_0 \subset T \text{ such that } \bar{\nu} (  \pi_{T_0}^{-1} ( {0}_{T_0} ) ) = 0. $
\end{enumerate}
\end{definition}
Then, (iii) is a stronger statement than (ii) and it has been shown  that a path Lévy measure is $\sigma$-finite if and only if (iii) holds -- see  e.g. \cite{RosinskiID}.  Condition (ii) states roughly speaking  that $\bar{\nu}$ "does not charge the origin". As we already mentioned, in general $0_T$ is not measurable and hence we can not state this condition as in the finite-dimensional case of Lévy measures.  One of the main results of the theory states that infinitely divisible processes are in bijection with functional triplets $(b,\Gamma , \bar{\nu})$, we refer to  \cite{RosinskiID} for the proof:
\begin{theorem} \label{theorem:IDprocessesTriplets}
For every infinitely divisible stochastic process $X=(X_t)_{t \in T}$ there exists a unique generating  triplet $({b}, {\Gamma} , \bar{\nu})$  consisting of a path $b \in \mathbb{R}^T$, a covariance function  $\Gamma$ in $T \times T$ and  a path Lévy measure  $\bar{\nu}$  in $\mathbb{R}^T$  such that for any finite  $I \subset \hat{T}$,
 \begin{equation} \label{formula:LevyKintchingID}
   \esp{\exp \left\{ i  \sum_{t \in I} \theta_t X_t \right\} }{} = \exp \left\{ i \langle b_I  , \theta \rangle -\frac{1}{2} \langle \theta \, \Gamma_I , \theta \rangle  + \int_{\mathbb{R}^T} \left( e^{i \langle \theta , e_I \rangle} - 1 - i \langle \theta , \llbracket e_I \rrbracket \rangle \right) \bar{\nu}(\dd e) \right\}.
\end{equation}
Conversely, for every generating triplet $({b}, {\Gamma} , \bar{\nu})$ there exists an infinitely divisible process satisfying (\ref{formula:LevyKintchingID}). 
\end{theorem}
Maintaining the notation of Theorem \ref{theorem:IDprocessesTriplets}, it follows in particular that the law of any ID process $X$ can be written as a sum of two independent processes 
$
     X \overset{\scr{L}}{=} G + P, 
 $
 where $G$ is  Gaussian with covariance $\Gamma$ and $P$ is a so-called  Poissonian ID process. When the equality $X = G+P$ holds almost surely, we call respectively $G$ and $P$ the Gaussian part and the Poissonian part of $X$. Let us conclude our presentation with the following notion that will be of use:  
\begin{definition} \label{definition:representantLevyMeasure}
A process $V = (V_t)_{t \in T}$ defined in a measure space $(S , \scr{S}, n)$ is called a representant of a path Lévy measure $\bar{ \nu}$ if for any finite $I \subset {T}$, we have 
\begin{equation*}
    n \left( s \in S \, : \, V_I (s) \in B \right) = \bar{\nu}_I(B), \quad \quad \text{ for every } B \in \scr{B}_0^I. 
\end{equation*}
The representation is called exact if $n \circ V^{-1} = \bar{\nu}$.
\end{definition}
This is, if $V$ is only a representant, the measure $\nu \circ V^{-1}$ might not be a Lévy measure since it might "charge the origin".
In the situations we will be interested in the representations will always be exact, and we only enunciate the weaker definition to write the results we need in their full generality. Representants allow to build explicitly Poissonian ID processes in terms of Poisson random measure, for more details we refer to \cite{RosinskiID}, see also our brief discussion before the proof of Proposition \ref{proposition:convOrigenNRLP} below. 

\subsubsection{The characteristic triplet of a  NRLP}

We can now start investigating Lévy processes and their reinforced counterparts as ID processes, and we start with a basic analysis of the former. More precisely, we  identify the path Lévy measure of Lévy processes as well as an exact representant. These results are known \cite[Example 2.23 ]{RosinskiID} and the statements are only included to contrast with the analogous results for NRLPs -- see Lemma \ref{teorema:InfDivNRLP} below.

\begin{lemma} \label{proposition:LevyProcessesID} The following assertions hold: 
\begin{enumerate}
    \item[\emph{(i)}]  Let $\xi$ be a  Lévy process with characteristic triplet $(a, q, \Lambda)$.  The path Lévy measure $\bar{\nu}$ of $\xi$ is given by, 
\begin{equation*}
     \bar{\nu} (\dd e) := ( \dd t \otimes \Lambda) \circ V^{-1} (\dd e),  
\end{equation*}
where we denoted by $V$ the mapping $V : \mathbb{R}^+ \times \mathbb{R} \mapsto \mathbb{R}^{ \mathbb{R}^+}$ defined as $V(s,x) := x \indic{\{ s \leq \cdot \}}$. 
    \item[\emph{(ii)}]  Consider a measure $\Lambda$ on $\mathbb{R}$ with $\Lambda(0) =0$ and let  $V$ be defined as in \emph{(i)}. Then, the condition $\int 1 \wedge |x|^2 \Lambda(\dd x) < \infty$ holds  if and only if $\bar{\nu} := ( \dd t \otimes \Lambda ) \circ V^{-1}$ is a path Lévy measure in $\mathbb{R}^{\mathbb{R}^+}$. Moreover, if the later holds, the path Lévy measure $\bar{\nu}$ is $\sigma$-finite. 
\end{enumerate}
\end{lemma}
In particular, from (i) we get that $V$ is an exact representant of $\bar{\nu}$, on
$(S , \scr{S}, n) := (\mathbb{R}^+ \times \mathbb{R} , \scr{B}(\mathbb{R}^+)\otimes \scr{B}(\mathbb{R}),\dd t\otimes \Lambda )$. 
We now turn our attention to noise reinforced Lévy processes and we start with the following technical lemma: 
\begin{lemma}\label{lema:integrabilidadCharFun}
Let $\hat{\xi}$ be  an NRLP of characteristic triplet $(a,0,\Lambda,p)$ and let $T = [0,1]$. Then, for any $t \in T$, we have 
\begin{equation} \label{equation:cotaDrift}
    \esp{ \int_{\mathbb{R}} \big|    \llbracket Y(t) x \rrbracket - Y(t) \llbracket x \rrbracket \big| \Lambda(dx)  }{} < \infty.
\end{equation}
\end{lemma}
\begin{proof}
First,  recalling that  $ \llbracket x \rrbracket = x \indic{\{ |x| \leq 1 \}}$,  we can write 
\begin{equation} \label{equation:IDdosTerminos}
    \esp{ \big|    \llbracket Y(t) x \rrbracket - Y(t) \llbracket x \rrbracket \big|   }{} =  \esp{ \big| Y(t) \llbracket x \rrbracket \big|  \indic{\{ |x Y(t)| >1 \}}  }{}
    +
     \esp{ \big|  Y(t) x - Y(t) \llbracket x \rrbracket \big|  \indic{\{ |x Y(t)| \leq 1 \}}  }{}.
\end{equation}
Remark that since  $Y$ takes values in $\{ 0,1,2, \dots \}$, the second term in the right-hand side vanishes.  On the other hand,  for any $q \in (\beta(\Lambda) \vee 1 , 1/p )$, we have 
\begin{align} \label{equation:IDunaestimacion}
    \int_{\mathbb{R}} \esp{ \big| Y(t) \llbracket x \rrbracket \big|  \indic{\{ |x| Y(t) >1 \}}  }{} \Lambda(dx) &=   \int_{\{ |x| \leq 1 \}} \esp{  Y(t)   \indic{\{ |x| Y(t) >1 \}}  }{} |x| \Lambda(dx) \nonumber \\
   &\leq \esp{  Y^q(t) }{}^{1/q} \int_{\{ |x| \leq 1 \}}  \proba{ Y(t) >1/|x|}{}^{(q-1)/q}  |x| \Lambda(dx), 
\end{align}
where we recall that  $Y \in L^q(\mathbb{P})$ for any $q < 1/p$. To conclude,   recall the asymptotic behaviour from (10) in \cite{BertoinNRLP},  
\begin{equation*}
    \proba{Y(t) > 1/|x|}{} \sim t \Gamma(1/p +1) |x|^{1/p},  \quad  \text{as } x \downarrow 0. 
\end{equation*}
It now follows that we can take $q$ close enough to $1/p$ such that the integral in \eqref{equation:IDunaestimacion} is finite and we deduce (\ref{equation:cotaDrift}).  
 \end{proof}

Now, we identify the path Lévy measure of NRLPs.

\begin{lemma} \label{teorema:InfDivNRLP} The following assertions hold: 
\begin{enumerate}
    \item[\emph{(i)}]  Let $\hat{\xi}$ be a NRLP with characteristic triplet $(a, q^2, \Lambda,p)$.  The path Lévy measure $\bar{\nu}$ of $\hat{\xi}$ is given by, 
\begin{equation*}
    \bar{\nu}:=(1-p) (\Lambda \otimes \mathbb{Q}) \circ V^{-1},
\end{equation*}
where    $V:D[0,1] \times \mathbb{R} \rightarrow \mathbb{R}^{[0,1]}$ is defined by  $V(x,y) := xy$. 
    \item[\emph{(ii)}]  Let $(a, 0 , \Lambda)$ be the characteristic triplet of a Lévy Process and let $V$ be defined as in \emph{(i)}. Then, if a memory parameter $p \in (0,1)$ is admissible for the triplet $(a, 0 , \Lambda)$, the measure  $\bar{\nu}:= (\Lambda \otimes \mathbb{Q}) \circ V^{-1}$ is a $\sigma$-finite path Lévy measure in $\mathbb{R}^{[0,1]}$.
    On the other hand, if $1/p < \beta(\Lambda)$, then the integrability condition \ref{definicion:condicionesPathLevy} - \emph{ (i)} fails.  
\end{enumerate}
\end{lemma}
In particular, from (i) we get that $V$ is an exact representant of $\bar{\nu}$, in
$(S, \scr{S}, n) = (D[0,1] \times \mathbb{R} , \scr{B}(D[0,1])\otimes \scr{B}(\mathbb{R}), \mathbb{Q} \otimes \Lambda (1-p) )$. On other hand, (ii) gives a natural interpretation in  the terminology of ID processes for the admissibility of $p$ for $\Lambda$. 

\begin{proof}
{To identify the Lévy measure, let us  write the characteristic function of the finite dimensional distributions of $\hat{\xi}$ in the form \eqref{formula:LevyKintchingID} and  to simplify notation, we suppose that $a,q = 0$.  In this direction, consider a finite   $I \subset  T$,  $\theta = (\theta_{t_1}, \dots , \theta_{t_k}) \in \mathbb{R}^I$,   and denote by $y = (y(t))_{t \in [0,1]}$ an arbitrary counting function. Recall the formula for the finite dimensional distributions of $\hat{\xi}$ from Proposition \ref{lemma:fddsExtendidas}, for   $t = 1$. It now follows  by Lemma \ref{lema:integrabilidadCharFun}  and the triangle inequality that we have:  
\begin{equation} \label{equation:section8Integrabilidad}
     \int_{\mathbb{R}}  \Lambda(\dd x)  \esp{ \left| e^{i \langle \theta , Y_I \rangle x} -1- i \langle \theta, \llbracket x Y_I \rrbracket \rangle \right| }{} < \infty.
\end{equation} Now, we can write 
\begin{align*} 
    \esp{ \exp \left\{ i \sum_{t \in I} \theta_{t}\hat{\xi}_{t} \right\} }{} 
    &= \exp \left\{ (1-p) \int_{\mathbb{R}} \Lambda(\dd x)  
    \esp{e^{ \langle \theta , Y_{I} \rangle x } -1- i \langle \theta , Y_{I} \rangle  \llbracket x  \rrbracket }{}
       \right\} \\
    &= \exp \bigg\{ (1-p) \int_{\mathbb{R} \times D[0,1]} \left(  e^{ \langle \theta , (xy)_{I} \rangle  } -1- i \langle \theta ,  \llbracket x y_{I}   \rrbracket \rangle \right) \Lambda \otimes \mathbb{Q} (\dd x , \dd y)   \\
    &\hspace{25mm}+
     i \int_{\mathbb{R} \times D[0,1]} \langle \theta   , \llbracket x y_I  \rrbracket \rangle -  \langle \theta , y_{I} \rangle  \llbracket x  \rrbracket  \Lambda \otimes \mathbb{Q} (\dd x , \dd y) (1-p) \bigg\}, 
\end{align*}
where all the terms in the previous expression   are well defined by Lemma \ref{lema:integrabilidadCharFun} and \eqref{equation:section8Integrabilidad}. Since  $(xy)_I = \pi_I(V(x,y) )$ and  $\bar{\nu} = (1-p) (\Lambda \otimes Q) \circ V^{-1}$, we obtain for a clear choice of  $b_I$ that  
\begin{equation*}
 \esp{ \exp \left\{ \sum_{t \in I} \theta_{t}\hat{\xi}_{t} \right\} }{} =  \exp \left\{ \int_{\mathbb{R}^T} \left(  e^{ \langle {\theta} ,  e_I \rangle } -1- i \langle {{\theta} } , \llbracket e_I  \rrbracket \rangle \right) \bar{\nu}(\dd e)  + \langle \theta , b_I \rangle \right\}.
\end{equation*}
}
Next, notice  that condition (iii) of Definition \ref{definicion:condicionesPathLevy}  is satisfied by $ (\Lambda \otimes \mathbb{Q}) \circ V^{-1}$. Indeed,  if we let $T_0 := \{ 1 \}$ and $\bm{0}_{T_0} := \{ e \in \mathbb{R}^T \, : \, e(1) =0 \}$,  recalling that $Y(1) \geq 1$ a.s., we deduce that  
\begin{equation*}
    \Lambda \otimes \mathbb{Q} \left( (x, y) : xy \in \{ \bm{0}_{T_0} \} \right) = \Lambda(\{ 0 \}) = 0.
\end{equation*}
To conclude, let us  show that $\bar{\nu}$ satisfies the integrability condition (i)  of Definition \ref{definicion:condicionesPathLevy} if   $p$ is an admissible memory parameter for $\Lambda$, viz. if $\beta(\Lambda) < 1/p$, while when  $\beta(\Lambda) > 1/p$, the condition fails.  By definition of $\bar{\nu}$, we have 
\begin{equation}\label{equacion:integrabilidadNecesaria}
    \int_{\mathbb{R}^T}\left( |e(t)|^2 \wedge 1 \right) \bar{\nu}(\dd e) = \int_\mathbb{R} \esp{|x Y_t|^2\wedge 1}{} \Lambda(\dd x),   
\end{equation}
and write:   
\begin{equation}
    \esp{|x Y_t|^2\wedge 1}{} = |x|^2 \esp{|Y_t|^2 \indic{\{ Y_t \leq 1/|x| \}}}{} + \proba{Y_t > 1/|x|}{}.
\end{equation}
Now, recalling  from (10) of \cite{BertoinNRLP} the asymptotic behaviour, 
\begin{equation*}
    \proba{Y_t > 1/|x|}{} \sim t \Gamma(p^{-1} +1) |x|^{1/p},  \quad \text{ as } |x| \downarrow 0, 
\end{equation*}
it follows that if  $\beta(\Lambda) < 1/p$,  the term  $\proba{Y_t > 1/|x|}{}$ is integrable with respect to $\Lambda$ and infinite if $\beta(\Lambda) > 1/p$. Let us  now show that the same holds for 
\begin{equation} \label{equation:unaeqdsa}
    \int_{[0,1]} |x|^2 \esp{|Y_t|^2 \indic{\{ Y_t \leq 1/|x| \}}}{} \Lambda(\dd x).
\end{equation}
Recalling Lemma \ref{lemma:yuleSimonProcess}, we get:  
\begin{equation*}
    \esp{|Y_t|^2 \indic{\{ Y_t \leq 1/|x| \}}}{}  = \sum_{n=1}^{\lfloor 1/|x| \rfloor } n^2 \proba{Y_t = n}{} = t p^{-1} \sum_{n=1}^{\lfloor 1/|x| \rfloor} n^2 \text{B} \left( n, p^{-1} +1 \right),    
\end{equation*}  
where we denoted by B the Beta function. Now, from  the asymptotic behaviour
\begin{equation*}
    \text{B}(n, p^{-1} +1) \sim n^{-(1+p)/p } \Gamma ( p^{-1} +1 ),  \quad  \text{ as } n \uparrow \infty, 
\end{equation*}
it follows that (\ref{equation:unaeqdsa}) is finite if  $\beta(\Lambda) < 1/p$ and infinite if $\beta(\Lambda) > 1/p$.
\end{proof}

Let us state the two last result that we  need for the proof of Proposition \ref{proposition:convOrigenNRLP}. {First, the Poissonian part of ID processes consists,  roughly speaking, in Poissonian sums of i.i.d. trajectories -- for instance, remark that  for  NRLPs those trajectories are the weighted Yule-Simon processes --  for more examples see e.g.  \cite[Section 3]{RosinskiID}.  More precisely, let  $X = (X_t)_{t \in T}$ be  an  infinitely divisible process with characteristic triplet $(b, \Sigma, \bar{\nu})$ and  suppose that $V = (V_t)_{t \in T}$ is a representant of $\bar{\nu}$ defined on a $\sigma$-finite measure space $(S, \scr{S}, n)$. To simplify notation, set  $\chi(u) := \indic{\{ |u| \leq 1 \}}$ and consider $\mathcal{M}$  a  Poisson random measure in $(S,\scr{S})$  with intensity $n$. Then, the following process has the same distribution as $X$, 
\begin{equation} \label{equation:representacionIDProcesses}
    b_t + G_t + \int_S V_t(s) \big( \mathcal{M}( \dd s ) - \chi(V_t(s)) n( \dd s ) \big),  \quad t \in T, 
\end{equation}
 where  $G = (G_t)_{t \in T}$ is an independent Gaussian process with covariance $\Sigma$. The integration in the previous display should read as a compensated integral, and  for a detailed statement we refer to \cite[ Proposition 3.1]{RosinskiID}. 
} For example, notice that if $X$ is a Lévy process, $\mathcal{M}$ is Poisson with intensity $\dd t \otimes  \Lambda(\dd x )$ and replacing $V$ by $x\indic{\{ s \leq \cdot \}}$   yields a Lévy-Itô representation.  Finally, we give one of the statements that we  use of the  Isomorphism Theorem of infinitely divisible processes needed for our proof. 

\begin{theorem} \emph{[Isomorphism Theorem]\cite[4.4]{RosinskiID}}
\label{theorem:Isomorphism}
Let $X = (X_t)_{t \in T}$ be an infinitely divisible process given by \eqref{equation:representacionIDProcesses}. Choose an arbitrary measurable function $\emph{q}: S \mapsto \mathbb{R}^+$ such that $\int_S \emph{q}(s) n( \dd s ) = 1$ and set $N(\emph{q}) := \int_S \emph{q}(s) \mathcal{M} ( \dd s )$. Then, for any measurable functional $F: \mathbb{R}^T \mapsto \mathbb{R}$, we have 
\begin{align*}
    \esp{F((X_t)_{t \in T}) \indic{\{ N(\emph{q}) >0 \}} }{} = \int_S \esp{F \left( (X_t + V_t(s))_{t \in T} \right) \left( N(\emph{q}) + \emph{q}(s) \right)^{-1} }{} \emph{q}(s) n( \dd s ). 
\end{align*}
\end{theorem} This allows for instance to study the law of $X$ under different conditionings, for appropriate elections of $\text{q}$. This will be used in our reasoning below.
Now, let us conclude  the proof of Proposition \ref{proposition:convOrigenNRLP}. 
\begin{proof}[Proof of Proposition \ref{proposition:convOrigenNRLP}.]
To simplify notation, we will perform a slight  abuse of notation by writing   ${\Lambda}$ instead of $(1-p)\Lambda$. We start by fixing $\delta \in (0,1)$ small enough such that $m= \Lambda(|x| > \delta )>0 $. Now, let $h>0$  and as usual, write  $y = (y(t))_{t \in [0,1]}$ for a generic counting trajectory in $D[0,1]$.  Recall the result of Lemma \ref{teorema:InfDivNRLP} and consider  a Poisson random measure $\mathcal{M}= \sum \delta_{(x_i, Y_i)}$ with intensity $\Lambda \otimes \mathbb{Q}$. Next,  we set 
\begin{equation*}
    \text{q}(y,x) := \frac{1}{mh} \indic{\{ |x| \geq \delta \}} \indic{\{ y(h) \geq 1 \}}, 
\end{equation*}
and take  
\begin{equation*}
    N(\text{q}) = \frac{1}{mh} \int_{D \times \{ |x| \geq \delta \}} \indic{ \{ y(h) \geq 1  \}} {\mathcal{M}} ( \dd x, \dd y ) =: \frac{1}{mh} S_h.
\end{equation*}
Then,   from the definition of $S_h$ we have 
\begin{equation*}
    S_h = \# \{ (x_i , Y_i) \, : \, |x_i| \geq \delta \text{ and } Y_i(h) \geq 1 \}  \leq  \# \{ |\Delta \hat{\xi}_s| \geq \delta \text{ for } s \leq h \}, 
\end{equation*}
where the inequality stands  since a jump $x_i$ is repeated at each jump time of its respective $Y_i$, and consequently might be repeated multiple times in $[0,h]$. However, we do have  $\# \{ |\Delta \hat{\xi}_s| \geq \delta \text{ for } s \leq h \} = 0$ when  $S_h = 0$. Finally, we consider  the functional $F(e ) := f(e(h))$ for $e \in D[0,1]$. An application of Theorem \ref{theorem:Isomorphism} yields: 
\begin{align*}
    \esp{f(\hat{\xi}_h) \indic{\{ S_h >0 \}} }{} &= \int_{D \times \{ |x| \geq \delta \}} \esp{f \big(\hat{\xi}_h + xy(h) \big) \frac{1}{S_h + 1} }{} \indic{\{ y(h) \geq 1 \}} \mathbb{Q}(dy) \Lambda(\dd x)\\
    &= \int_{D \times \{ |x| \geq \delta \}} G_h(xy(h)) \indic{\{ y(h) \geq 1 \}} \mathbb{Q}(\dd y) \Lambda(\dd x), 
\end{align*}
for $G_h(z) = \esp{f (\hat{\xi}_h + z) \frac{1}{S_h + 1} }{}$ and notice that $\lim_{h \downarrow 0} G_h(z) = f(z)$ by right-continuity --  remark that the previous display can be interpreted as the law of $\hat{\xi}_h$ conditioned at having at least one jump before time $h$ of size  greater than $\delta$. If we let $\eta$ be a random variable distributed Yule-Simon with parameter $1/p$ under $\mathbb{P}$, this entails that we can write: 
\begin{align*}
    \frac{1}{h} \esp{ f(\hat{\xi} _h)}{} 
    &= \frac{1}{h} \esp{ f(\hat{\xi} _h) \indic{\{S_h =0 \}} }{} 
    + \frac{1}{h} \esp{ f(\hat{\xi} _h) \indic{\{S_h > 0 \}} }{} \\
    &=  \frac{1}{h} \esp{ f(\hat{\xi} _h) \indic{\{S_h =0 \}} }{} +
    \int_{\{ |x| \geq \delta \}} \esp{ G_h(x \eta ) }{} \Lambda(\dd x), 
\end{align*}
where in the last equality we used that  the law of $y(h)$  under
\begin{equation*}
     \frac{ \indic{\{y(h) \geq 1\}}}{ \mathbb{Q}(y(h) \geq 1)} \mathbb{Q}(\dd y ),  
\end{equation*}
is the  Yule-Simon distribution with parameter $1/p$ by Lemma \ref{lemma:yuleSimonProcess} and that  $\mathbb{Q}(y(h) \geq 1) = h$.  Consequently, we deduce that 
\begin{align*}
    &\Big| h^{-1}  \mathbb{E}\big[  {f(\hat{\xi}_h )}{} \big] - \int_\mathbb{R} \esp{f(x\eta)}{} \Lambda(\dd x) \Big| \\
    &\leq h^{-1} \esp{ |f(\hat{\xi} _h)| \indic{\{S_h =0 \}} }{} + \int_{\{ |x| \leq \delta \}} \mathbb{E} \big[ |f(x \eta)| \big] \Lambda(\dd x) 
    + \Big| \int_{\{ |x| \geq \delta \}} \esp{G_h(x\eta)}{}  -  \esp{f(x \eta)}{} \Lambda(\dd x) \Big|\\
    &=: K_1(h , \delta) + K_2(\delta) + K_3(h,\delta). 
\end{align*}
 Now, we  study the  limit as $h \downarrow 0$ of these three terms separately and we start with $K_1(h,\delta)$. Recall the notation introduced in \ref{subsubsection:jointLaw_tightnessGeneral} for the compensated integrals as well as  $\Sigma_{\delta,\infty} := \indic{(-\delta, \delta)^c} x * \hat{\mathcal{N}}$ for the process obtained by adding jumps of size greater that $\delta > 0$. Recall that  on $\{ S_h =0 \}$, the process $\hat{\xi}$ doesn't have jumps of seize greater than $\delta$ before time $h$. It now follows that,  restricted to  $\{ S_h = 0 \}$, the following equality holds: 
\begin{align} \label{equation:NRLPsinsaltosdelta}
    \hat{\xi}_h = \hat{\xi}_h - \sum_{i} x_i Y_i(h) \indic{\{ |x_i| \geq \delta \}} 
    = a \cdot h +  \hat{\xi}_{0,1}^{(3)} (h) + \Sigma_{1,\infty} (h) - \Sigma_{\delta , \infty}(h)
    = \hat{\xi}_{0,\delta}^{(3)}(h) -  c_\delta \cdot h,     
\end{align}
for  $c_\delta := -a +  (1-p)^{-1} \int_{\{ \delta \leq |x| \leq 1 \}} x \Lambda(\dd x)$ and  denote  the right hand side of (\ref{equation:NRLPsinsaltosdelta}) by $\hat{\xi}^\delta_h$. Now let us  first consider the case $\beta(\Lambda) < 2$. Since $f$ is bounded and $O(|x|^2)$ at the origin,   for any $q \in (\beta(\Lambda) \vee 1 , 1/p \wedge 2)$ satisfying $q < r$ we can bound $|f(x)| \leq C |x|^q$ for all $x \in \mathbb{R}$, for some constant $C$ large enough. Then, for a constant  $C'$ that only depends on $q$ we have 
\begin{align*}
    K_1(h , \delta) &=  h^{-1} \esp{ |f(\hat{\xi} _h)| \indic{\{S_h =0 \}} }{} \\
    &\leq C h^{-1} \esp{ | \hat{\xi}_h^\delta |^q }{} \leq C'h^{-1} \esp{| \hat{\xi}_{0,\delta}^{(3)}(h) |^q}{} + C'h^{q-1}|c^\delta|^q.
\end{align*}
Now, arguing as in  (\ref{equation:cotaCompensacion1}),  (\ref{equation:cotaCompensacion2}),  recall that for $q \in (\beta(\Lambda) \vee 1 , 1/p \wedge 2 )$ we have the following bound for the compensated sum of Yule-Simon processes:
\begin{equation} \label{inequality:IDboundOrigin}
    \esp{ |\hat{\xi}_{0,\delta}^{(3)}(h) |^q }{} \leq \esp{Y(h)^q}{} \int_{\{ |x| \leq \delta \}} |x|^q \Lambda(\dd x) = h \cdot \esp{\eta^q}{} \int_{\{ |x| \leq \delta \}} |x|^q \Lambda(\dd x) < \infty. 
\end{equation}
Since $q-1>0$, we have   $\limsup_{h \downarrow 0} K_1(h,\delta) \leq \esp{\eta^q}{} \int_{\{ |x| \leq \delta \}} |x|^q \Lambda(\dd x)$ which can be made arbitrarily small for an appropriate choice of $\delta$. Remark that  the same reasoning applies for $K_2(\delta)$, by making use once again of the bound $|f(x)| \leq C |x|^q$. Finally, since for any choice of $\delta$, $K_3(h , \delta) \downarrow 0$ as $h \downarrow 0$, we obtain the desired result. \par 
If $\beta(\Lambda) = 2$, we set $q=2$ and once again recall  from page 9 of Bertoin \cite{BertoinNRLP} that the inequality (\ref{inequality:IDboundOrigin}) still holds. In this case, since $p \beta(\Lambda) < 1$, $p$ must be smaller than $1/2$ and consequently $\esp{\eta^2}{} < \infty$, while of course  $\int_{\{|x|\leq \delta\}}|x|^2\Lambda(\dd x) < \infty$ by definition of a Lévy measure. We can then proceed as before. 
\end{proof}
{ 
\addtocontents{toc}{\vspace{\contentsSpacingBefore}}
\section{Appendix}
This short section is devoted to proving a technical identity needed for the proof of  Lemma \ref{lemma:nrpp_jumpProcess}. The proof was omitted from the main discussion for readability purposes. 
\medskip \\Fix a Lévy measure $\Lambda$ in $\mathbb{R}$,  $p \in (0,1)$  and  denote the law of the standard Yule process $Z = (Z(t))_{t \in \mathbb{R}^+}$ started at $Z_0=1$  by  $\mbm{Z}$. We write $D[0,\infty)$ the space of $\mathbb{R}^+$ indexed,  $\mathbb{R}$-valued rcll functions. Since  $\mbm{Z}$ is supported on the subset of $D[0,\infty)$ of counting functions, $z=(z_t)_{t \in \mathbb{R}^+}$ in the sequel stands for a generic counting function. Moreover,   if  $F:\mathbb{R}^+ \times D[0,\infty) \mapsto \mathbb{R}^+$ is a  measurable function, we write $\mbm{Z}^\bullet$ for the measure in $\mathbb{R}^+ \times D[0,\infty)$ defined as  $\mbm{Z}^\bullet (F) := \int_\mathbb{R^+} \dd u \, \mathbb{E}[  F(u,Z)]$. 
Roughly speaking, the objective is to describe the law of the following "process": 
\begin{equation} \label{formula:extensionYS}
    (u,z) \mapsto \left(  \indic{\{u \leq t\}} z_{ p(\ln(t) - \ln(u)) } \, : \, t \in \mathbb{R}^+ \right)  \quad  \in D[0,\infty)
\end{equation}
defined on the measure space $(\mathbb{R} \times D[0,\infty), \mbm{Z}^\bullet)$, under different  restrictions of the measure $\mbm{Z}^\bullet$. In this direction, for $T > 0$ we write 
\begin{equation*}
     \mbm{Z}^\bullet (\cdot \, | u \leq T) := \frac{\indic{\{ u \leq T \} }}{T} \dd u \, \mbm{Z}(\dd z ), 
\end{equation*}
which is now a probability measure on $\mathbb{R}^+ \times D[0,\infty)$. The main properties of interest are stated in the following lemma, and shares obvious similarities with Lemma  \ref{lemma:yuleSimonProcess}. 
\begin{lemma}\label{lemma:restriccionesRepresentante}
The following properties hold: 
\begin{enumerate}
    \item[\emph{(i)}] For each fixed $t>0$, the random variable 
\begin{equation*}
 (u,z) \mapsto  \indic{\{u \leq t \}}z_{ p \ln(t/u) } \quad \text{ under } \quad \mbm{Z}^\bullet (\cdot \, | u \leq t), 
\end{equation*}
has the same distribution as the Yule Simon random variable $\eta$ with parameter $1/p$.
    \item[\emph{(ii)}] For every $T>0$, the process
\begin{equation*}
  (u,z) \mapsto  \left( \indic{\{ u \leq T t \}} z_{p \ln(Tt/u)} \,  : \, t \in [0,1] \right) \quad \text{ under} \quad   \mbm{Z}^\bullet (\cdot \, | u \leq T), 
\end{equation*}
has the same law as the Yule-Simon process $(Y(t))_{t \in [0,1]}$ with parameter $1/p$. 
\end{enumerate}
\end{lemma}
 Notice that the conditioning $\{ u \leq t \}$ is playing the exact same role as the conditioning on $\{ Y(t) \geq 1 \}$ in Lemma \ref{lemma:yuleSimonProcess}.  Heuristically, \eqref{formula:extensionYS} is then a Yule-Simon process started at a time chosen according to  $\dd u$ in $\mathbb{R}^+$.
\begin{proof}
(i) Since for each fixed $t$, $\dd u\otimes \mbm{Z} (u \leq t) =t$, for every bounded measurable function $f:\mathbb{R} \rightarrow \mathbb{R}$, we have 
\begin{align}\label{equation:eqYS}
   \mbm{Z}^\bullet \left(  f\left( \indic{\{ u \leq t \}} z_{ p \ln(t/u)} \right) | u \leq t \right) 
    = t^{-1} \int_0^t \dd u \, \mathbb{E} \Big[ f \Big(Z \big(p( \ln(t) - \ln(u) )\big) \Big) \Big],  
\end{align}
where we denoted by $Z$ a standard Yule process. Since $Z_r$ is distributed geometric with parameter $e^{-r}$, it follows from the change of variable $y = (u/t)^p$ and  (\ref{formula:densityYuleSimon}) that (\ref{equation:eqYS}) equals 
\begin{align*}
     p^{-1} \sum_{k \geq 1}f(k) \text{B}(k,1+1/p),  
\end{align*}
and the claim follows from (\ref{formula:densityYuleSimon}). \par 
(ii) In order to show the second claim, we fix an arbitrary collection of bounded measurable functions $(f_i)_{i \leq k}$ with $f_i: \mathbb{R} \mapsto \mathbb{R}$, and an increasing sequence of times $0\leq t_1 <  \dots < t_k \leq 1$, and notice that 
\begin{align*}
    \mbm{Z}^\bullet \left( \prod_{i=1}^k f_i\left( \indic{\{ u \leq T t_i \}}z_{p \ln(Tt_i/u)} \right)   | u \leq T \right) 
    &= \int_0^1  \dd u \,  \esp{\prod_{i=1}^k f_i\left( \indic{\{ u \leq t_i \}} Z \left( p (\ln(t_i) - \ln(u)) \right) \right)}{}. 
\end{align*}
The claim now follows from the description (\ref{formula:RepresentacionYuleSimon}) by independence between $U$ and $Z$.
\end{proof}
}

 \textit{Acknowledgement: I warmly thank Jean Bertoin for the discussions and attention provided  through the making of this work, as well as for introducing me to noise reinforced Lévy processes. }

\end{document}